\documentclass[11pt, twoside]{article}
\usepackage[english]{babel}
\usepackage{amssymb}
\usepackage{mathrsfs}
\usepackage{amsmath}
\usepackage{amsthm}
\usepackage{amsfonts}
\usepackage{latexsym}
\usepackage{indentfirst}
\usepackage{color}
\usepackage{txfonts}
\usepackage{enumerate}

\usepackage[colorlinks=true,
linkcolor=blue,
citecolor=red,
urlcolor=magenta, 
]
{hyperref}

\usepackage{txfonts}

\usepackage{anysize}

\allowdisplaybreaks

\newtheorem{theorem}{Theorem}[section]
\newtheorem{lemma}[theorem]{Lemma}

\newtheorem{proposition}[theorem]{Proposition}

\theoremstyle{definition}
\newtheorem{remark}[theorem]{Remark}

\newtheorem{definition}[theorem]{Definition}
\newcounter{assum}

\newtheorem{assumption}[assum]{Assumption}
\renewcommand{\appendix}{\par
\setcounter{section}{0}%
\setcounter{subsection}{0}%
\setcounter{subsubsection}{0}%
\gdef\thesection{\@Alph\c@section}%
\gdef\thesubsection{\@Alph\c@section.\@arabic\c@subsection}%
\gdef\theHsection{\@Alph\c@section.}%
\gdef\theHsubsection{\@Alph\c@section.\@arabic\c@subsection}%
\csname appendixmore\endcsname
}

\numberwithin{equation}{section}

\begin{document}
\title{\bf\Large
Local Hardy Spaces Associated with  Operators
and Ball Quasi-Banach Function Spaces on Doubling Metric Measure Spaces and Their Applications
\footnotetext{\hspace{-0.35cm} 2020 {\it
Mathematics Subject Classification}.
Primary 42B30; Secondary 42B25, 42B35, 46E30.
\endgraf {\it Key words and phrases.}
ball quasi-Banach function space, local Hardy space, atom, molecule,
maximal function, Gaussian upper bound
\endgraf This project is  supported by the
Natural Science Foundation of Fujian Province (Grant No. 2026J008197).}}
\author{Xiaosheng Lin\footnote{Corresponding
author, E-mail: \texttt{xslin@jmu.edu.cn}/{\color{red} \today}/Final version.}
\ and Xiaotian Zhu}
\date{}
\maketitle

\vspace{-0.7cm}

\begin{center}
\begin{minipage}{13cm}
{\small {\bf Abstract.}\quad
Let $(\mathcal{X},d,\mu)$ be a doubling metric measure space, $X$ a ball quasi-Banach
function space on $\mathcal{X}$, and $L$ a non-negative self-adjoint operator on
$L^2(\mathcal{X})$ whose heat kernel satisfies a Gaussian upper bound. In this article,
we study the local Hardy space $h_{X,L}(\mathcal{X})$ associated with both $X$ and $L$.
We first establish the atomic and molecular characterizations of
$h_{X,L}(\mathcal{X})$. As applications of these characterizations, we obtain the
relations between $h_{X,L}(\mathcal{X})$ and the global Hardy spaces
$H_{X,L}(\mathcal{X})$ and $H_{X,L+mI}(\mathcal{X})$. We also establish the radial and
non-tangential maximal function characterizations of $h_{X,L}(\mathcal{X})$. Using
these maximal function characterizations, under the additional assumptions that the
heat kernel satisfies the conservation property and a H\"older regularity estimate,
we further show that $h_{X,L}(\mathcal{X})$ coincides with the local atomic Hardy space
$h_{X,\mathrm{at}}^p(\mathcal{X})$ with equivalent quasi-norms in the corresponding
range of indices. Finally, we apply the above results to Lebesgue spaces, Orlicz spaces,
weighted Lebesgue spaces, and variable Lebesgue spaces.}
\end{minipage}
\end{center}

\section{Introduction}
Stein and Weiss \cite{sw60} introduced the real-variable theory of the classical Hardy
spaces $H^p(\mathbb{R}^n)$, $p\in(0,1]$, and Fefferman and Stein \cite{fs72} developed
this theory further. The spaces $H^p(\mathbb{R}^n)$ provide useful substitutes for
$L^p(\mathbb{R}^n)$ when $p\leq 1$, but their global cancellation is not suitable for
some local questions. Goldberg \cite{g79} pointed out, in particular, that
$H^p(\mathbb{R}^n)$ does not contain all Schwartz functions and that
pseudo-differential operators need not be bounded on $H^p(\mathbb{R}^n)$. To deal with
such problems, he introduced the local Hardy space $h^p(\mathbb{R}^n)$ by restricting
the maximal function to small scales. In the corresponding atomic decomposition,
cancellation is imposed on atoms supported on small balls, while it is not required
for atoms supported on large balls. This local theory is better suited to many
problems involving partial differential equations and pseudo-differential operators;
see \cite{g79}.

Another development of Hardy space theory comes from operators which are not classical
Calder\'on--Zygmund operators. For example, let
$L:=-\operatorname{div}(A\nabla)$ be a second-order divergence form elliptic operator
on $\mathbb{R}^n$ with complex bounded measurable coefficients. The Riesz transform
$\nabla L^{-1/2}$ may fail to be bounded from $H^1(\mathbb{R}^n)$ to
$L^1(\mathbb{R}^n)$; see, for example, \cite{hmm11}. Such phenomena motivated the
development of Hardy spaces associated with operators. Auscher et al. \cite{adm05}
introduced $H_L^1(\mathbb{R}^n)$ for operators satisfying suitable kernel estimates
and established its molecular characterization. Duong and Yan
\cite{dy05jams,dy05cpam} studied the corresponding BMO spaces and duality, while
Hofmann et al. \cite{hlmmy11} developed the real-variable theory under the
Davies--Gaffney estimate. We also refer to \cite{y08} and the references therein for further developments of
global Hardy spaces associated with operators.
Local Hardy spaces associated with operators have also been studied in several settings.
Gong et al. \cite{gly13} defined $h_L^1$ on a doubling metric measure space by an area
function and proved its atomic characterization; under a Moser-type local boundedness
condition, they also obtained a Littlewood--Paley characterization. Cao et al.
\cite{cmy17} studied local Hardy spaces associated with inhomogeneous higher order
elliptic operators and obtained molecular, square function, and maximal function
characterizations, together with interpolation and duality results and a relation with
global Hardy spaces associated with a shifted operator. Bui et al. \cite{bdl18}
established local radial and non-tangential maximal function characterizations on
spaces of homogeneous type and also considered local Hardy spaces associated with a
critical function. In the weighted setting, Gong et al. \cite{gsx} obtained an atomic
characterization of $h_{L,w}^1(\mathbb{R}^n)$ and identified its dual weighted local BMO
space. Almeida et al. \cite{abdr20} considered variable-exponent local Hardy spaces
associated with non-negative self-adjoint operators, established a molecular
characterization, and used it to obtain relations with the corresponding global Hardy
spaces.

Motivated by different applications, Hardy spaces built on various function spaces have also attracted much attention.
For example, Calder\'on et al. \cite{c77,ccfjr78} used weighted Hardy spaces in the study of the Cauchy integral on Lipschitz curves,
while Kenig \cite{k80} investigated weighted Hardy spaces on Lipschitz domains.
We refer to \cite{h13,h15,h17scm} for Hardy--Morrey spaces, to
\cite{ns12,s13,s14,h17tmj,inns23,yzz18,yz16} for variable Hardy spaces
(associated with operators), to
\cite{ns14,ans21,ins22,bckyy13,bl11,jy10,jy11,yy14} for
(Musielak--)Orlicz--Hardy spaces (associated with operators), and to
\cite{c77,ccfjr78,k80,b14,bckyy131,ls13} for weighted Hardy spaces
(associated with operators).
To treat these Hardy-type spaces in a unified framework, it is useful to work with a
general class of underlying function spaces. Based on the notion of quasi-Banach
function spaces in \cite{bs88}, Sawano et al. \cite{shyy17} introduced ball
quasi-Banach function spaces and developed the Hardy space $H_X(\mathbb{R}^n)$
associated with such a space $X$. This class contains, among others, Lebesgue spaces,
weighted Lebesgue spaces, Orlicz spaces, variable Lebesgue spaces, and Morrey spaces;
some of these examples need not be quasi-Banach function spaces in the usual sense.

The aim of this article is to develop further the local Hardy space theory associated
with both an operator and a ball quasi-Banach function space on a doubling metric
measure space. Let $(\mathcal{X},d,\mu)$ be a doubling metric measure space in the
sense of Coifman and Weiss \cite{cw77}, let $X$ be a ball quasi-Banach function space
on $\mathcal{X}$, and let $L$ be a non-negative self-adjoint operator on
$L^2(\mathcal{X})$ whose heat kernel satisfies a Gaussian upper bound. We define
$h_{X,L}$ by using the local Lusin area function associated with $L$ together with
$S_I(e^{-L}f)$. Our first result gives the atomic and molecular characterizations of
$h_{X,L}$ (see Theorem \ref{local-mo-at}). Using these characterizations, we then
establish the relations between $h_{X,L}$ and the global Hardy spaces. More precisely,
$h_{X,L}=H_{X,L}$ with equivalent quasi-norms when the spectrum of $L$ is bounded away
from zero, and $h_{X,L}=H_{X,L+mI}$ with equivalent quasi-norms for every
$m\in\mathbb{N}$; we also obtain a further comparison with $X$ under an additional
ball Banach function space assumption (see Theorem \ref{hH}).

We next establish radial and non-tangential maximal function characterizations for the
local Hardy spaces associated with $L$ (see Theorem \ref{thm-maxi}). These
characterizations also provide the link with local atomic Hardy spaces which are not
adapted to the operator. More precisely, if the heat kernel satisfies the conservation
property
$\int_{\mathcal{X}}K_t(x,y)\,d\mu(y)=1$ and a H\"older regularity estimate in the space
variable, then the local maximal Hardy space associated with $L$ coincides with the
local atomic Hardy space $h_{X,\mathrm{at}}^p$ with equivalent quasi-norms (see
Theorem \ref{key}). Together with Theorem \ref{thm-maxi}, this yields the corresponding
identification of $h_{X,L}$ and $h_{X,\mathrm{at}}^p$ in the common range of their
assumptions. Finally, we apply the above results to Lebesgue spaces, Orlicz spaces,
weighted Lebesgue spaces, and variable Lebesgue spaces.

The remainder of this article is organized as follows. In Section \ref{main}, we
recall the needed facts about doubling metric measure spaces, ball quasi-Banach
function spaces, and the operator $L$. We then introduce the local Hardy, atomic,
molecular, and maximal spaces and state the main results. Section \ref{proof of mr}
is devoted to their proofs. In particular, Subsection \ref{fenzikehua} proves the
atomic and molecular characterizations, Subsection \ref{guanxi} proves the relations
between local and global Hardy spaces, and Subsection \ref{jida} proves the maximal
function characterizations and the comparison with the local atomic space. In
Section \ref{apply}, we apply the general results to Lebesgue, Orlicz, weighted
Lebesgue, and variable Lebesgue spaces.

Finally, we make some conventions on notation. Let $\mathbb{N}:=\{1,2,\ldots\}$,
$\mathbb{Z}_+:=\mathbb{N}\cup\{0\}$, and let $\mathbb{Z}$ be the set of all integers.
For any $p\in[1,\infty]$, $p'$ denotes its conjugate index, namely,
$1/p+1/p'=1$. The symbol $C$ denotes a positive constant independent of the main
parameters, and its value may change from line to line. We use
$C_{(\alpha,\beta,\ldots)}$ or $c_{(\alpha,\beta,\ldots)}$ when the dependence on the
indicated parameters needs to be shown. If $f\leq Cg$, we write $f\lesssim g$ or
$g\gtrsim f$; if $f\lesssim g\lesssim f$, we write $f\sim g$. For any set
$E\subset\mathcal{X}$, $\mathbf{1}_E$ denotes its characteristic function and
$E^\complement$ its complement in $\mathcal{X}$. For any $x\in\mathcal{X}$ and
nonempty sets $E_1,E_2\subset\mathcal{X}$, set
$$
d(x,E_1):=\inf_{y\in E_1}d(x,y)
\quad\text{and}\quad
d(E_1,E_2):=\inf_{y\in E_1,\,z\in E_2}d(y,z).
$$
For any $x\in\mathcal{X}$ and $r\in(0,\infty)$, $B(x,r)$ denotes the ball centered at
$x$ with radius $r$. If $B$ is a ball and $\alpha\in(0,\infty)$, then $\alpha B$
denotes the ball with the same center and radius $\alpha$ times that of $B$. For any
$j\in\mathbb{N}$, let $U_j(B):=(2^jB)\setminus(2^{j-1}B)$ and $U_0(B):=B$. The symbol
$\mathscr{M}$ denotes the set of all $\mu$-measurable functions on $\mathcal{X}$. For
any $\mu$-measurable set $E\subset\mathcal{X}$ and $p\in(0,\infty)$, let
$$
L^p(E):=\left\{f\in\mathscr{M}:\
\|f\|_{L^p(E)}:=\left[\int_E|f(z)|^p\,d\mu(z)\right]^{\frac1p}<\infty\right\}.
$$
Throughout this article, $\mathcal{X}$ always denotes the fixed underlying doubling
metric measure space. When no confusion arises, we write $L^p$ instead of
$L^p(\mathcal{X})$. In the proofs, we keep the notation used in the corresponding
statements.
\section{Main Results}\label{main}
In this section, 
we state our main results.
We  first recall the concepts of doubling metric measure spaces $(\mathcal{X},d,\mu)$,  the ball quasi-Banach function $X$ on $\mathcal{X}$, and  some fundamental assumptions on the operator $L$ under consideration.
\begin{definition}
A \emph{quasi-metric space} $(\mathcal{X},d)$ is a non-empty set $\mathcal{X}$ equipped with a
\emph{quasi-metric} $d$, namely a non-negative function defined on $\mathcal{X}\times\mathcal{X}$
satisfying that, for any $x,y,z\in\mathcal{X}$,
\begin{enumerate}
\item[{\rm(i)}] $d(x,y)=0$ if and only if $x=y$;
\item[{\rm(ii)}] $d(x,y)=d(y,x)$;
\item[{\rm(iii)}] there exists a constant $A_0\in[1,\infty)$, independent of $x,y$, and $z$, such that
\begin{align*}
d(x,z)\leq A_0[d(x,y)+d(y,z)].
\end{align*}
\end{enumerate}
\end{definition}

\begin{definition}
Let $(\mathcal{X},d)$ be a quasi-metric space and $\mu$ a non-negative measure on $\mathcal{X}$.
The triple $(\mathcal{X},d,\mu)$ is called a \emph{space of homogeneous type} if $\mu$ satisfies the
following \emph{doubling condition}: there exists a constant $C_{(\mu)}\in[1,\infty)$ such that, for any
ball $B\subset\mathcal{X}$,
\begin{equation*}
\mu(2B)\leq C_{(\mu)}\mu(B).
\end{equation*}
If $A_0:=1$, then $(\mathcal{X},d,\mu)$ is called a \emph{metric measure space of homogeneous type} or,
simply, a \emph{doubling metric measure space}.
\end{definition}
In what follows, we \emph{always} assume that
$(\mathcal{X},d,\mu)$ is a  doubling metric measure space. Then
the above doubling condition further implies that, for any ball $B\subset\mathcal{X}$ and any
$\lambda\in[1,\infty)$,
\begin{equation}\label{d1}
\mu(\lambda B)\leq C_{(\mu)}\lambda^{n}\mu(B),
\end{equation}
where $n:= \log_2 C_{(\mu)}$ is called the \emph{upper dimension} of $\mathcal{X}$. Moreover,
for any $x,y\in\mathcal{X}$ and $r\in(0,\infty)$,
\begin{align}\label{d2}
V(y,r)\leq C_{(\mu)}\left[1+\frac{d(x,y)}{r}\right]^{n}V(x,r),
\end{align}
here, and thereafter, $V(x,r):=\mu(B(x,r)).$
Throughout this article, according to \cite[pp.\,587-588]{cw77}, we \emph{always} make the following assumptions
on $(\mathcal{X},d,\mu)$:
\begin{enumerate}
\item[\rm(i)] for any point $x\in\mathcal{X}$, the balls $\{B(x,r)\}_{r\in(0,\infty)}$ form a basis of
open neighborhoods of $x$;
\item[{\rm(ii)}] $\mu$ is \emph{Borel regular} which means that all open sets are $\mu$-measurable
and every set $A\subset\mathcal{X}$ is contained in a Borel set $E$ such that $\mu(A)=\mu(E)$;
\item[{\rm(iii)}] for any $x\in\mathcal{X}$ and $r\in(0,\infty)$, $\mu(B(x,r))\in(0,\infty)$;
\item[{\rm(iv)}] $\mathrm{diam}\, \mathcal{X}=\infty$ and $(\mathcal{X},d,\mu)$ is \emph{non-atomic}
which means $\mu(\{x\})=0$ for any $x\in\mathcal{X}$. Here, and thereafter, $\mathrm{diam}\, \mathcal{X}
:=\sup\{d(x,y):\ x,y\in\mathcal{X}\}$.
\end{enumerate}
Note that $\mathrm{diam}\, \mathcal{X}=\infty$ implies that $\mu(\mathcal{X})=\infty$ (see 
\cite[Lemma 5.1]{ny97}).

The following concept of ball (quasi-)Banach function spaces on $\mathcal{X}$ was introduced in \cite{syy22jga}.

\begin{definition}\label{debb}
A quasi-normed linear space $X\subset\mathscr{M}$, equipped with a \emph{quasi-norm}
$\|\cdot\|_{X}$ which makes sense for all functions in $\mathscr{M}$, is called
a \emph{ball quasi-Banach function space} (for short, $\mathrm{BQBF}$ \emph{space}) on $\mathcal{X}$ if it satisfies
the following conditions:
\begin{enumerate}
\item [(i)] for any $f\in\mathscr{M}$, $\|f\|_{X}=0$ if and only if
$f=0$ $\mu$-almost everywhere;
\item [(ii)] for any $f,g\in\mathscr{M}$, $|g|\le|f|$ $\mu$-almost everywhere
implies that $\|g\|_{X}\le \|f\|_{X}$;
\item [(iii)] for any $\{f_k\}_{n\in\mathbb{N}}\subset\mathscr{M}$ and
$f\in\mathscr{M}$, $0\le f_k\uparrow f$ $\mu$-almost everywhere as $k\to\infty$
implies that $\|f_k\|_{X}\uparrow \|f\|_{X}$ as $k\to\infty$;
\item [(iv)] for any ball $B\subset\mathcal{X}$, $\mathbf{1}_B\in X$.
\end{enumerate}
A normed linear space $X\subset\mathscr{M}$ is called a
\emph{ball Banach function space} (for short, $\mathrm{BBF}$ \emph{space}) on $\mathcal{X}$ if
it satisfies (i)-(iv) and
\begin{enumerate}
\item [(v)] for any ball $B$ of $\mathcal{X}$, there exists a positive constant $C_{(B)}$, depending
only on $B$, such that, for any $f\in X$,
$\int_{B}|f(x)|\,d\mu(x)\leq C_{(B)}\|f\|_{X}.$
\end{enumerate}
\end{definition}

\begin{remark}
\begin{itemize}
\item[\rm{(i)}]
Let $X$ be a $\mathrm{BQBF}$ space.
By \cite[Theorem 2]{dfmn21}, we find that both
(ii) and (iii) of Definition \ref{debb} imply that $X$ is complete.

\item[\rm{(ii)}]
In Definition \ref{debb},
replacing  any ball $B$ by any bounded $\mu$-measurable
set $E\subset \mathcal{X}$, then we obtain an equivalent definition of the ball (quasi-)Banach function space
on $\mathcal{X}$.

\item[\rm{(iii)}]
Also, in Definition \ref{debb}, if any ball $B$ is replaced by any $\mu$-measurable set
$E\subset\mathcal{X}$ with $\mu(E)<\infty$, then we obtain the definition of the (quasi-)Banach function
space on $\mathcal{X}$; see, for example, \cite[p.\,3, Definition 1.3]{bs88} and \cite[Definition 2.4]{yhyy23}. Obviously,
any (quasi-)Banach function space  is a ball (quasi-)Banach function space.
However, the inverse is  not true. For example, the weight Lebesgue space and the Morrey space are not quasi-Banach
function spaces in general, but they are ball quasi-Banach function spaces (see \cite[Section 7]{shyy17} for more details).

\item[\rm(iv)]
In Definition \ref{debb},
if condition (iv) is replaced
by the following \emph{saturation property} that
\begin{enumerate}
\item[\rm(a)]
for any $\mu$-measurable set $E\subset\mathcal{X}$
with $\mu(E)\in(0,\infty)$, there exists a $\mu$-measurable set $F\subset E$
with $\mu(F)\in(0,\infty)$ satisfying that $\mathbf{1}_F\in X$,
\end{enumerate}
we then obtain the concept of quasi-Banach function spaces
in \cite{ln20231}.
\end{itemize}
\end{remark}

Let us recall the concept of the associate space and the convexification of a $\mathrm{BQBF}$ space as follows
(see, for example, \cite[p.\,8, Definition 2.1]{bs88}, \cite[p.\,53]{lt79}, and \cite[Definition 2.6]{shyy17}).
\begin{definition}
Let $X$ be a $\mathrm{BBF}$ space. The \emph{associate space} $X'$ of a $\mathrm{BBF}$ space $X$ is defined by setting
$X':=\left\{f\in\mathscr{M}:\ \|f\|_{X'}<\infty\right\},$
where, for any $f\in \mathscr{M}$,
\begin{align*}
\|f\|_{X'}:=\sup\left\{\|fg\|_{L^1}: \ g\in X,\ \|g\|_{X}=1\right\}.
\end{align*}
\end{definition}
\begin{definition}
Let $p\in(0,\infty)$ and $X$ be a $\mathrm{BQBF}$ space. The \emph{p-convexification} $X^p$ of $X$ is defined by setting
\begin{equation*}
X^p:=\left\{f\in\mathscr{M}:\ \|f\|_{X^p}:=\left\||f|^p\right\|_{X}^{\frac{1}{p}}<\infty\right\}.
\end{equation*}
\end{definition}
Recall that the \emph{Hardy--Littlewood maximal operator} $\mathcal{M}$ on $\mathcal{X}$ is defined by setting, for any
$f\in\mathscr{M}$ and $x\in\mathcal{X}$,
\begin{align*}
\mathcal{M}(f)(x):=\sup_{B\ni x}\frac{1}{\mu(B)}\int_B|f(z)|\,d\mu(z),
\end{align*}
where the supremum is taken over all  balls $B$ containing $x$. The following two assumptions play
key roles throughout this article.
\begin{assumption}\label{vector1}
Let $X$ be a $\mathrm{BQBF}$ space. Assume that there exists 
$p_-\in(0,\infty)$ such that, for any given
$p\in(0,p_-)$ and $r\in(1,\infty)$,
there exists a positive constant $C$ such
that, for any $\{f_j\}_{j\in\mathbb{N}}\subset
\mathscr{M},$
\begin{align*}
\left\|\left\{\sum_{j\in\mathbb{N}}\left[\mathcal{M}(f_j)\right]^r\right\}^{\frac{1}{r}}
\right\|_{X^{\frac{1}{p}}}\leq C\left\|\left(\sum_{j\in\mathbb{N}}|f_j|^r\right)^{\frac{1}{r}}\right\|_{X^{\frac{1}{p}}}.
\end{align*}
\end{assumption}
\begin{assumption}\label{vector2}
Let $X$ be a $\mathrm{BQBF}$ space. Assume that there exist constants $s\in(0,\infty)$
and $q\in(s,\infty)$ such that $X^{\frac{1}{s}}$ is a $\mathrm{BBF}$ space and the
Hardy--Littlewood maximal operator $\mathcal{M}$ is bounded on the $\frac{1}{(q/s)'}$-convexification
of the associate space $(X^{1/s})'$, where $\frac{1}{(q/s)'}+\frac{1}{q/s}=1$.
\end{assumption}
In this article, we \emph{always} make the following assumptions on the operator $L$.
\begin{assumption}\label{l1}
$L$ is a non-negative and self-adjoint operator on $L^2$.
\end{assumption}
\begin{assumption}\label{l2}
The kernels of the semigroup $\{e^{-tL}\}_{t\in(0,\infty)}$, denoted by
$\{K_t\}_{t\in(0,\infty)}$, are measurable functions on $\mathcal{X}\times\mathcal{X}$ and satisfy the Gaussian
upper bound estimate; that is, there exist positive constants $C$ and $c$ such that,
for any $t\in(0,\infty)$ and $x,y\in\mathcal{X}$,
\begin{align*}
|K_t(x,y)|\leq  \frac{C}{V(x,\sqrt{t})}\exp\left\{-\frac{d(x,y)^2}{ct}\right\}.
\end{align*}
\end{assumption}
\begin{remark}
\begin{itemize}
\item  [$\mathrm{(i)}$] 
Let $L$ be a non-negative and self-adjoint operator on $L^2$. Then
by the functional
calculus of $L$, we find that $-L$ 
generates the $C_0$-semigroup $\{e^{-tL}\}_{t\in(0,\infty)}$. 
\item [$\mathrm{(ii)}$] 
By \cite[Theorem 6.17]{O05}, we find that, for any $k\in\mathbb{N},$ there exists a positive constant $C_{(k)}$, depending only on both $k$ and $L$, such that,
for any $x,y\in\mathcal{X}$,
\begin{align}\label{lzy}
\left|t^k\frac{\partial^k K_t(x,y)}{\partial t^k}\right|\leq
\frac{C_{k}}{V(x,\sqrt{t})}\exp\left\{-c\frac{d(x,y)^2}{t}\right\},
\end{align}
where $c$ is a positive constant depending only on $L$.
\item [$\mathrm{(iii)}$]
The typical examples of operators $ L $ satisfying Assumption \ref{l2} include the Schr\"{o}dinger operator $ L := -\Delta + V $ with $ 0 \le V \in L^1_{\mathrm{loc}}(\mathcal{X}) $\,(the set of all locally integrable functions on $\mathcal{X}$) and the second-order divergence form elliptic operator $ L := -\mathrm{div}(A \nabla) $ with $ A := \{a_{ij}\}_{i,j=1}^n $ satisfying certain conditions .We refer the reader to \cite[Section 6]{bdl18} for more examples.

\end{itemize}
\end{remark}
For any $f\in L^2,$
the \emph{local Lusin area function} $S_{L,\mathrm{loc}}(f),$ associated with $L$, and the \emph{Lusin area function} $S_{I}(f)$, associated with 
the identity operator $I$,
are  defined by setting, for any $x\in\mathcal{X},$
\begin{align*}
S_{L,\mathrm{loc}}(f)(x):=\left[\int_0^1\int_{B(x,t)}\left|t^2Le^{-t^2L}(f)(y)\right|^2\,\frac{d\mu(y)\,dt}{V(x,t)t}\right]^{\frac{1}{2}},
\end{align*}
and
\begin{align*}
S_I(f)(x):
=\left[\int_0^\infty\int_{B(x,t)}\left|t^2e^{-t^2}f(y)\right|^2\,\frac{d\mu(y)\,dt}{V(x,t)t}\right]^{\frac{1}{2}}.
\end{align*}
\begin{definition}
Let $X$ be a $\mathrm{BQBF}$ space. Then the \emph{local Hardy space} $h_{X,L}$, associated with both $X$ and
$L$, is defined as the completion of the set
$$
h_{X,L}\cap L^2:=\left\{f \in L^2:\ \|f\|_{h_{X,L}}<\infty\right\}
$$
with respect to the \emph{quasi-norm}
\begin{align*}
\|f\|_{h_{X,L}}:=\left\|S_{L,\mathrm{loc}}(f)\right\|_{X}+\left\|S_{I}(e^{-L}f)\right\|_{X}.
\end{align*}
\end{definition}
We introduce the \emph{local molecular Hardy space} $h_{X,L,\rm{mol}}^{M,p,\epsilon}$, and the \emph{local atomic Hardy space} $h_{X,L,\rm{at}}^{M,p}$ as follows.
\begin{definition}
Let $p\in[1,\infty),$
$M\in\mathbb{N}$, $\epsilon\in(0,\infty)$, and
$X$ be a $\mathrm{BQBF}$ space.
Denote  the domain of $L^M$ by $\mathrm{Dom}\,(L^M)$. 
A function $m\in L^p$ is called a \emph{local $(X,M,p,\epsilon)_{L}$-molecule}
, associated with  the ball $B\subset\mathcal{X}$, if the following 
conditions are satisfied:
\begin{itemize}
\item  [$\mathrm{(i)}$]
When $r_B\in [1,\infty)$, then,  for any $i\in\mathbb{Z}_+$,
\begin{align}\label{pmol1}
\|m\|_{L^p(U_i(B))}
\leq2^{-i(\epsilon-\frac{n}{p})}\left[\mu(B)\right]^{\frac{1}{p}}\|\mathbf{1}_{B}\|_{X}^{-1}.
\end{align}
\item  [$\mathrm{(ii)}$] When $r_B\in (0,1)$, then there exists a function $b\in \mathrm{Dom}\,(L^M)$ such that $m=L^M(b)$ and
\begin{align}\label{pmol2}
\left\|\left(r_B^{2}L\right)^j(b)\right\|_{L^p(U_i(B))}
\leq2^{-i(\epsilon-\frac{n}{p})}r_B^{2M}\left[\mu(B)\right]^{\frac{1}{p}}\|\mathbf{1}_{B}\|_{X}^{-1}
\end{align}
for any $j\in\{0,\ldots,M\}$ and $i\in\mathbb{Z}_+$.
\end{itemize}
\end{definition}
\begin{definition}
Let $p\in[1,\infty),$
$M\in\mathbb{N}$, $\epsilon\in(0,\infty)$, and $X$ be a $\mathrm{BQBF}$ space satisfying Assumption
\ref{vector2} for some $s\in(0,\infty)$ and $q\in(s,\infty)$. 
For any $f\in L^2$, $f=\sum_{j\in\mathbb{N}} \lambda_jm_j$ is called a
\emph{local molecular $(X,M,p,\epsilon)_L$-representation} of $f$ if, for any $j\in\mathbb{N}$, $m_j$ is
a local $(X,M,p,\epsilon)_L$-molecule, associated with the ball $B_j\subset\mathcal{X}$, the summation
converges in $L^2$, and $\{\lambda_j\}_{j\in\mathbb{N}}\subset[0,\infty)$ satisfies
\begin{align*}
\Lambda\left(\left\{\lambda_jm_j\right\}_{j\in\mathbb{N}}\right)
:=\left\|\left\{\sum_{j\in\mathbb{N}}
\left[\frac{\lambda_j}{\|\mathbf{1}_{B_j}\|_{X}}\right]^{s}
\mathbf{1}_{B_j}\right\}^{\frac{1}{s}}\right\|_{X}<\infty.
\end{align*}
Let
\begin{align*}
&\widetilde{h}^{M,p,\epsilon}_{X,L,\rm{mol}}\\
&\quad:=\left\{f\in L^2:\ f \
\text{has a local molecular}\ (X,M,p,\epsilon)_L\text{-representation}\right\}
\end{align*}
equipped with the \emph{quasi-norm} $\|\cdot\|_{\widetilde{h}^{M,p,\epsilon}_{X,L,\rm{mol}}}$
defined by setting, for any $f\in\widetilde{h}^{M,p,\epsilon}_{X,L,\rm{mol}}$,
\begin{align*}
\|f\|_{\widetilde{h}^{M,p,\epsilon}_{X,L,\rm{mol}}}
&:=\inf\Bigg\{\Lambda\left(\left\{\lambda_jm_j\right\}_{j\in\mathbb{N}}\right):\
f=\sum_{j\in\mathbb{N}} \lambda_jm_j \\
&\quad\quad\quad\quad \ \text{is a local molecular}\ (X,M,p,\epsilon)_L\text{-representation}\Bigg\},
\end{align*}
where the infimum is taken over all the local molecular $(X,M,p,\epsilon)_L$-representations of $f$ as above.
The \emph{local molecular Hardy  space} $h^{M,p,\epsilon}_{X,L,\rm{mol}}$ is
defined as the completion of $\widetilde{h}^{M,p,\epsilon}_{X,L,\rm{mol}}$ with
respect to the quasi-norm $\|\cdot\|_{h^{M,p,\epsilon}_{X,L,\rm{mol}}}$.
\end{definition}
\begin{definition}
Let $M\in\mathbb{N}$, $p\in[1,\infty]$, and
$X$ be a $\mathrm{BQBF}$ space.
A function $a\in L^p$ is called a \emph{local $(X,M,p)_{L}$-atom}
,associated with the ball $B\subset\mathcal{X}$, if 
$\mathrm{supp}\,(a)\subset B$ and 
the following 
conditions are satisfied:
\begin{itemize}
\item  [$\mathrm{(i)}$]
When $r_B\in [1,\infty)$, then
\begin{align*}
\|a\|_{L^p}\leq [\mu(B)]^{\frac{1}{p}}\|\mathbf{1}_{B}\|_{X}^{-1}.
\end{align*}
\item  [$\mathrm{(ii)}$] When $r_B\in (0,1)$, then there exists a function $b\in \mathrm{Dom}\,(L^M)$ such that $a=L^M(b)$, $\mathrm{supp}\,(L^i(b))\subset B$, and
\begin{align}\label{atom2}
\left\|\left(r_B^{2}L\right)^i(b)\right\|_{L^p}
\leq
r_B^{2M}\left[\mu(B)\right]^{\frac{1}{p}}\|\mathbf{1}_{B}\|_{X}^{-1}
\end{align}
for any $i\in\{0,\ldots,M\}$.
\end{itemize}
In a similar way, the set $\widetilde{h}^{M,p}_{X,L,\rm{at}}$
and the \emph{local atomic Hardy space} $h^{M,p}_{X,L,\rm{at}}$ are defined
in the same way, respectively, as $\widetilde{h}^{M,p,\epsilon}_{X,L,\rm{mol}}$ and
$h^{M,p,\epsilon}_{X,L,\rm{mol}}$ with local $(X,M,p,\epsilon)_L$-molecules replaced by local $(X,M,p)_L$-atoms.
\end{definition}

Our first main result is the following atomic and molecular characterizations of $h_{X,L}.$
\begin{theorem}\label{local-mo-at}
Let $X$ be a $\mathrm{BQBF}$ space satisfying Assumptions \ref{vector1} and \ref{vector2} for some $p_-\in (0,\infty)$, $s\in (0,\min\{p_-,1\}]$, and $q \in (s,2]$. Assume that $p\in (1,\infty)\cap [q,\infty)$,  $M\in(\frac{n}{2}(\frac{1}{s}-\frac{1}{p}),\infty)\cap\mathbb{N}$, and $\epsilon\in
(\frac{n}{s},\infty)$.
Then the spaces $h_{X,L}$, $h^{M,p,\epsilon}_{X,L,\rm{mol}},$
and $h^{M,p}_{X,L,\rm{at}}$ coincide with
equivalent quasi-norms.
\end{theorem}
The proof of Theorem \ref{local-mo-at} is given in Subsection \ref{fenzikehua}.
As an application of Theorem \ref{local-mo-at}, we 
establish the relationship theorem of local Hardy space $h_{X,L}$  and global Hardy space $H_{X,L}.$
\begin{definition}
Let $X$ be a $\mathrm{BQBF}$ space. The \emph{ Hardy space} $H_{X,L}$, associated with both $X$ and
$L$, is defined as the completion of the set
$$
H_{X,L}\cap L^2:=\left\{f \in L^2:\ \|f\|_{H_{X,L}}<\infty\right\}
$$
with respect to the \emph{quasi-norm}
\begin{align*}
\|f\|_{H_{X,L}}:=\left\|S_{L}(f)\right\|_{X},
\end{align*}
where, the \emph{Lusin area function} $S_L(f)$, associated with $L$, is defined by setting, for any $x \in \mathcal{X}$,
\begin{align*}
S_{L}(f)(x):=\left[\int_0^\infty\int_{B(x,t)}\left|t^2Le^{-t^2L}(f)(y)\right|^2\,\frac{d\mu(y)\,dt}{V(x,t)t}\right]^{\frac{1}{2}}.
\end{align*}
\end{definition}
By the  assumption that 
the operator $L$  satisfies Assumptions \ref{l1} and \ref{l2}, we  conclude that
the operator $L+I$ also satisfies Assumptions \ref{l1} and \ref{l2}. Thus, the  space 
$H_{X,L+I}$ is defined in the same way  as $H_{X,L}$ with $L$
replaced by $X+L.$ 
As a consequence  of Theorem \ref{local-mo-at},
We have the following theorem, which is proved in Subsection \ref{guanxi}.
\begin{theorem}\label{hH}
Let $X$ be a $\mathrm{BQBF}$ space
satisfying Assumptions \ref{vector1} and \ref{vector2}
for some $p_-\in(0,\infty)$, $s\in (0,\mathrm{min}\{p_-,1\}]$, and $q\in(s,2]$. 
\begin{itemize}
\item[\rm(i)] Assume that $\mathrm{inf}\,\sigma(L)\in (0,\infty)$, where $\sigma(L)$ denotes the spectrum of $L$. Then the spaces $h_{X,L}$ and $H_{X,L}$ coincide with equivalent quasi-norms.
\item[\rm(ii)] For any given $m\in\mathbb{N}$, the spaces $h_{X,L}$ and $H_{X,L+mI}$ coincide with equivalent quasi-norms.
\item[\rm(iii)] Assume that $p_0\in(1,\infty)$, $X^{\frac{1}{p_0}}$ is a $\mathrm{BBF}$ space, and the \emph{Hardy-Littlewood} maximal operator $\mathcal{M}$ is bounded on $(X^{\frac{1}{p_0}})'$. Then the spaces $h_{X,L}\cap L^2$, $H_{X,L}\cap L^2$, and $X\cap L^2$ coincide with equivalent quasi-norms.
\end{itemize}
\end{theorem}
Next, we establish several maximal function characterizations of the local  Hardy space $h_{X,L}.$ Denote the space of all Schwartz functions on $\mathbb{R}$
by $\mathcal{S}(\mathbb{R})$  equipped with the well-known topology determined by a countable family of norms.
\begin{definition}
Let $X$ be a $\mathrm{BQBF}$ space.
\begin{itemize}
\item  [$\mathrm{(i)}$]Let $\phi\in\mathcal{S}(\mathbb{R})$ be an even function with $\phi(0)=1$. 
For any given $\alpha\in(0,\infty)$ and  any $f\in L^2$, the \emph{local non-tangential maximal function} $\phi_{L,\alpha}^\ast(f)$
and the \emph{local radial maximal function} $\phi^+_L(f)$
are, respectively,  
defined by setting, for any $x\in\mathcal{X}$,
\begin{equation*}
\phi_{L,\,\alpha}^\ast(f)(x):=\sup_{t\in(0,1)}
\sup_{y\in B(x,\alpha t)}
\left|\phi\left(t\sqrt{L}\right)(f)(y)\right|
\end{equation*}
and
\begin{align*}
\phi_L^+(f)(x):=\sup_{t\in(0,1)}\left|\phi\left(t\sqrt L\right)(f)(x)\right|.
\end{align*}
For the simplicity of the presentation, in the case where $\phi(x):=e^{-|x|^2}$ for any $x\in\mathbb{R}$,
we denote
$\phi_{L,\,1}^\ast(f)$ 
and $\phi^+_L(f)$,
respectively, by $f_{L}^\ast$ and
$f_{L}^+$.

\item  [$\mathrm{(ii)}$]
The \emph{local Hardy  space} $h_{X,L,\max}^{\phi,\alpha},$ associated with both $X$ and $L$,
is defined as the completion of the set
$$h_{X,L,\max}^{\phi,\alpha}\cap L^2:=\left\{f\in L^2:\
\|f\|_{h_{X,L,\max}^{\phi,\alpha}}:=\|\phi_{L,\alpha}^\ast(f)
\|_X<\infty\right\}$$
with respect to the quasi-norm $\|\cdot\|_{h_{X,L,\max}^{\phi,\alpha}}$.
The \emph{local Hardy  space} $h_{X,L,\mathrm{rad}}^\phi$ is defined in the same way
as $h_{X,L,\max}^{\phi,\alpha}$ with $\phi_{L,\alpha}^\ast$ replaced by
$\phi_L^+$. In the case  where 
$\alpha:=1$ and $\phi(x):=e^{-|x|^2}$ for any $x\in\mathbb{R}$,
we denote the spaces $h_{X,L,\max}^{\phi,\alpha}$ and $h_{X,L,\mathrm{rad}}^\phi$,
respectively, by 
$h_{X,L,\max}$ and $h_{X,L,\mathrm{rad}}$.
\end{itemize}
\end{definition}
The following theorem is the second main result of this article, whose proof is given in Subsection \ref{jida}.
\begin{theorem}\label{thm-maxi}
Let $X$ be a $\mathrm{BQBF}$ space satisfying Assumptions \ref{vector1} and \ref{vector2} for some $p_-\in (0,\infty)$, $s\in (0,1]$, and $q\in (s,\infty)$. 
Let  $M\in\mathbb{N}\cap (\frac{n}{2\min\{s,p_-\}},\infty)$, $p\in (1,\infty]\cap[q,\infty],$
$\alpha\in(0,\infty),$
and   $\varphi\in\mathcal{S}(\mathbb{R})$ be an even function satisfying $\varphi(0)=1$.
Then the spaces
$h_{X,L,\mathrm{at}}^{M,p}$,
$h_{X,L,\mathrm{max}}^{\varphi,\alpha}$,
$h_{X,L,\mathrm{rad}}^\varphi$,
$h_{X,L,\mathrm{max}}$, and $h_{X,L,\mathrm{rad}}$ coincide with equivalent quasi-norms.
\end{theorem}
As a direct  corollary of
Theorems \ref{local-mo-at} and \ref{thm-maxi}, we immediately
have the following theorem.
\begin{theorem}
Let $X$ be a $\mathrm{BQBF}$ space satisfying Assumptions \ref{vector1} and \ref{vector2} for some $p_-\in (0,\infty)$, $s\in (0,\min\{p_-,1\}]$, and $q\in (s,2]$. 
Assume that $p\in (1,\infty]\cap[q,\infty],$
$M\in\mathbb{N}\cap (\frac{n}{2\min\{s,p_-\}},\infty)$, 
and $\epsilon\in(\frac{n}{s},\infty)$.
Then the spaces
$h_{X,L}$, $h_{X,L,\mathrm{mol}}^{M,p,\epsilon}$,
$h_{X,L,\mathrm{at}}^{M,p}$,
$h_{X,L,\mathrm{max}}$, and $h_{X,L,\mathrm{rad}}$ coincide with equivalent quasi-norms.
\end{theorem}
As an application of Theorem \ref{thm-maxi}, we show the the local Hardy space $h_{X,L,\mathrm{max}}$,
associated with the operator $L$,
coincides with local atomic Hardy space $h_{X,\mathrm{at}}^p$ under some additional assumptions of $L.$
\begin{definition}
Let $p\in[1,\infty]$ and 
$X$ be a $\mathrm{BQBF}$ space. 
Then a function $a\in L^2$ is called a \emph{local $(X,p)$-atom} if there exists a ball $B\subset \mathcal{X}$ such that 
\begin{itemize}
\item [$\mathrm{(i)}$] 
$$\mathrm{supp}\, (a)\subset B,$$
\item [$\mathrm{(ii)}$] 
\begin{align*}
\|a\|_{L^p}\leq [\mu(B)]^{\frac{1}{p}}\|\mathbf{1}_B\|_{X}^{-1},
\end{align*}
\item  [$\mathrm{(ii)}$] when $r_B\in (0,1)$,
\begin{align*}
\int_{\mathcal{X}}a(x)\,d\mu(x)=0.
\end{align*}
\end{itemize}
In a similar way, the set $\widetilde{h}^{p}_{X,\rm{at}}$
and the \emph{local atomic Hardy space} $h^{p}_{X,\rm{at}}$ are defined
in the same way, respectively, as $\widetilde{h}^{M,p,\epsilon}_{X,L,\rm{mol}}$ and
$h^{M,p,\epsilon}_{X,L,\rm{mol}}$ with local $(X,M,p,\epsilon)_L$-molecules replaced by local $(X,p)$-atoms.
\end{definition}
Recall that, in this article, we always assume that $L$ satisfies Assumptions \ref{l1}
and \ref{l2} and $\{K_t\}_{t\in(0,\infty)}$ denotes the kernel of $\{e^{-tL}\}_{t\in(0,\infty)}$
. To obtain the relationship theorem of local Hardy spaces $h_{X,L,\mathrm{max}}$ and $h_{X,\mathrm{at}}^p$, we need some additional assumptions on $L$. 
Assume that there exist 
positive constants $c, C,$ and $\delta$ such that,
for any $x\in\mathcal{X}$ and $t\in(0,\infty)$,
\begin{align}\label{1914}
\int_{\mathcal{X}}K_t(x,y)\,d\mu(y)=1,
\end{align}
and, 
for any $t\in(0,\infty)$ and $x,x', y\in \mathcal{X}$ with $d(x,x')\leq[\sqrt{t}+d(x,y)]/2$,
\begin{align}\label{A3}
\left|K_t(x,y)-K_t(x',y)\right|\leq C\left[\frac{d(x,x')}{\sqrt{t}}\right]^{\delta}\frac{1}{V(x,\sqrt{t})}\mathrm{exp}\left(-\frac{d(x,y)^2}{ct}\right).
\end{align}
\begin{theorem}\label{key}
Let $L$ be an operator on $L^2$ satisfying
Assumptions \ref{l1} and \ref{l2} and  \eqref{1914} and  \eqref{A3} for some $\delta\in (0,\infty).$
Assume that  $X$ is a $\mathrm{BQBF}$ space
satisfying Assumptions \ref{vector1} and \ref{vector2}
for some $p_-\in(0,\infty)$, $s\in (\frac{n}{n+\delta},\min\{p_-,1\}]$, and $q\in(s,\infty)$. Let $p\in(1,\infty]\cap(q,\infty].$
Then the spaces  $h_{X,L,\mathrm{max}}$  and $h_{X,\mathrm{at}}^p$ coincide with equivalent quasi-norms.
\end{theorem}
The proof of Theorem \ref{key} is given in Subsection \ref{jida}.
\section{Proof of Main Results}\label{proof of mr}
In this section, we prove the main results in Section \ref{main}.
\subsection{Proof of Theorem \ref{local-mo-at}\label{fenzikehua}}
In this subsection, we prove Theorem \ref{local-mo-at}, which can be deduced from Theorems \ref{th1-1} and \ref{recon} below.
\begin{theorem}\label{th1-1}
Let  $X$ be a $\mathrm{BQBF}$ space satisfying Assumptions \ref{vector1} and \ref{vector2} for some $p_-\in(0,\infty)$, $s\in (0,\min\{p_-,1\}]$, and  $q\in(s,2]$. Assume that $M \in \mathbb{N}$ and $p\in (1,\infty)$. Then there exists a positive constant $C$ such that, for any $f \in h_{X,L} \cap L^2$,  $f \in \widetilde{h}_{X,L,\mathrm{at}}^{M,p}$ and $\|f\|_{\widetilde{h}_{X,L,\mathrm{at}}^{M,p}} \leq C\|f\|_{h_{X,L}}$.
\end{theorem}
\begin{theorem}\label{recon}
Let $X$ be a $\mathrm{BQBF}$ space
satisfying Assumption \ref{vector2}
for some $s\in (0,1]$. Let
$p\in (1,\infty)\cap [q,\infty)$,
$M\in(\frac{n}{2}(\frac{1}{s}-\frac{1}{p}),\infty)\cap\mathbb{N}$, and 
$\epsilon\in
(\frac{n}{s},\infty)$. Then there exists a positive constant $C$ such that, for any $f\in
\widetilde{h}^{M,p,\epsilon}_{X,L,\rm{mol}}$,
\begin{align*}
\|f\|_{h_{X,L}}\leq C\|f\|_{\widetilde{h}^{M,p,\epsilon}_{X,L,\rm{mol}}}.
\end{align*}	
\end{theorem}
\begin{proof}[Proof of Theorem \ref{local-mo-at}]
By Theorems \ref{th1-1} and  \ref{recon}, we conclude that 
\begin{align*}
\widetilde{h}^{M,p}_{X,L,\mathrm{at}}
\subset
\widetilde{h}^{M,p}_{X,L,\mathrm{mol}}
\subset
h_{X,L}\cap L^2
\subset
\widetilde{h}^{M,p}_{X,L,\mathrm{at}}
\subset
\widetilde{h}^{M,p}_{X,L,\mathrm{mol}}
\end{align*}
and the spaces $\widetilde{h}^{M,p}_{X,L,\mathrm{at}}, \widetilde{h}^{M,p}_{X,L,\mathrm{mol}}$, and $h_{X,L}\cap L^2$ coincide with equivalent quasi-norms. From this and a density argument, it follows that the desired conclusions hold.
This finishes the proof of Theorem \ref{local-mo-at}.
\end{proof}
In the remainder of this subsection, we prove  Theorems \ref{th1-1} and \ref{recon}. To this end,  we need to do some preparations. We first recall the concept of the \emph{Hardy space} $H_{X,I}$, associated with the space $X$ and the identity operator $I$.
\begin{definition}
Let $X$ be a $\mathrm{BQBF}$ space. Then the \emph{ Hardy space} $H_{X,I}$, associated with  $X$ and $I$, is defined as the completion of the set
$$
H_{X,I}\cap L^2:=\left\{f \in L^2:\ \|f\|_{H_{X,I}}<\infty\right\}
$$
with respect to the \emph{quasi-norm}
\begin{align*}
\|f\|_{H_{X,I}}:=\left\|S_{I}(f)\right\|_{X}.
\end{align*}
\end{definition}
\begin{definition}
Let $M\in\mathbb{N}$, $p\in[1,\infty)$,
and $X$ be a $\mathrm{BQBF}$ space.
A function $a\in L^p$ is called an \emph{$(X,M,p)_{I}$-atom},
associated with  the ball $B\subset\mathcal{X}$, if 
$\mathrm{supp}\, (a)\subset B$ and
\begin{align}\label{934}
\left\|a\right\|_{L^p}
\leq
r_B^{2i}\left[\mu(B)\right]^{\frac{1}{p}}\|\mathbf{1}_{B}\|_{X}^{-1}
\end{align}
for any $i\in\{0,\ldots,M\}$.
\end{definition}
\begin{proposition}\label{945}
Let $X$ be a $\mathrm{BQBF}$ space space  satisfying Assumption \ref{vector1} for some $p_-\in(0,\infty)$. Let $p\in(1,\infty)$, $M\in \mathbb{N}\cap (\frac{n}{2}(\frac{1}{\min\{p,p_-\}}-\frac{1}{p}),\infty)$, and
$a$ be an $(X,M,p)_{I}$-atom, associated with the ball $B \subset \mathcal{X}$. Then there exists a positive constant $C$, depending only on $X$, such that
C$a$ is an $(X,M,p)_{I}$-atom, associated with a ball of radius greater than or equal than $1.$
\end{proposition}
\begin{proof}
Indeed, if $r_B\in [1,\infty)$, then the desired conclusion holds  automatically. If $r_B\in (0,1)$, let $\widetilde{B}:=B(x_B,1)$, where $x_B$ denotes the center of $B$.
Let $r\in(0,\min\{p,p_-\})$ such that $M>\frac{n}{2}(\frac{1}{r}-\frac{1}{p}).$ 
Observe that, for any $x\in\mathcal{X},$
\begin{align*}
\mathbf{1}_{\widetilde{B}}(x)
\leq
\frac{\mu(\widetilde{B})}{\mu(B)}\mathcal{M}(\mathbf{1}_{B})(x).
\end{align*}
Using this and Assumption \ref{vector1},
we obtain
\begin{align*}
\|\mathbf{1}_{\widetilde{B}}\|_{X^{\frac{1}{r}}}
\lesssim
\frac{\mu(\widetilde{B})}{\mu(B)}\|\mathbf{1}_{B}\|_{X^{\frac{1}{r}}}.
\end{align*}
This, together with  \eqref{934} with $i:=M$ and  \eqref{d1}, implies that
\begin{align*}
\left\|a\right\|_{L^p}
&\leq
r_B^{2M}\left[\mu(B)\right]^{\frac{1}{p}}\|\mathbf{1}_{B}\|_{X}^{-1}
\lesssim
r_B^{2M}\left[\frac{\mu(\widetilde{B})}{\mu(B)}\right]^{\frac{1}{r}-\frac{1}{p}}\left[\mu(\widetilde{B})\right]^{\frac{1}{p}}\|\mathbf{1}_{\widetilde{B}}\|_{X}^{-1}\\
&\lesssim
r_B^{2M-n(\frac{1}{r}-\frac{1}{p})}\left[\mu(\widetilde{B})\right]^{\frac{1}{p}}\|\mathbf{1}_{\widetilde{B}}\|_{X}^{-1}
\leq
\left[\mu(\widetilde{B})\right]^{\frac{1}{p}}\|\mathbf{1}_{\widetilde{B}}\|_{X}^{-1}.
\end{align*}
This shows that the desired conclusion holds, which finishes the proof of Proposition \ref{945}.
\end{proof}

We have the following  atomic decomposition  theorem of $H_{X,I}$, which is useful in the proof of Theorem \ref{th1-1}.
\begin{theorem}\label{HI}
Let 
$X$ be a $\mathrm{BQBF}$ space satisfying Assumption \ref{vector1} for some $p_-\in(0,\infty)$. Let $M\in\mathbb{N}$ and 
$s\in (0,p_-]$. Then, for any $f\in H_{X,I}\cap L^2,$ there exists 
a sequence $\{\lambda_{j}\}_{j\in\mathbb{N}}\subset [0,\infty)$ and a sequence $\{a_j\}_{j\in\mathbb{N}}$ of $(X,M,2)_I$-atoms, associated, respectively, with balls $\{B_{j}\}_{j\in\mathbb{N}}$ such that
\begin{align}\label{huaiyi1}
f=\sum_{j\in\mathbb{N}}\lambda_ja_j\ \ \text{in}\ \ L^2
\end{align}
and
\begin{align}\label{huaiyi2}
\left\|\left\{\sum_{j\in\mathbb{N}}
\left[\frac{\lambda_j}{\|\mathbf{1}_{B_j}\|_{X}}\right]^{s}\mathbf{1}_{B_j}
\right\}^{\frac{1}{s}}\right\|_X
\lesssim\|f\|_{H_{X,I}},
\end{align}
where the implicit positive constant  is independent of $f$.
\end{theorem}
\begin{proof}
It is easy to check that, for any
$t\in (0,\infty)$, any  closed sets $E, F \subset \mathcal{X}$,
and $f \in L^2$ with $\mathrm{supp}\,(f)\subset E$,
the identity operator $I$ satisfies
\begin{align*}
\left\|e^{-tI}(f)\right\|_{L^2(F)} 
\leq 
\begin{cases}
0,&\ \ \text{if}\ \ \mathrm{dist}(E,F)\in (0,\infty);\\
e^{-t}\|f\|_{L^2(E)},&\ \ \text{if}\ \ \mathrm{dist}(E,F)=0.
\end{cases}
\end{align*}
Thus, $\{e^{-tI}\}_{t\in (0,\infty)}$ satisfies the Davies--Gaffney estimate. By this and
\cite[Proposition 3.16]{lyyy26cms},
we conclude that 
the desired conclusions hold. This
finishes the proof of Theorem \ref{HI}.
\end{proof}

Let us recall the concept of a unit cube structure on $\mathcal{X}$, which was introduced in \cite[Definition 4.1]{cmm11}.
\begin{definition}\label{unit}
A unit cube structure on $\mathcal{X}$ is a
countable collection $\mathcal{Q}$ of pairwise disjoint measurable sets that cover $\mathcal{X}$,
for which there exists $\delta\in (0,1]$ and a sequence of balls $\{B_j\}_{j\in\mathbb{N}}$ on $\mathcal{X}$ of radius equal to
1 such that $\delta B_j\subset Q_j\subset B_j$ for any $Q_j\in \mathcal{Q}.$
\end{definition}
\begin{remark}
By \cite[Lemma 4.1]{cmm11}, we find that such a unit cube structure exists.
\end{remark}
The following proposition is precisely \cite[Proposition 2.14]{syy22fm}. 
\begin{proposition}\label{pras}
Let $X$ be a $\mathrm{BQBF}$ space satisfying Assumption \ref{vector2} for some $s\in(0,\infty)$
and $q\in(s,\infty)$. Assume that $p\in[q,\infty)$. Then there exists a positive constant $C$
such that, for any $\theta\in[1,\infty)$, $\{\lambda_j\}_{j\in\mathbb{N}}\subset (0,\infty)$, and
$\{a_j\}_{j\in\mathbb{N}}\subset L^p$ satisfying both  $\|a_j\|_{L^p}\le
\lambda_j[\mu(B_j)]^{\frac{1}{p}}$ and $\mathrm{supp}\,(a_j)\subset \theta B_j$ for any $j\in\mathbb{N}$,
\begin{equation*}
\left\|\sum_{j\in\mathbb{N}}\left|a_j\right|^{s}\right\|_{X^{\frac{1}{s}}}
\le C\theta^{(1-\frac{s}{p})n}\left\|\sum_{j\in\mathbb{N}}\left|\lambda_j\right|^{s}
\mathbf{1}_{B_j}\right\|_{X^{\frac{1}{s}}}.
\end{equation*}
\end{proposition}
\begin{theorem}\label{dec-HI}
Let $\mathcal{Q}$ be a unit cube structure on $\mathcal{X}$.
Let 
$X$ be a $\mathrm{BQBF}$ space satisfying Assumptions \ref{vector1} and \ref{vector2} for some $p_-\in(0,\infty)$, $s\in (0,\min\{p_-,1\}]$, and  $q\in(s,2]$.
Then, for any $f\in H_{X,I}\cap L^2,$
\begin{align*}
\left\|\sum_{Q_i\in\mathcal{Q}}[\mu(Q_i)]^{-\frac{1}{2}}\|f\mathbf{1}_{Q_i}\|_{L^2}\mathbf{1}_{Q_i}\right\|_X
\lesssim
\|f\|_{H_{X,I}}.
\end{align*}
\end{theorem}
\begin{proof}
Let $f\in H_{X,I}\cap L^2$.
By Theorem \ref{HI}, we find that 
there exists 
a sequence $\{\lambda_{j}\}_{j\in\mathbb{N}}\subset [0,\infty)$ and a sequence $\{a_j\}_{j\in\mathbb{N}}$ of $(X,M,2)_I$-atoms, associated, respectively, with balls $\{B_{j}\}_{j\in\mathbb{N}}$ such that \eqref{huaiyi1} and \eqref{huaiyi2} hold.
By Proposition \ref{945}, we can assume that $r_{B_j}\in[1,\infty)$ for any $j\in\mathbb{N}.$
Recall that, for any $\theta\in(0,1]$ and
$\{b_j\}_{j\in\mathbb{N}}\subset [0,\infty)$,
\begin{align}\label{jiben}
\sum_{j\in\mathbb{N}}b_j\leq\left[ \sum_{j\in\mathbb{N}}b_j^\theta\right]^{\frac{1}{\theta}}.
\end{align}
For any $j\in\mathbb{N}$, let $\mathrm{I}_j:=\{i\in\mathbb{N}:\ Q_i\cap B_j\not=\emptyset\}.$
Using \eqref{huaiyi1} and \eqref{jiben}, we find that
\begin{align*}
\left\|\sum_{Q_i\in\mathcal{Q}}[\mu(Q_i)]^{-\frac{1}{2}}\|f\mathbf{1}_{Q_i}\|_{L^2}\mathbf{1}_{Q_i}\right\|_X
&\leq
\left\|\sum_{Q_i\in\mathcal{Q}}\sum_{j\in\mathbb{N}}\lambda_j[\mu(Q_i)]^{-\frac{1}{2}}\|a_j\mathbf{1}_{Q_i}\|_{L^2}\mathbf{1}_{Q_i}\right\|_X\\
&\leq
\left\|\left\{\sum_{j\in\mathbb{N}}\left[\sum_{i\in\mathbf{I}_j}\lambda_j[\mu(Q_i)]^{-\frac{1}{2}}\|a_j\mathbf{1}_{Q_i}\|_{L^2}\mathbf{1}_{Q_i}\right]^s\right\}^{\frac{1}{s}}\right\|_X.
\end{align*}
For any $j\in\mathbb{N}, $ let 
\begin{align*}
A_j:=\sum_{i\in \mathrm{I}_j}\lambda_j[\mu(Q_i)]^{-\frac{1}{2}}\|a_j\mathbf{1}_{Q_i}\|_{L^2}\mathbf{1}_{Q_i}.
\end{align*}
By the fact  that $\mathcal{Q}=\{Q_i\}_{i\in\mathbb{N}}$ is disjoint and \eqref{934}, we have, for any $j \in \mathbb{N}$,
\begin{align}\label{19228}
\|A_j\|_{L^2}
=
\left\{\sum_{i\in \mathrm{I}_j}\lambda_j^2\|a_j\mathbf{1}_{Q_i}\|_{L^2}^2\right\}^{\frac{1}{2}}
=
\lambda_j\|a_j\|_{L^2}\leq \lambda_j[\mu(B_j)]^{\frac{1}{2}}\|\mathbf{1}_{B_j}\|_{X}^{-1}.
\end{align}
Since $a_j$ is supported in $B_j$ with $r_{B_j}\in [1,\infty)$ and $Q_i$ is contained in a ball with radius equal to 1, 
it follows that $A_j$ is supported in $3B_j$. From this, \eqref{19228}, Proposition \ref{pras}, and \eqref{huaiyi2}, it follows that
\begin{align*}
&\left\|\sum_{Q_i\in\mathcal{Q}}[\mu(Q_i)]^{-\frac{1}{2}}\|f\mathbf{1}_{Q_i}\|_{L^2}\mathbf{1}_{Q_i}\right\|_X\\
&\quad\leq
\left\|\left\{\sum_{j\in\mathbb{N}}[A_i]^s\right\}^{\frac{1}{s}}\right\|_X
\lesssim
\left\|\left\{\sum_{j\in\mathbb{N}}
\left[\frac{\lambda_j}{\|\mathbf{1}_{B_j}\|_{X}}\right]^{s}\mathbf{1}_{B_j}
\right\}^{\frac{1}{s}}\right\|_X
\lesssim \|f\|_{H_{X,I}}.
\end{align*}
This finishes the proof of Theorem \ref{dec-HI}.
\end{proof}
To prove Theorem \ref{th1-1}, we need the decomposition theorem of the \emph{local tent space} $t_X$. For any open set $O\subset\mathcal{X}$, define the  \emph{local tent $t(O)$} by setting
\begin{align*}
t(O):=\left\{(x,t) \in \mathcal{X} \times (0,1] :\ d(x,O^\complement)\geq t\right\}.
\end{align*} 
For any given $\mathrm{BQBF}$ space $X $, the \emph{local $X$-tent space} $t_X$ is defined to be the set of all $\mu$-measurable functions $f:\ \mathcal{X}\times(0,1] \to {\mathbb C}$ with the
following finite \emph{quasi-norm}
$$\|f\|_{t_X}:=\left\|{\mathcal A}_{\mathrm{loc}}(f)\right\|_{X },$$
where, for any $x\in\mathcal{X},$
\begin{align*}
\mathcal{A}_{\mathrm{loc}}(f)(x):=\left[\int_0^1\int_{B(x,t)}|f(y,t)|^2
\,\frac{d\mu(y)\,dt}{V(x,t)t}\right]^{\frac{1}{2}}.
\end{align*}
In particular,  when  $X:=L^p$  with $p\in (0,\infty)$, we denote $t_X$ by $t^p$.
\begin{definition}
Let $X $ be a $\mathrm{BQBF}$ space and $p\in(1,\infty)$. A measurable function $a:\ \mathcal{X}\times(0,1] \to
{\mathbb C}$ is said to be a \emph{$(t_X,p)$-atom} if there exists a ball $B\subset\mathcal{X}$ such that
\begin{enumerate}
\item[(i)] ${\rm supp}\,(a):=\{(x,t)\in\mathcal{X}\times(0,1]:\ a(x,t)\neq0\} \subset t(B);$
\item[(ii)] $\|a\|_{t^p}\le[\mu(B)]^{\frac{1}{p}}\|\mathbf{1}_B\|_{X }^{-1}$.
\end{enumerate}
Furthermore, if $a$ is a $(t_X,p)$-atom for any $p\in(1,\infty)$, then $a$ is called a \emph{$(t_X,\infty)$-atom}.
\end{definition}
We have the following atomic decomposition of the local tent space $t_X(\mathcal{X}\times(0,1])$; since its proof is similar to \cite[Theorem 3.9]{lyyy26cms}, we omit the details here.
\begin{theorem}\label{local-tent}
Assume that $X$ is a  $\mathrm{BQBF}$ space satisfying Assumption \ref{vector1} for some
$p_-\in(0,\infty)$. Let $f\in t_X(\mathcal{X}\times(0,1])$ and $s\in(0,p_-]$. Then there exist a sequence
$\{\lambda_j\}_{j\in\mathbb{N}}\subset[0,\infty)$ and a sequence $\{a_j\}_{j\in\mathbb{N}}$ of
$(t_X,\infty)$-atoms associated, respectively,  with the balls $\{B_j\}_{j\in\mathbb{N}}$ such that,
for almost every $(x,t)\in\mathcal{X}\times(0,\infty)$,
\begin{equation}\label{ad-tent-ae}
f(x,t)=\sum_{j\in\mathbb{N}} \lambda_j a_j(x,t)
\end{equation}
and
\begin{align*}
\left\|\left\{\sum_{j\in\mathbb{N}}
\left[\frac{\lambda_j}{\|\mathbf{1}_{B_j}\|_{X  }}\right]^{s}\mathbf{1}_{B_j}
\right\}^{\frac{1}{s}}\right\|_X
\lesssim\|f\|_{t_X},
\end{align*}
where the implicit positive constant is independent of $f$. Moreover,
if $f\in t^2\cap t_X,$ then \eqref{ad-tent-ae} holds true in $t^2.$
\end{theorem}

In what follows, for any operator $T$, we denote its integral kernel by $K_T$. 
Let $L$ be an operator satisfying both Assumptions
\ref{l1} and \ref{l2}.
By \cite[Theorem 3.14]{cs08}, we find that there exists a positive constant $C$ such that, for any $t\in(0,\infty)$,
the kernel $K_T$ of the operator $T:=\cos\left(t\sqrt{L}\right)$ satisfies
\begin{align*}
\mathrm{supp}\,(K_T)\subset D_t:=\left\{(x,y)\in\mathcal{X}\times\mathcal{X}:\ d(x,y)\leq C t\right\}.
\end{align*}
Since the precise value of $C$ is inessential, we always assume that $C:=1.$
Let $\varphi\in C^\infty_{\rm{c}}(\mathbb{R})$ be an even function satisfying $\int_{\mathbb{R}}\varphi(x)\,dx=2\pi$ and $\mathrm{supp}\,(\varphi)\subset(-1,1)$.
Denotes the Fourier transform of $\varphi$
by $\hat{\varphi}$.
For any $\ell,j\in\mathbb{Z}_+$ and $\xi\in\mathbb{C},$ let
\begin{align}\label{psi}
\Psi^{(\ell)}(\xi):=\frac{d^\ell\hat{\varphi}(\xi)}{d\xi^\ell}
\ \ \text{and}\ \ 
\Psi^{(\ell)}_{j}(\xi):=\xi^{j}\Psi^{(\ell)}(\xi).
\end{align}

By \cite[Lemma 2.1]{gly13}, we have the following lemma.
\begin{lemma}\label{gj1}
Let $\varphi$ be as above and $j,\ell\in\mathbb{Z}_+$ with $j+\ell\in2 \mathbb{Z}_+$. Then there exists a positive constant $C$ such that, for any $t\in(0,\infty)$, the kernel $K_{\Psi_j^{(\ell)}\left(t\sqrt{L}\right)}$
satisfies
\begin{align*}
\mathrm{supp}\,\left(K_{\Psi_j^{(\ell)}\left(t\sqrt{L}\right)}\right)\subset\left\{(x,y)\in\mathcal{X}\times\mathcal{X}:\
d(x,y)\leq t\right\}
\end{align*}
and
\begin{align*}
K_{\Psi_j^{(\ell)}\left(t\sqrt{L}\right)}(x,y)
\leq
\frac{C}{V(x,t)}
\end{align*}
for any $x,y\in\mathcal{X}.$
\end{lemma}
The following lemma is exactly \cite[Lemma 2.2]{gly13}.
\begin{lemma}\label{SL}
Let $p \in (1,\infty)$ and $j,\ell \in \mathbb{N}$ with $j+\ell \in 2\mathbb{N}$. Then there exists a positive constant $C$ such that, for any $f \in L^p$,
\begin{align*}
\left\| \mathcal{A}\left(\Psi_j^{(\ell)}\left(t\sqrt{L}\right)(f)\right)\right\|_{L^{p}} \leq C\|f\|_{L^{p}},
\end{align*}
where, for any measurable function $F \in \mathcal{X}\times(0,\infty)$ and $x\in\mathcal{X},$
\begin{align*}
\mathcal{A}(F)(x):=\left[\int_0^\infty\int_{B(x,t)}|F(y,t)|^2
\,\frac{d\mu(y)\,dt}{V(x,t)t}\right]^{\frac{1}{2}}.
\end{align*}
\end{lemma}
Let  $M,j,\ell \in \mathbb{N}$ with $j+\ell \in 2\mathbb{N}.$
For any $F \in t^2(\mathcal{X}\times(0,1])$ and $x \in
\mathcal{X}$,  let
\begin{align*}
\Pi_{M,L}(F)(x):=\int_0^1 (t^2 L)^M \Psi_{j}^{(\ell)}\left(t\sqrt{L}\right)F(x,t) \,\frac{dt}{t}.
\end{align*}
\begin{lemma}\label{gj2-2}
Let $X$ be a $\mathrm{BQBF}$ space, $ M \in \mathbb{N},$ and $p \in (1,\infty)$. Then
\begin{enumerate}
\item[\rm{(i)}] $\Pi_{M,L}$ is bounded from $t^2$ to $L^2$;
\item [\rm{(ii)}] For any $(t_X, \infty)$-atom a, $\Pi_{M,L}(a)$ is a local $(X,M,p)_L$-atom up to a harmless positive constant multiple.
\end{enumerate}
\end{lemma}
\begin{proof}
We first prove (i). By \eqref{d2}, we obtain, for any $(y,t)\in \mathcal{X}\times(0,\infty),$
\begin{align}\label{3.10}
\int_{B(y,t)}\,\frac{d\mu(x)}{V(x,t)}\sim\int_{B(y,t)}\,\frac{d\mu(x)}{V(y,t)}=1.
\end{align}
From this, Fubini's theorem, the assumption that $L$ is a self-adjoint operator, H\"older's inequality, and Lemma \ref{SL}, we deduce that, for any $g \in L^2$,
\begin{align*}
&\left|\int_{\mathcal{X}} \Pi_{M,L}(F)(z)\overline{g(z)} \, d\mu(z)\right|\\
&\quad\leq
\int_{\mathcal{X}}\int_{0}^{1}\left|(t^2L)^M\Psi_{j}^{(\ell)}\left(t\sqrt{L}\right)(F(\cdot, t))(z)g(z)\,\frac{dt}{t}\right|\,d\mu(z)\\
&\quad\sim
\int_{0}^{1}\int_{\mathcal{X}}\int_{B(z,t)}\left|(t^2L)^M\Psi_{j}^{(\ell)}\left(t\sqrt{L}\right)(F(\cdot, t))(z)g(z)\right|\, \frac{d\mu(x)}{V(x,t)}\,d\mu(z)\, \frac{dt}{t}\\
&\quad= 
\int_{\mathcal{X}}\int_{0}^{1}\int_{B(x,t)}\left|F(z,t)\Psi_{2M+j}^{(\ell)}\left(t\sqrt{L}\right)(g)(z)\right|\, \frac{d\mu(z)\,dt}{V(x,t)t}\,d\mu(x)\\
&\quad\leq
\int_{\mathcal{X}}|\mathcal{A}_{loc}(F)(x)|\left| \mathcal{A}_{\rm{loc}}\left(\Psi_{2M+j}^{(\ell)}\left(t\sqrt{L}\right)g\right)(x)\right|\, d\mu(x)\\
&\quad\leq
\|F\|_{t^2}	\left\| \mathcal{A}_{\rm{loc}}\left(\Psi_{2M+j}^{(\ell)}\left(t\sqrt{L}\right)g\right)\right\|_{L^{2}}
\lesssim \|F\|_{t^2}\|g\|_{L^2}.
\end{align*}
By the duality of $L^2$ space, we conclude that $\Pi_{M,L}$ is bounded from $t^2$ to $L^2$. 

Now, we prove (ii). Let $a$ be a $(t_X,\infty)$-atom, associated with the ball $B$. Since $\mathrm{supp}\,(a) \subset t(B)$, 
it follows that $\mathrm{supp}\,(\Pi_{M,L}(a)) \subset B$. For any $x \in \mathcal{X}$, let 
\begin{align*}
b(x):=\int_{0}^{1}t^{2M}\Psi_{j}^{(\ell)}\left(t\sqrt{L}\right)a(\cdot,t)\,\frac{dt}{t}.
\end{align*}
Then we have $\Pi_{M,L}(a) = L^M(b)$. To show that $\Pi_{M,L}(a)$ is a local $(X,M,p)_L$-atom up to a harmless positive constant multiple, it remains to prove that, for any $k \in \{0,\ldots, M\}$, 
\begin{align}\label{1848}
\left\|(r_{B}^2L)^k(b)\right\|_{L^p} \lesssim r_{B}^{2M}[\mu(B)]^{\frac{1}{p}}\left\|\mathbf{1}_{B}\right\|_{X}^{-1}.
\end{align}
By Fubini's theorem, the assumption that $L$ is self-adjoint, and the functional calculus, we find that, for any $g \in L^{p'}$ ,
\begin{align*}
&\int_{\mathcal{X}} \left( r_{B}^2 L \right)^k b(z) \overline{g(z)} \, d\mu(z)\nonumber\\
&\quad= r_{B}^{2k} \int_0^1 \int_{\mathcal{X}} t^{2M-2k} \left( t^2 L \right)^{k} \Psi_{j}^{(\ell)}\left( t\sqrt{L} \right) \bigl( a(\cdot, t) \bigr)(z) \; \overline{g(z)} \; \frac{d\mu(z) \, dt}{t}\nonumber\\
&\quad= r_{B}^{2k} \int_0^1 \int_{\mathcal{X}} t^{2M-2k} a(z,t) \; \Psi_{2k+j}^{(\ell)}\left( t\sqrt{L} \right)(g)(z) \,\frac{d\mu(z) \, dt}{t}.
\quad\end{align*}
From this, Lemma \ref{SL}, \eqref{3.10}, the assumption that $a$ is a $(t_X, \infty)$-atom, H\"older's inequality, and Fubini's theorem, we deduce that
\begin{align*}
&\Bigl| \int_{\mathcal{X}} \left( r_{B}^2 L \right)^k b(z)  \overline{g(z)} \, d\mu(z) \Bigr|\\
&\quad\leq r_{B}^{2M} \int_0^1 \int_{\mathcal{X}} \left|a(z,t) \Psi_{2k+j}^{(\ell)}\left( t\sqrt{L} \right)(g)(z) \right| \, \frac{d\mu(z)\,dt}{t}\\
&\quad\sim r_{B}^{2M} \int_0^1 \int_{\mathcal{X}} \int_{B(z,t)}\left|a(z,t) \Psi_{2k+j}^{(\ell)}\left( t\sqrt{L} \right)(g)(z)\right| \,\frac{d\mu(x)\,d\mu(z) \, dt}{V(x,t)t}\\
&\quad= r_{B}^{2M} \int_{\mathcal{X}} \int_{0}^{1}\int_{B(x,t)}\left|a(z,t) \Psi_{2k+j}^{(\ell)}\left( t\sqrt{L} \right)(g)(z)\right| \, \frac{d\mu(z) \, dt}{V(x,t)t}d\mu(x)\\
&\quad\leq r_{B}^{2M} \int_{\mathcal{X}}\left|\mathcal{A}_{loc}(a)(x)\right|\left|\mathcal{A}_\mathrm{loc}\left(\Psi_{2k+j}^{(\ell)}\left( t\sqrt{L} \right)(g)\right)(x)\right| \,d\mu(x)\\
&\quad\leq r_{B}^{2M} \left\|\mathcal{A}_{loc}(a)\right\|_{L^p}
\left\|\mathcal{A}_\mathrm{loc}\left(\Psi_{2k+j}^{(\ell)}\left( t\sqrt{L} \right)(g)\right)\right\|_{L^{p'}}\\
&\quad\lesssim r_{B}^{2M}\left\|a\right\|_{t^p}\left\|g\right\|_{L^{p'}}
\leq r_{B}^{2M}[\mu(B)]^{\frac{1}{p}}\left\|\mathbf{1}_{B}\right\|_{X}^{-1}\left\|g\right\|_{L^{p'}}.
\end{align*}
This implies that, for any $k \in \{0,\ldots , M\}$, \eqref{1848} holds true, which completes the proof of Lemma \ref{gj2-2}.
\end{proof}

The following proposition is precisely \cite[Proposition 4.8(i)]{syy22jga}.
\begin{proposition}\label{prfs}
Let $X$ be a $\mathrm{BQBF}$ space satisfying Assumption \ref{vector1} for some $p_-\in(0,\infty)$.
Then, for any given $p\in(0,\infty)$ and $t\in(\max\{1,\frac{p}{p_-}\},\infty)$,
there exists a positive constant $C$ such that, for any $\tau\in[1,\infty)$, any $\{\lambda_j\}_{j\in\mathbb{N}}\in
[0,\infty)$, and any sequence $\{B_j\}_{j\in\mathbb{N}}$ of balls,
\begin{align*}
\left\|\sum_{j\in\mathbb{N}}\lambda_j\mathbf{1}_{\tau B_j}
\right\|_{X^{\frac{1}{p}}}\le C\tau^{tn }\left\|\sum_{j\in\mathbb{N}}
\lambda_j\mathbf{1}_{B_j}\right\|_{X^{\frac{1}{p}}}.
\end{align*}
\end{proposition}
Now, we  prove Theorem \ref{th1-1}.
\begin{proof}[Proof of Theorem \ref{th1-1}]
Let $f\in h_{X,L}\cap L^2$ 
and let  $\varphi$ be the same in \eqref{psi}. By \cite[Lemma 3.9]{dhmmy13}, we have
\begin{align}\label{yindu}
f&=\sum_{\ell=0}^{2(M+1)} \sum_{i=0}^{\lfloor \frac{2(M+1)-\ell}{2} \rfloor } c_1(M,\ell,i) \int_0^1 (t^2 L)^M \Psi_{2M+2-2i-\ell}^{(\ell)}\left(t\sqrt{L}\right) t^2 L e^{-t^2 L} f \frac{dt}{t}\notag\\
&\quad-\sum_{m=0}^{2M+1} \sum_{\ell=0}^{2M+1-m}  \sum_{i=0}^{\lfloor \frac{2M+1-m-\ell}{2}\rfloor} c_2(M,m,\ell,i) \Psi_{4M+2-2m-2i-\ell}^{(\ell)}(\sqrt{L}) e^{-L} f\notag\\
&=:f_1-f_2\ \ \text{in}\ \ L^2.
\end{align}
We first deal with $f_1.$
For any $(y,t)\in\mathcal{X}\times(0,1)$, let $F(y,t):=\mathbf{1}_{(0,1)}(t)t^2Le^{-t^2L}f(y)$. From \cite[(3.13)]{lyyy26cms}, we deduce that $S_{L,\mathrm{loc}}$ is bounded on $L^2.$ Then $F\in t_X\cap t^2.$ By
this and Theorem \ref{local-tent}, we find that there exists a sequence $\{\lambda_j\}_{j\in\mathbb{N}}$ and a sequence
$\{a_j\}_{j\in\mathbb{N}}$ of $(t_X,\infty)$-atoms associated, respectively, with the balls $\{B_j\}_{j\in\mathbb{N}}$ such that 
\begin{align}\label{yindu2}
F=\sum_{j\in\mathbb{N}}\lambda_ja_j\ \ \text{in}\ \ t^2
\end{align}
and
\begin{align*}
\left\|\left\{\sum_{j\in\mathbb{N}}
\left[\frac{\lambda_j}{\|\mathbf{1}_{B_j}\|_{X  }}\right]^{s}\mathbf{1}_{B_j}
\right\}^{\frac{1}{s}}\right\|_X\lesssim\|F\|_{t_X}=\|S_{L,\mathrm{loc}}(f)\|_X.
\end{align*}
Using \eqref{yindu}, \eqref{yindu2}, and
Lemma \ref{gj2-2}(i), we obtain
\begin{align*}
f_1
&=\sum_{\ell=0}^{2(M+1)} \sum_{i=0}^{\lfloor \frac{2(M+1)-\ell}{2} \rfloor } c_1(M,\ell,i) \int_0^1 (t^2 L)^M \Psi_{2M+2-2i-\ell}^{(\ell)}\left(t\sqrt{L}\right)\left(\sum_{j\in\mathbb{N}}\lambda_ja_j\right)\,\frac{dt}{t}\\
&=\sum_{\ell=0}^{2(M+1)} \sum_{i=0}^{\lfloor \frac{2(M+1)-\ell}{2} \rfloor }\sum_{j\in\mathbb{N}}\lambda_{\ell,i,j}a_{\ell,i,j}\ \ \text{in}\ \ L^2,
\end{align*}
where, 
\begin{align*}
\lambda_{\ell,i,j}:=c_1(M,\ell,i)\lambda_{j}
\end{align*}
and
\begin{align*}
a_{\ell,i,j}:=\int_0^1 (t^2 L)^M \Psi_{2M+2-2i-\ell}^{(\ell)}\left(t\sqrt{L}\right)a_j \,\frac{dt}{t}.
\end{align*}
By Lemma \ref{gj2-2}(ii), we find that, for any $\ell, i, j \in \mathbb{Z}_+$, $a_{\ell, i, j}$ is a local $(X,M,p)_L$-atom , associated with $B_j$. This obtains the local atom $(X,M,p)_L$-representation of $f_1.$

Now, we deal with $f_2.$ 
By \eqref{yindu}, we have
\begin{align*}
f_2 = \sum_{m=0}^{2M+1} \sum_{\ell=0}^{2M+1-m}  \sum_{i=0}^{\lfloor \frac{2M+1-m-\ell}{2}\rfloor}
\sum_{j\in \mathbb{N}}a_{m,\ell,i,j}\lambda_{m,\ell,i,j}
\ \ \text{in} \ L^2,
\end{align*}
where,
\begin{align*}
a_{m,\ell,i,j}:=\mu(Q_j)\|\mathbf{1}_{Q_j}e^{-L}f\|_{L^1}^{-1}\|\mathbf{1}_{2B_j}\|_{X}^{-1}\Psi_{4M+2-2m-2i-\ell}^{(\ell)}\left(\sqrt{L}\right) \left(\mathbf{1}_{Q_j}e^{-L} f\right)
\end{align*}
and
\begin{align*}
\lambda_{m,\ell,i,j}:=c_2(M,m,\ell,i)
\mu(Q_j)^{-1}\|\mathbf{1}_{Q_j}e^{-L}f\|_{L^1}\|\mathbf{1}_{2B_j}\|_{X}.
\end{align*}
By the fact that  $Q_j$ is contained in a ball with radius equal to $1$ and  Lemma \ref{gj1},  we   find that
$\mathrm{supp}\,(a_{m,\ell,i,j})\subset 2B_j$,  
\begin{align*}
\left\|\Psi_{4M+2-2m-2i-\ell}^{(\ell)}(\sqrt{L}) (\mathbf{1}_{Q_j}e^{-L} f)\right\|_{L^p}^p
&\lesssim
\mu(2B_j)\|\mathbf{1}_{Q_j}e^{-L}f\|_{L^1}^p\sup_{x\in 2B_j}\frac{1}{V(x,1)^p}\\
&\lesssim
\mu(2B_j)\mu(Q_j)^{-p}\|\mathbf{1}_{Q_j}e^{-L}f\|_{L^1}^p,
\end{align*}
and
\begin{align*}
\|a_{m,\ell,i,j}\|_{L^p}
\lesssim
[\mu(2B_j)]^{\frac{1}{p}}\|\mathbf{1}_{2B_j}\|_{X}^{-1}.
\end{align*}
Thus, $a_{m,\ell,i,j}$ is a local $(X,M,p)_{L}$-atom.
Moreover, by the fact that  $\mathcal{Q}=\{Q_j\}_{j\in\mathbb{N}}$ is disjoint, Proposition \ref{prfs},
H\"older's inequality, and Theorem \ref{dec-HI},
we conclude that
\begin{align*}
&\left\|\left\{\sum_{m=0}^{2M+1} \sum_{\ell=0}^{2M+1-m}  \sum_{i=0}^{\lfloor \frac{2M+1-m-\ell}{2}\rfloor}
\sum_{j\in \mathbb{N}}
\left[\frac{\lambda_{m,\ell,i,j}}{\|\mathbf{1}_{2B_j}\|_{X  }}\right]^{s}\mathbf{1}_{2B_j}
\right\}^{\frac{1}{s}}\right\|_X\\
&\quad\lesssim
\left\|\left\{\sum_{j\in\mathbb{N}}
\left[\mu(Q_j)^{-1}\|\mathbf{1}_{Q_j}e^{-L}f\|_{L^1}\right]^{s}\mathbf{1}_{2B_j}
\right\}^{\frac{1}{s}}\right\|_X
\lesssim
\left\|\left\{\sum_{j\in\mathbb{N}}
\left[\mu(Q_j)^{-1}\|\mathbf{1}_{Q_j}e^{-L}f\|_{L^1}\right]^{s}\mathbf{1}_{Q_j}
\right\}^{\frac{1}{s}}\right\|_X\\
&\quad=
\left\|\sum_{j\in\mathbb{N}}
\mu(Q_j)^{-1}\|\mathbf{1}_{Q_j}e^{-L}f\|_{L^1}\mathbf{1}_{Q_j}
\right\|_X
\lesssim
\|S_I(e^{-L}f)\|_{X}.
\end{align*}
This establishes the local atom $(X,M,p)_L$-representation of $f_2,$
which completes the proof of Theorem \ref{th1-1}.
\end{proof}
In order to prove Theorem \ref{recon}, we need the following  lemmas.
\begin{lemma}\label{smol}
Let $X$ be a $\mathrm{BQBF}$ space, $s\in(0,\infty),$ $p\in(\max\{1,s\},\infty),$
$\epsilon\in (\frac{n}{s},\infty)$, and $ M\in(\frac{n}{2}(\frac{1}{s}-\frac{1}{p}),\infty)\cap\mathbb{N}$. Then 
there exist positive constants $C$ and $\eta\in (n(\frac{1}{s}-\frac{1}{p}),\infty)$ such that, 
for any local $(X,M,p,\epsilon)_L$-molecule $m$, associated with the ball $B$, and  any $i\in\mathbb{Z}_+$,
\begin{align}\label{guji2}
\|S_{L,\mathrm{loc}}(m)\|_{L^p(U_i(B))}
\leq
C2^{-\eta i}[\mu(B)]^{\frac{1}{p}}\left\|\mathbf{1}_B\right\|_{X}^{-1}.
\end{align}
\end{lemma}
\begin{proof}
By \cite[Theorem 2.13]{bckyy13}, we find that $S_{L,\mathrm{loc}}$ is bounded on $L^p$.
From this, we infer  that \eqref{guji2} holds  with $i\in\{0,1,2,3\}.$ Let $i\in\mathbb{N}\cap [4,\infty).$

Case 1) $r_B\in (0,1).$
Let $\delta\in(0,1)$ be determined later. Then we have
\begin{align*}
\notag&\|S_{L,\mathrm{loc}}(m)\|_{L^p(U_i(B))}^p\\ 
&\quad\lesssim
\int_{U_i(B)}\left[\int_0^{2^{i\delta}r_B}\int_{B(x,t)}
\left|t^2Le^{-t^2L}(m)(y)\right|^2\,\frac{d\mu(y)\,dt}{V(x,t)t}\right]^{\frac{p}{2}}\,d\mu(x)\\ \notag
&\quad\quad+\int_{U_i(B)}\int_{2^{i\delta}r_B}^\infty\int_{B(x,t)}\cdots=:\mathrm{I}+\mathrm{II}.
\end{align*}
For any given $k\in\mathbb{N}$ and any $g\in L^2,$ let
\begin{align*}
S_{L,k}(g)(x):=\left[\int_0^\infty\int_{B(x,t)}\left|(t^2L)^{k}e^{-t^2L}g(y)\right|^2\,\frac{d\mu(y)\,dt}{V(x,t)t}\right]^{\frac{1}{2}}.
\end{align*}
By \cite[Theorem 2.13]{bckyy13}, we find that $S_{L,k}$ is bounded  on $L^p.$
Since  $m$ is a local $(X,M,p,\epsilon)_L$-molecule, we deduce that
there exists a function $b\in \mathrm{Dom}\,(L^M)$ such that $m=L^M(b)$ and \eqref{pmol2} holds. From this, we infer that
\begin{align}\label{jiegong}
\mathrm{II}&=
\int_{U_i(B)}\left[\int_{2^{i\delta}r_B}^{\infty}\int_{B(x,t)}
t^{-4M}\left|\left(t^2L\right)^{M+1}e^{-t^2L}(b)(y)\right|^2\,\frac{d\mu(y)\,dt}{V(x,t)t}\right]^\frac{p}{2}\,d\mu(x)\notag\\
&\lesssim
(2^{i\delta}r_B)^{-2pM}\|S_{L,M+1}(b)\|_{L^p}^p
\lesssim
(2^{i\delta}r_B)^{-2pM}\|b\|_{L^p}^p
\lesssim
2^{-2pM\delta i}\mu(B)\|\mathbf{1}_{B}\|_{X}^{-p}.
\end{align}
Next, we estimate $\mathrm{I}.$
Let
\begin{align*}
S_i(B):=\left(2^{i+1}B\right)\setminus\left(2^{i-2}B\right)\ \text{and}\ \widetilde{S}_i(B):
=\left(2^{i+2}B\right)\setminus\left(2^{i-3}B\right).
\end{align*}
Then
\begin{align}\label{3888}
m=m\mathbf{1}_{\widetilde{S}_i(B)}+
m\mathbf{1}_{[\widetilde{S}_i(B)]^\complement}=:m_1+m_2
\end{align}
and
\begin{align}\label{heibang1}
\mathrm{I}
&\lesssim
\int_{U_i(B)}\left[\int_0^{2^{i\delta}r_B}\int_{B(x,t)}
\left|t^2Le^{-t^2L}(m_1)(y)
\right|^2\,\frac{d\mu(y)\,dt}{V(x,t)t}\right]^{\frac{p}{2}}\,d\mu(x)\notag\\ 
&\quad+\int_{U_i(B)}\left[\int_0^{2^{i\delta}r_B}\int_{B(x,t)}\left|t^2Le^{-t^2L}(m_2)(y)\right|^2\,\frac{d\mu(y)\,dt}{V(x,t)t}\right]^{\frac{p}{2}}\,d\mu(x)=:\rm{I}_1+\rm{I}_2.
\end{align}
Since $S_{L,1}$ is bounded on $L^p$, it follows that
\begin{align}\label{heibang2}
\mathrm{I}_1
\leq
\|S_{L,1}(m\mathbf{1}_{\widetilde{S}_i(B)})\|_{L^p}^p
\lesssim
\|m\|_{L^p(\widetilde{S}_i(B))}^p
\lesssim
2^{-i(p\epsilon-n)}\mu(B)\|\mathbf{1}_{B}\|_{X}^{-p}.
\end{align}
Observe that, if $t\in(0,2^{i\delta}r_B)$ and $x\in U_i(B)$, then $B(x,t)\subset S_i(B)$ and  
$\mathrm{dist}\,(S_i(B),[\widetilde{S}_i(B)]^\complement)\sim 2^ir_B.$  Using this and \eqref{lzy}, we conclude that,
for any $t\in(0,2^{i\delta}r_B)$, $x\in U_i(B)$, and $y\in B(x,t),$
\begin{align*}
\left|t^2Le^{-t^2L}(m_2)(y)\right|
&\lesssim
\int_{\mathcal{X}}\frac{1}{V(y,t)}\exp\left\{-c\frac{d(y,z)^2}{t^2}\right\}|m_2(z)|\,d\mu(z)\\
&\lesssim
\exp\left\{-\frac{c}{2}\left(\frac{2^ir_B}{t}\right)^2\right\}
\frac{1}{V(y,t)}\left[\int_{\mathcal{X}}\exp\left\{-\frac{cp'd(y,z)^2}{2t^2}\right\}\,d\mu(z)\right]^{\frac{1}{p'}}\|m\|_{L^p}\\
&\lesssim
\exp\left\{-\frac{c}{2}\left(\frac{2^ir_B}{t}\right)^2\right\}V(y,t)^{-\frac{1}{p}}\|m\|_{L^p}.
\end{align*}
This,  combined with \eqref{d2},  implies that, for any  $x\in U_i(B),$
\begin{align}\label{yinengjing1}
&\int_0^{2^{i\delta}r_B}\int_{B(x,t)}\left|t^2Le^{-t^2L}(m_2)(y)\right|^2\,\frac{d\mu(y)\,dt}{V(x,t)t}\notag\\
&\quad\lesssim
\|m\|_{L^p}^2\int_0^{2^{i\delta}r_B}\int_{B(x,t)}\exp\left\{-c\left(\frac{2^ir_B}{t}\right)^2\right\}V(y,t)^{-\frac{2}{p}}\,\frac{d\mu(y)\,dt}{V(x,t)t}\notag\\
&\quad\sim
\|m\|_{L^p}^2\int_0^{2^{i\delta}r_B}\exp\left\{-c\left(\frac{2^ir_B}{t}\right)^2\right\}V(x,t)^{-\frac{2}{p}}\,\frac{dt}{t}.
\end{align}
Notice that , for any $t\in (0,2^{i\delta}r_B)$ and $x\in U_i(B)$, 
\begin{align*}
\mu(2^iB)
\lesssim
\left(\frac{2^ir_B}{t}\right)^nV(x_B,t)
\lesssim
\left(\frac{2^ir_B}{t}\right)^{2n}V(x,t).
\end{align*}
By this and \eqref{yinengjing1}, we have, for any $x\in U_i(B),$
\begin{align*}
&\int_0^{2^{i\delta}r_B}\int_{B(x,t)}\left|t^2Le^{-t^2L}(m_2)(y)\right|^2\,\frac{d\mu(y)\,dt}{V(x,t)t}\notag\\
&\quad\lesssim
\|m\|_{L^p}^2[\mu(2^iB)]^{-\frac{2}{p}}\int_0^{2^{i\delta}r_B}\exp\left\{-c\left(\frac{2^ir_B}{t}\right)^2\right\}\left(\frac{2^ir_B}{t}\right)^{\frac{4n}{p}}\,\frac{dt}{t}\\
&\quad\lesssim
\|m\|_{L^p}^2[\mu(2^iB)]^{-\frac{2}{p}}
\int_0^{2^{i\delta}r_B}\left(\frac{t}{2^ir_B}\right)^{N-\frac{4n}{p}}\,\frac{dt}{t}
\lesssim
2^{-i(1-\delta)(N-\frac{4n}{p})}\|m\|_{L^p}^2[\mu(2^iB)]^{-\frac{2}{p}}.
\end{align*}
From this,  we deduce that
\begin{align*}
\mathrm{I}_2
\lesssim
2^{-i(1-\delta)(\frac{pN}{2}-2n)}\|m\|_{L^p}^p
\lesssim
2^{-i(1-\delta)(\frac{pN}{2}-2n)}
\mu(B)\|\mathbf{1}_{B}\|_{X}^{-p}.
\end{align*}
Using this, \eqref{heibang2}, and \eqref{heibang1}, we conclude   that
\begin{align}\label{heibang}
\mathrm{I}\lesssim
(2^{-i(p\epsilon-n)}+2^{-i(1-\delta)(\frac{pN}{2}-2n)})
\mu(B)\|\mathbf{1}_{B}\|_{X}^{-p}.
\end{align}
This,  combined with \eqref{jiegong}, implies  that 
\eqref{guji2} holds with
\begin{align*}
\eta:=\min\left\{2M\delta,\epsilon-\frac{n}{p},(1-\delta)\left(\frac{N}{2}-\frac{2n}{p}\right)\right\}.
\end{align*}
By the assumptions that $\epsilon\in(\frac{n}{s},\infty)$ and $M\in\mathbb{N}\cap(\frac{n}{2}(\frac{1}{s}
-\frac{1}{p}),\infty)$, we  choose $\delta\in(0,1)$ and $N\in\mathbb{N}$ large enough such that
$\eta\in (n(\frac{1}{s}-\frac{1}{p}),\infty)$.
This proves \eqref{guji2} 
in this case.

Case 2) $r_B\in [1,\infty).$
In this case, the proof of \eqref{guji2} is the same as $\mathrm{I}$ with $2^{i\delta}r_B$ replaced by $1;$
we omit the details here. This finishes the proof of Lemma \ref{smol}.
\end{proof}
\begin{lemma}\label{smol2}
Let $X$ be a $\mathrm{BQBF}$ space,
 $s\in(0,\infty),$
$p\in(\max\{1,s\},\infty)$,
$\epsilon\in (\frac{n}{s},\infty),$ and $ M\in(\frac{n}{2}(\frac{1}{s}-\frac{1}{p}),\infty)\cap\mathbb{N}$.  Then 
there exist positive constants $C$ and $\eta\in (n(\frac{1}{s}-\frac{1}{p}),\infty)$ such that, 
for any local $(X,M,p,\epsilon)_L$-molecule $m$, associated with the ball $B$ and  any $i\in\mathbb{Z}_+$,
\begin{align}\label{1030}
\|S_{I}(e^{-L}m)\|_{L^p(U_i(B))}
\leq 
C2^{-\eta i}[\mu(B)]^{\frac{1}{p}}\left\|\mathbf{1}_B\right\|_{X}^{-1}.
\end{align}
\end{lemma}
\begin{proof}
We first show  that $S_{I}(e^{-L})$ is bounded on $L^p$. Let $f\in L^2\cap L^p.$
By \eqref{lzy}, we have, for any $t\in  (0,1]$ and $x,y\in\mathcal{X}$ with $d(x,y)<t\leq1$,
\begin{align}\label{xiyang}
|e^{-L}f(y)|
&\lesssim
\int_{\mathcal{X}}\frac{1}{V(y,1)}e^{-cd(y,z)^2}|f(z)|\,d\mu(z)\notag\\
&\sim
\frac{1}{V(y,1)}\left[\int_{B(y,1)}|f(z)|\,d\mu(z)
+\sum_{k\in\mathbb{N}}\int_{2^{k-1}<d(y,z)\leq 2^k}e^{-c4^k}|f(z)|\,d\mu(z)\right]\notag\\
&\lesssim
\mathcal{M}(f)(x)+\sum_{k\in\mathbb{N}}\frac{V(y,2^k)}{V(y,1)}e^{-c4^k}\mathcal{M}(f)(x)
\lesssim
\mathcal{M}(f)(x).
\end{align}
 Similarly, we also have,
for any  $t\in  (1,\infty)$ and $x,y\in\mathcal{X}$ with $d(x,y)<t$,
\begin{align*}
|e^{-L}f(y)|
&\lesssim
\frac{1}{V(y,1)}\left[\int_{B(y,t)}|f(z)|\,d\mu(z)+\sum_{k\in\mathbb{N}}\int_{2^{k-1}t<d(y,z)\leq 2^kt}e^{-c(2^kt)^2}|f(z)|\,d\mu(z)\right]\\
&\lesssim
\frac{V(y,t)}{V(y,1)}\mathcal{M}(f)(x)
+\sum_{k\in\mathbb{N}}\frac{V(y,2^kt)}{V(y,1)}e^{-c(2^kt)^2}\mathcal{M}(f)(x)\\
&\lesssim
(t^n+1)\mathcal{M}(f)(x)\lesssim
t^n\mathcal{M}(f)(x).
\end{align*}
From this and \eqref{xiyang}, it follows that, for any $t\in (0,\infty)$ and  $x,y\in\mathcal{X}$ with $d(x,y)<t$
\begin{align}\label{xiyang2}
|e^{-L}f(y)|\lesssim
\max\{t^n,1\}\mathcal{M}(f)(x).
\end{align}
Thus, for any $x\in\mathcal{X},$
\begin{align*}
S_I(e^{-L}f)(x)
\lesssim
\mathcal{M}(f)(x)\int_0^\infty t^3e^{-2t^2}\max\{t^n,1\}\,dt
\lesssim
\mathcal{M}(f)(x).
\end{align*}
This  implies that $S_{I}(e^{-L})$ is bounded on $L^p.$

Now, we prove \eqref{1030}. If
$i\in\{0,1,2,3\}$, then \eqref{1030} follows from the $L^p$-boundedness of $S_{I}(e^{-L})$.
In what follows,  let $i\in \mathbb{N}\cap [4,\infty).$
We first consider the case $r_B\in [1,\infty).$ 
Then we have
\begin{align*}
\|S_{I}(e^{-L}m)\|_{L^p(U_i(B))}^p
&=\int_{U_i(B)}\left[\int_0^{2^{i-2}r_B}\int_{B(x,t)}\left|t^2e^{-t^2}(e^{-L}m)(y)\right|^2\,\frac{d\mu(y)\,dt}{V(x,t)t}\right]^{\frac{p}{2}}\,d\mu(x)\\ 
&\quad\quad+\int_{U_i(B)}\int_{2^{i-2}r_B}^\infty\int_{B(x,t)}\cdots=:\mathrm{I}+\mathrm{II}.
\end{align*}
By \eqref{xiyang2}, we obtain
\begin{align*}
\mathrm{II}
\lesssim
\int_{\mathcal{X}}\left[\mathcal{M}(m)(x)\right]^p\int_{2^{i-2}r_B}^\infty t^{3+2n}e^{-2t^2}\,dt\,d\mu(x)
\lesssim
\|m\|_{L^p}^p(2^ir_B)^{-N}
\lesssim
2^{-iN}\mu(B)\|\mathbf{1}_{B}\|_X^{p}.
\end{align*}
Let $m_1$ and $m_2$ be as in \eqref{3888}.
Let $k\in\{1,2\}$ and 
$\mathrm{I}_k$ be as  $\mathrm{I}$ with $m$ replace by $m_k.$ 
Then we only need to deal with $\mathrm{I}_1$ and $\mathrm{I}_2$.
By the $L^p$-boundedness of 
$S_{I}(e^{-L})$ and  \eqref{pmol1}, we find that
\begin{align*}
\mathrm{I}_1
\leq
\|S_{I}(e^{-L}m_1)\|_{L^p}^p
\lesssim
\|m\|_{L^p(\widetilde{S}_i(B))}^p
\lesssim
2^{-i(p\epsilon-n)}\mu(B)\|\mathbf{1}_{B}\|_{X}^{-p}.
\end{align*}
Notice that if $t\in(0,2^{i-2}r_B)$ and $x\in U_i(B)$, then $B(x,t)\subset S_i(B).$ Moreover,
$\mathrm{dist}\,(S_i(B),[\widetilde{S}_i(B)]^\complement)\sim 2^ir_B.$
By  this and \eqref{lzy}, we conclude that,
for any $t\in(0,2^{i-2}r_B)$, $x\in U_i(B)$, and $y\in B(x,t),$
\begin{align*}
|e^{-L}m_2(y)|
&\lesssim
e^{-\frac{c(2^ir_B)^2}{2}}\int_{\mathcal{X}}\frac{1}{V(y,1)}e^{-\frac{cd(y,z)^2}{2}}|m_2(z)|\,d\mu(z)\\
&\lesssim
e^{-\frac{c(2^ir_B)^2}{2}}\|m\|_{L^p}\frac{1}{V(y,1)}\left[\int_{\mathcal{X}}e^{-\frac{p'cd(y,z)^2}{2}}\,d\mu(z)\right]^{\frac{1}{p'}}
\lesssim
e^{-\frac{c(2^ir_B)^2}{2}}\|m\|_{L^p}V(y,1)^{-\frac{1}{p}}.
\end{align*}
Using the fact that $2^ir_B>1$, \eqref{d1},
and \eqref{d2}, we obtain, for any $y\in S_i(B),$
\begin{align*}
\mu(2^iB)
\lesssim
(2^ir_B)^nV(x_B,1)
\lesssim
(2^ir_B)^{2n}V(y,1).
\end{align*}
This implies that,
for any $t\in(0,2^{i-2}r_B)$, $x\in U_i(B)$, and $y\in B(x,t),$
\begin{align*}
|e^{-L}m_2(y)|
&\lesssim
e^{-\frac{c(2^ir_B)^2}{2}}\|m\|_{L^p}(2^ir_B)^{\frac{2n}{p}}[\mu(2^iB)]^{-\frac{1}{p}}
\lesssim
(2^ir_B)^{-N+\frac{2n}{p}}\|m\|_{L^p}[\mu(2^iB)]^{-\frac{1}{p}}\\
&\lesssim
2^{-i(N-\frac{2n}{p})}\|m\|_{L^p}[\mu(2^iB)]^{-\frac{1}{p}},
\end{align*}
which further implies that
\begin{align*}
\mathrm{I}_2
&=
\int_{U_i(B)}\left[\int_0^{2^{i-2}r_B}\int_{B(x,t)}\left|t^2e^{-t^2}(e^{-L}m_2)(y)\right|^2\,\frac{d\mu(y)\,dt}{V(x,t)t}\right]^{\frac{p}{2}}\,d\mu(x)\\
&\lesssim
2^{-i(pN-2n)}\|m\|_{L^p}^p\int_0^\infty t^3e^{-2t^2}\,dt
\lesssim
2^{-i(pN-2n)}\mu(B)\|\mathbf{1}_{B}\|_{X}^{-p}.
\end{align*}
Combined all the above estimates,  we  conclude  that
\begin{align*}
\|S_{I}(e^{-L}m)\|_{L^p(U_i(B))}
\lesssim
2^{-i\eta}\mu(B)^{\frac{1}{p}}\|\mathbf{1}_{B}\|_{X}^{-1}
\end{align*}
with 
\begin{align*}
\eta:=\min\left\{\frac{N}{p},\epsilon-\frac{n}{p},N-\frac{2n}{p}\right\}.
\end{align*} 
When $N$ is sufficiently large, we have $\eta\in (n(\frac{1}{s}-\frac{1}{p}),\infty)$
This proves \eqref{1030}
in the case $r_B\in [1,\infty)$.  

The proof of \eqref{1030} in
the case $r_B\in (0,1)$ is similar; we omit the details here. This finishes the proof of  Lemma \ref{smol2}.
\end{proof}
The following proposition is precisely \cite[Proposition 3.19]{lyyy25jga}.
\begin{proposition}\label{pharve}
Let $X$ be a $\mathrm{BQBF}$ space satisfying Assumption \ref{vector2} for some $s\in(0,\infty)$ and $q\in(s,\infty)$.
Assume that $p\in[q,\infty)$. Then there exists a positive constant $C$ such that, for any $\delta\in(n(\frac{1}{s}-\frac{1}{p}),\infty)$,
$\{\lambda_j\}_{j\in\mathbb{N}}\subset (0,\infty)$, and
$\{a_j\}_{j\in\mathbb{N}}\subset L^p$ satisfying  $\|a_j\mathbf{1}_{U_i(B_j)}\|_{L^p}\le
2^{-i\delta}\lambda_j[\mu(B_j)]^{\frac{1}{p}}$
for any $j\in\mathbb{N}$ and $i\in\mathbb{Z}_+$,
\begin{equation*}
\left\|\sum_{j\in\mathbb{N}}\left|a_j\right|^{s}\right\|_{X^{\frac{1}{s}}}
\leq
C\left\|\sum_{j\in\mathbb{N}}
\left|\lambda_j\right|^{s}\mathbf{1}_{B_j}\right\|_{X^{\frac{1}{s}}}.
\end{equation*}
\end{proposition}
\begin{lemma}\label{beike}
Let $X$ be a $\mathrm{BQBF}$ space satisfying Assumption \ref{vector2} for some $s\in(0,1]$ and 
$q\in(s,\infty)$.
Let $T$ be a bounded sublinear operator on $L^2$, $f\in L^2$, and 
$p\in(1,\infty)\cap[q,\infty)$.
Assume that
there exists a sequence $\{\lambda_j\}_{j\in\mathbb{N}}\subset [0,\infty)$ and a sequence $\{m_j\}_{j\in\mathbb{N}}\subset L^2$
, associated with balls $\{B_j\}_{j\in\mathbb{N}}$, such that
$f=\sum_{j\in\mathbb{N}}\lambda_jm_j$ in $L^2.$
Assume further that 
there exist positive constants $C_1$ and $\eta\in (n(\frac{1}{s}-\frac{1}{p}),\infty)$, independent of $f$,
such that, for any $j\in\mathbb{N}$ and $i\in\mathbb{Z}_+,$ it   holds 
\begin{align}\label{T}
\|T(m_j)\|_{L^p(U_i(B_j))}
\leq 
C_12^{-\eta i}[\mu(B_j)]^{\frac{1}{p}}\left\|\mathbf{1}_{B_j}\right\|_{X}^{-1}.
\end{align}
Then  there exists a  positive  constant $C_2$, independent of $f$, such that 
\begin{align}\label{mb}
\|T(f)\|_X
\leq
C_2\left\|\left\{\sum_{j\in \mathbb{N}}\left[\frac{\lambda_{j}}{\|\mathbf{1}_{B_j}\|_X}\right]^s\mathbf{1}_{B_j}\right\}^{\frac{1}{s}}\right\|_X.
\end{align}
\end{lemma}
\begin{proof}
Since $T$ is bounded on $L^2$, $f=\sum_{j\in\mathbb{N}}\lambda_jm_j$ in $L^2$, and $s\in (0,1]$, it follows that
\begin{align*}
|T(f)|
\leq
\sum_{j\in\mathbb{N}}\lambda_j|T(m_j)|
\leq
\left\{\sum_{j\in\mathbb{N}}\left[\lambda_j|T(m_j)|\right]^s\right\}^{\frac{1}{s}}.
\end{align*}
From this,  \eqref{T}, and Proposition \ref{pharve}, we deduce that
\begin{align*}
\|T(f)\|_X^s
\leq
\left\|\sum_{j\in \mathbb{N}}\left[\lambda_{j}T(m_j)\right ]^s\right\|_{X^{\frac{1}{s}}} 
\lesssim 
\left\|\sum_{j\in \mathbb{N}}\left[\frac{\lambda_{j}}{\|\mathbf{1}_{B_j}\|_X}\right]^s\mathbf{1}_{B_j}\right\|_{X^{\frac{1}{s}}}.
\end{align*}
Thus, \eqref{mb} holds. This finishes the proof of Lemma \ref{beike}.
\end{proof}
Now, we use Lemmas  \ref{smol}, \ref{smol2}, and \ref{beike} to prove Theorem \ref{recon}.
\begin{proof}[Proof of Theorem \ref{recon}]
By the definition of $\widetilde{h}^{M,p,\epsilon}_{X,L,\mathrm{mol}}$, we find that there exist
a sequence  $\{\lambda_j\}_{j\in\mathbb{N}}\subset [0,\infty)$
and a sequence $\{m_j\}_{j\in\mathbb{N}}$ of local
$(X,M,p,\epsilon)_L$-molecules
such that 
\begin{align*}
f=\sum_{j\in\mathbb{N}}\lambda_jm_j\ \ \text{in}\ \ L^2
\end{align*}
and
\begin{align*}
\left\|\left\{\sum_{j\in\mathbb{N}}
\left[\frac{\lambda_j}{\|\mathbf{1}_{B_j}\|_{X}}\right]^{s}
\mathbf{1}_{B_j}\right\}^{\frac{1}{s}}\right\|_{X}
\lesssim
\|f\|_{\widetilde{h}^{M,p,\epsilon}_{X,L,\mathrm{mol}}}.
\end{align*}
Using this and applying Lemma \ref{beike} with $T:=S_{L,\mathrm{loc}}$
and $T:=S_{I}(e^{-L})$, we find that 
\begin{align*}
\|f\|_{h_{X,L}}
&=
\|S_{L,\mathrm{loc}}(f)\|_{X}
+
\|S_{I}(e^{-L}f)\|_{X}\\
&\lesssim
\left\|\left\{\sum_{j\in\mathbb{N}}\left[\frac{\lambda_j}{\|\mathbf{1}_{B_j}\|_{X}}\right]^s\mathbf{1}_{B_j}\right\}^{\frac{1}{s}}\right\|_X
\lesssim
\|f\|_{\widetilde{h}^{M,p,\epsilon}_{X,L,\mathrm{mol}}}.
\end{align*}
This finishes the proof of Theorem \ref{recon}.
\end{proof}
\subsection{Proof of Theorem \ref{hH}}\label{guanxi}
In this subsection, we prove Theorem \ref{hH}. We need the following lemmas.
\begin{lemma}\label{smo22}
Let $X$ be a $\mathrm{BQBF}$ space, $s\in(0,2),$
$\epsilon\in (\frac{n}{s},\infty)$, and $ M\in(\frac{n}{2}(\frac{1}{s}-\frac{1}{2}),\infty)\cap\mathbb{N}.$
Let $m$ be a local $(X,M,2,\epsilon)_L$-molecule, associated with the ball $B$. 
\begin{itemize}
\item[$\mathrm{(i)}$] Assume  that $\mathrm{inf}\,\sigma (L) \in (0,\infty)$.  Then  there exist positive constants $C_1$ and $\eta_1\in (n(\frac{1}{s}-\frac{1}{2}),\infty)$, independent of $m$,  such that, for any $i\in\mathbb{Z}_+$,
\begin{align}\label{guji3}
\|S_{L}(m)\|_{L^2(U_i(B))} \leq C_12^{-\eta_1 i}[\mu(B)]^\frac{1}{2}\|\mathbf{1}_B\|_X^{-1}.
\end{align}
\item[$\mathrm{(ii)}$]
Let $ M\in(\frac{n}{2}(\frac{1}{s}-\frac{1}{2})+1,\infty)\cap\mathbb{N}$. Then 
there exist positive constants $C_2$ and $\eta_2\in (n(\frac{1}{s}-\frac{1}{2}),\infty)$, independent of $m$, such that, for  any $i\in\mathbb{Z}_+$,
\begin{align}\label{guji5}
\|S_{L+I}(m)\|_{L^2(U_i(B))} \leq 
C_22^{-\eta_2 i}[\mu(B)]^\frac{1}{2}\|\mathbf{1}_B\|_X^{-1}.
\end{align}
\end{itemize}
\end{lemma}
\begin{proof}
We  first show (i).
If $r_B \in (0,1)$, then 
\eqref{guji3} follows from \cite[(3.19)]{lyyy26cms}.
Now, we assume that $r_B \in [1,\infty)$. 
Let $i \in \mathbb{N}$ and  
$\delta\in(0,1)$ be determined later. 
Then we have 
\begin{align*}
\|S_{L}(m)\|_{L^2(U_i(B))}^2
&= \int_{U_i(B)}\int_{0}^{2^{i\delta}r_B}\int_{B(x,t)}|t^2Le^{-t^2L}(m)(y)|^2\,\frac{d\mu(y)\,dt}{V(x,t)t}\,d\mu(x)\\
&\quad+ \int_{U_i(B)}\int_{2^{i\delta}r_B}^{\infty}\int_{B(x,t)}\cdots
=: \mathrm{I} + \mathrm{II}.
\end{align*}
By \eqref{heibang}, we conclude that 
\begin{align}\label{hebing}
\mathrm{I}
\lesssim
(2^{-i(1-\delta)(N-2n)}+2^{-i(2\epsilon-n)})
\mu(B)\|\mathbf{1}_B\|_{X}^{-2},
\end{align}
where $N\in\mathbb{N}$.
By the assumption that $\mathrm{inf}\,\sigma (L) \in (0,\infty)$
and \cite[Proposition 3]{k}, 
we find  that, there exists a positive constant $c$ such that, for any $t \in (0,\infty)$,
\begin{align*}
\int_{\mathcal{X}}\left|t^2Le^{-t^2L}(m)(y)\right|^2\,d\mu(y)
\lesssim
e^{-ct^2}\|m\|_{L^2}^2.
\end{align*} 
Using this, \eqref{d1}, and Fubini's Theorem, we obtain 
\begin{align*}
\mathrm{II}
\notag&\sim
\int_{2^{i\delta}r_B}^{\infty}\int_{\mathcal{X}}\left|t^2Le^{-t^2L}(m)(y)\right|^2\,d\mu(y)\,\frac{dt}{t}\\
&\lesssim
\|m\|_{L^2}^2\int_{2^{i\delta}r_B}^{\infty}\frac{e^{-ct^2}}{t}\,dt
\lesssim
(2^{\delta i }r_B)^{-N}\mu(B)\|\mathbf{1}_B\|_X^{-2}
\leq
2^{-N\delta i}\mu(B)\|\mathbf{1}_B\|_X^{-2},
\end{align*}
where we used $r_B\in [1,\infty)$ in the last step.
From this and \eqref{hebing},   it follows  that  \eqref{guji3} holds with 
\begin{align*}
\eta_1:=\mathrm{min}\left\{ (1-\delta)\left(\frac{N}{2}-n\right), \epsilon-\frac{n}{2}, \frac{N\delta}{2} \right\}.
\end{align*}
By the assumption that $\epsilon\in(\frac{n}{s},\infty)$, we  choose $\delta\in(0,1)$ and $N\in\mathbb{N}$ large enough such that
$\eta_1\in (n(\frac{1}{s}-\frac{1}{2}),\infty)$. This proves (i).

Next, we prove (ii).  
If $r_B \in [1,\infty)$, 
it is easy to verify that,
$m$ is also a local $(X,M,2,\epsilon)_{L+I}$-molecule. Using this, the fact that $\sigma(L+I) \in(0,\infty),$ and replacing $L$  by $L+I$ in \eqref{guji3}, we find that \eqref{guji5} holds.
Let $r_B\in (0,1)$. Then there exists $b\in\mathrm{Dom}(L^M)$ such that $m=L^M(b).$ Let $a:=L^{M-1}(b).$ Then $a$ is local $(X,M-1,2,\epsilon)_L$-molecule.
From this and \cite[(3.19)]{lyyy26cms}, it follows that, for any $i\in\mathbb{Z}_+$,  
\begin{align*}
\|S_{L+I}(m)\|_{L^2(U_i(B))}
\lesssim
\|S_{L}(m)\|_{L^2(U_i(B))}+
\|S_{L}(a)\|_{L^2(U_i(B))}
\lesssim
2^{-\eta_2 i}[\mu(B)]^\frac{1}{2}\|\mathbf{1}_B\|_X^{-1},
\end{align*}
where $\eta_2\in (n(\frac{1}{s}-\frac{1}{2}),\infty)$. This proves (ii), which   completes  the proof of Lemma \ref{smo22}. 
\end{proof}
\begin{lemma}\label{smo5}
Let $X$ be a $\mathrm{BQBF}$ space, $s\in(0,\infty)$,
$\epsilon\in (\frac{n}{s},\infty)$, and $ M\in(\frac{n}{2}(\frac{1}{s}-\frac{1}{2}),\infty)\cap\mathbb{N}$. Then 
there exist positive constants $C$ and $\eta\in (n(\frac{1}{s}-\frac{1}{2}),\infty)$ such that, 
for any $(X,M,2,\epsilon)_{L+I}$-molecule $m$, associated with the ball $B$, and  any $i\in\mathbb{Z}_+$,
\begin{align}\label{02041}
\|S_{L,\mathrm{loc}}(m)\|_{L^2(U_i(B))}
\leq
C2^{-\eta i}[\mu(B)]^{\frac{1}{2}}\left\|\mathbf{1}_B\right\|_{X}^{-1}
\end{align}
and
\begin{align}\label{shixian}
\|S_{I}(e^{-L}m)\|_{L^2(U_i(B))}
\leq
C2^{-\eta i}[\mu(B)]^{\frac{1}{2}}\left\|\mathbf{1}_B\right\|_{X}^{-1}.
\end{align}
\end{lemma}
\begin{proof}
We first show \eqref{02041}.
When $r_B \in [1,\infty)$, 
\eqref{02041} directly follows from 
Lemma \ref{smol}.
Let $r_B \in (0,1).$  Observe that, for any $t \in (0,\infty)$,
\begin{align*}
t^2Le^{-t^2L}=t^2(L+I)e^{-t^2(L+I)}+(1-e^{-t^2})t^2Le^{-t^2L}-t^2e^{-t^2}e^{-t^2L}.
\end{align*}
From this, we infer that, for any $g \in L^2$,
\begin{align}\label{uki}
S_{L,\mathrm{loc}}(g)\leq S_{L+I,\mathrm{loc}}(g)+U(g)+V(g),
\end{align}
where, for any $x \in \mathcal{X}$,
\begin{align*}
U(g)(x):=\left[\int_{0}^{1}\int_{B(x,t)}\left|(1-e^{-t^2})t^2Le^{-t^2L}(g)(y)\right|^2\,\frac{d\mu(y)\,dt}{V(x,t)t}\right]^{\frac{1}{2}}
\end{align*}
and 
\begin{align*}
V(g)(x):=\left[\int_{0}^{1}\int_{B(x,t)}\left|t^2e^{-t^2}e^{-t^2L}(g)(y)\right|^2\,\frac{d\mu(y)\,dt}{V(x,t)t}\right]^{\frac{1}{2}}.
\end{align*}
Let $m$ be a $(X,M,2,\epsilon)_{L+I}$-molecule.
Applying
Lemma \ref{smol} with $L$ replaced by $L+I$, 
we conclude that  there exists $\eta_1 \in (n(\frac{1}{s}-\frac{1}{2}),\infty)$ such that, for any $i\in\mathbb{Z}_+$,
\begin{align}\label{S1}
\left\|S_{L+I,\mathrm{loc}}(m)\right\|_{L^2(U_i(B))}\lesssim 2^{-\eta_1 i}[\mu(B)]^{\frac{1}{2}}\|\mathbf{1}_B\|_X^{-1}.
\end{align}
Since $1-e^{-t^2}\leq t^2\leq 1$ for any $t\in (0,1]$, it follows  that,
\begin{align*}
\|U(m)\|_{L^2(U_i(B))}^2 
&\leq \int_{U_i(B)}\int_{0}^{2^{i\delta}r_B}\int_{B(x,t)}\left|t^2Le^{-t^2L}(m)(y)\right|^2\,\frac{d\mu(y)\,dt}{V(x,t)t}\,d\mu(x)\\
&\quad+ \int_{U_i(B)}\int_{2^{i\delta}r_B}^{1}\int_{B(x,t)}\cdots
=:\mathrm{I} + \mathrm{II},
\end{align*}
where $\delta\in (0,1)$ 
and 
$\mathrm{II}=0$ when $2^{i\delta}r_B\in (1,\infty)$.
By \eqref{heibang}, we  have
\begin{align}\label{02051}
\mathrm{I}
\lesssim
(2^{-i(1-\delta)(N-2n)}+2^{-i(2\epsilon-n)})
\mu(B)\|\mathbf{1}_B\|_{X}^{-2}.
\end{align}
Since $m$ is  a $(X,M,2,\epsilon)_{L+I}$-molecule,
we deuce that there exists $b\in \mathrm{Dom}((L+I)^M)$ such that $m=(L+I)^M(b)$ and 
\begin{align}\label{bolin}
\|b\|_{L^2}^2
\lesssim
r_B^{4M}\mu(B)\|\mathbf{1}_B\|_{X}^{-2}.
\end{align}
By \eqref{lzy}, we find that, for any $f \in L^2$ and $t\in (0,1]$,
\begin{align*}
\left\|t^{2M}(L+I)^Me^{-t^2L/2}f\right\|_{L^2} 
\lesssim
\sum_{k=0}^M\left\|t^{2M}L^ke^{-t^2L/2}f\right\|_{L^2}
\lesssim
\sum_{k=0}^Mt^{2M-2k}\|f\|_{L^2}\lesssim
\|f\|_{L^2}.
\end{align*}
By this and \eqref{bolin},
we obtain, for any $t\in (0,1],$
\begin{align*}
\left\|t^2Le^{-t^2L}(m)\right\|_{L^2}^2
&=t^{-4M}\left\|t^2Le^{-t^2L/2}t^{2M}(L+I)^Me^{-t^2L/2}(b)\right\|_{L^2}^2\\
&\lesssim
t^{-4M}\|b\|_{L^2}
\lesssim
t^{-4M}r_B^{4M}\mu(B)\|\mathbf{1}_B\|_{X}^{-2}.
\end{align*}
This, combined with Fubini's Theorem, implies that, for any $i\in\mathbb{Z}_+,$ 
\begin{align*}
\mathrm{II}
&\sim
\int_{2^{i\delta}r_B}^1
\left\|t^2Le^{-t^2L}(m)\right\|_{L^2}^2\,\frac{dt}{t}
\lesssim
2^{-i4M\delta}\mu(B)\|\mathbf{1}_B\|_{X}^{-2}.
\end{align*}
which, together with  \eqref{02051},  further implies  that
\begin{align}\label{U}
\|U(m)\|_{L^2(U_i(B))}\lesssim2^{-\eta_2 i}[\mu(B)]^{\frac{1}{2}}\|\mathbf{1}_B\|_X^{-1},
\end{align}
where $\eta_2:=\mathrm{min}\left\{(1-\delta)\left(\frac{N}{2}-n\right), \epsilon-\frac{n}{2}, 2M\delta\right\}$. Choosing suitable $\delta \in(0,1)$ and $N \in \mathbb{N}$, we  have $\eta_2 \in(n(\frac{1}{s}-\frac{1}{2}),\infty)$. By an argument similar to that used in the proof of \eqref{U}, we conclude that
\begin{align*}
\|V(m)\|_{L^2(U_i(B))}\lesssim2^{-\eta_2 i}[\mu(B)]^{\frac{1}{2}}\|\mathbf{1}_B\|_X^{-1}.
\end{align*} 
From this, \eqref{U}, \eqref{S1}, and \eqref{uki}, we deduce that \eqref{02041} holds. Since  the proof of \eqref{shixian}  similar; we omit the details here. This finishes  the proof of Lemma \ref{smo5}.
\end{proof}
Now, we recall the concept of the Muckenhoupt weight class $A_p$.
Denote by $\mathbb{B}$ the set of all balls in $\mathcal{X}$.
\begin{definition}
A locally integrable function $\omega:\ \mathcal{X}\to[0,\infty)$ is called the
$A_p $-\emph{weight} with $p\in[1,\infty)$ if
\begin{align*}
[w]_{A_p }:=\sup_{B\in\mathbb{B}}\left\{[\mu(B)]^{-p}\|\omega\|_{L^1(B)}
\left\|\omega^{-1}\right\|_{L^{\frac{1}{p-1}}(B)}\right\}<\infty,
\end{align*}
where $\frac{1}{p-1}:=\infty$ when $p=1.$ Moreover, let
$A_\infty :=\bigcup_{p\in[1,\infty)}A_p.$
\end{definition}
We have  the following  extrapolation theorem
of ball quasi-Banach function  space, whose proof  is similar to that of \cite[Theorem 4.6]{cmp}; we omits the details here.
\begin{lemma}\label{1616}
Let $X$ be a $\mathrm{BQBF}$ space. Assume that $p_0\in (0,\infty)$, $X^{1/p_0}$ is a $\mathrm{BBF}$ space, and the \emph{Hardy-Littlewood} maximal operator $\mathcal{M}$ is bounded on ${(X^{1/p_0})}'$. Let $\mathcal{F}$ be the set of all pairs $(F,G)$ of nonnegative measurable functions such that, for any given $\omega \in A_1$,
\begin{align*}
\int_{\mathcal{X}}[F(x)]^{p_0}\omega(x)\,d\mu(x)\leq C_{(p_0,[\omega]_{A_1})}\int_{\mathcal{X}}[G(x)]^{p_0}\omega(x)\,d\mu(x),
\end{align*}
where $C_{(p_0,[\omega] _{A_1})}$ is a positive constant independent of $(F,G)$, but depends on both $p_0$ and $[\omega]_{A_1}$. Then there exists a positive constant $C_0$ such that, for any $(F,G)\in \mathcal{F}$,
\begin{align*}
\|F\|_X \leq C_0\|G\|_X.
\end{align*}
\end{lemma}
Now, we prove Theorem \ref{hH} via Lemmas \ref{smo22}, \ref{smo5}, and \ref{1616}.
\begin{proof}[Proof of Theorem  \ref{hH}]
We  first prove (i). By the molecule characterizations of  $H_{X,L}$(see  \cite[Theorem 3.4]{lyyy26cms}) and $h_{X,L}$ (Theorem \ref{local-mo-at}), we immediately have
\begin{align}\label{yaohao}
H_{X,L}\cap L^2 \subset h_{X,L}\cap L^2,
\end{align}
where we do not use the assumption that  $\mathrm{inf}\,\sigma (L) \in (0,\infty)$ in this step.
Let $f\in h_{X,L}\cap L^2,$
$\epsilon\in (\frac{n}{s},\infty)$, and $ M\in(\frac{n}{2}(\frac{1}{s}-\frac{1}{2}),\infty)\cap\mathbb{N}.$
By the molecule  characterization of
$h_{X,L}$ (Theorem \ref{local-mo-at}), we find that there exist
a sequence  $\{\lambda_j\}_{j\in\mathbb{N}}\subset [0,\infty)$
and a sequence $\{m_j\}_{j\in\mathbb{N}}$ of local
$(X,M,2,\epsilon)_L$-molecules
such that 
\begin{align*}
f=\sum_{j\in\mathbb{N}}\lambda_jm_j\ \ \text{in}\ \ L^2
\end{align*}
and
\begin{align*}
\left\|\left\{\sum_{j\in\mathbb{N}}
\left[\frac{\lambda_j}{\|\mathbf{1}_{B_j}\|_{X}}\right]^{s}
\mathbf{1}_{B_j}\right\}^{\frac{1}{s}}\right\|_{X}
\lesssim
\|f\|_{h_{X,L}}.
\end{align*}
Using this and Lemma \ref{smo22}$\mathrm{(i)}$ and applying Lemma \ref{beike} with $T:=S_{L}$, we find that 
\begin{align}\label{yige}
\|f\|_{H_{X,L}}
=\|S_{L}(f)|
\lesssim
\left\|\left\{\sum_{j\in\mathbb{N}}\left[\frac{\lambda_j}{\|\mathbf{1}_{B_j}\|_{X}}\right]^s\mathbf{1}_{B_j}\right\}^{\frac{1}{s}}\right\|_X
\lesssim
\|f\|_{h_{X,L}}.
\end{align}
This implies that $h_{X,L}\cap L^2\subset H_{X,L}\cap L^2$, which, combined with \eqref{yaohao}, further implies that $h_{X,L}\cap L^2= H_{X,L}\cap L^2.$
From this and a density argument, we deduce that  $h_{X,L}$ and
$H_{X,L}$
coincide with equivalent quasi-norms. This shows (i). 

Next, we prove $\mathrm{(ii)}$.
By an argument similar to that used in \eqref{yige} and applying Lemma   \ref{beike}, respectively, with $T:=S_{L+I}$, $T:=S_{L,\mathrm{loc}}$, and $T:=S_{I}(e^{-L})$, we  conclude that
\begin{align}\label{xiede}
H_{X,L+I}=h_{X,L}.
\end{align}
By this fact that $\sigma(L+I)\in(0,\infty)$ and using (i) with $L$ replaced by $L+I$,  we have
\begin{align*}
h_{X,L+I}=H_{X,L+I}.
\end{align*}
From this and applying  \eqref{xiede} with $L$ replaced by $L+I$, we deduce that 
\begin{align*}
H_{X,L+2I}=h_{X,L+I}=H_{X,L+I}=h_{X,L}.
\end{align*} 
This shows that  the  spaces $H_{X,L+2I}$ and $h_{X,L}$ coincide with equivalent quasi-norms. By induction, we find that the  desired conclusion holds. This shows (ii).

Now, we prove $\mathrm{(iii)}$.
By an argument similar to that used in \eqref{yige} and applying Lemma   \ref{beike} with $T:=I$, we  conclude that
\begin{align}\label{002}
h_{X,L}\cap L^2\subset X\cap L^2.
\end{align} 
By \cite[Theorem 1.1$\mathrm{(a)}$]{gx12}, we find that, for any $\omega\in A_1$ and $f\in L^{p_0}\cap L^2,$
\begin{align*}
\int_{\mathcal{X}}[S_L(f)(x)]^p\omega(x)\,d\mu(x)
\lesssim
\int_{\mathcal{X}}|f(x)|^p\mathcal{M}\omega(x)\,d\mu(x)
\lesssim
[\omega]_{A_1}\int_{\mathcal{X}}|f(x)|^p\omega(x)\,d\mu(x).
\end{align*}
From this and  Lemma \ref{1616}, we infer that, for any $f\in X\cap L^2,$
\begin{align*}
\|f\|_{H_{X,L}}=\|S_L(f)\|_X \lesssim\|f\|_X.
\end{align*}
This implies that $X\cap L^2\subset H_{X,L}\cap L^2$. By this, \eqref{002}, and \eqref{yaohao}, we obtain
\begin{align*}
X\cap L^2 \subset H_{X,L}\cap L^2 \subset h_{X,L}\cap L^2 \subset X\cap L^2.
\end{align*}
This shows (iii), which completes the proof of Theorem \ref{hH}. 
\end{proof}  
\subsection{Proof of Theorems \ref{thm-maxi} and \ref{key}}\label{jida}
In this subsection, we prove Theorems \ref{thm-maxi} and \ref{key}.
Let $\alpha,\lambda\in(0,\infty)$. 
For any $F\in\mathcal{X}^+$ and $x\in\mathcal{X}$, let
\begin{align*}
F^\ast_{\alpha}(x):=\sup_{t\in(0,1)}\sup_{y\in B(x,\alpha t)}|F(y,t)|
\end{align*}
and
\begin{align*}
F^\ast_{\alpha,\lambda}(x)
:=\sup_{t\in(0,1)}\sup_{y\in\mathcal{X}}|F(y,t)|
\left(1+\frac{d(x,y)}{\alpha t}\right)^{-\lambda}.
\end{align*}
If $\alpha:=1$, we write $F^\ast$ instead of $F^\ast_1.$
\begin{proposition}\label{kenan1}
Let $X$ be a $\mathrm{BQBF}$ space satisfying Assumption \ref{vector1} for some $p_-\in (0,\infty).$ Let $\lambda\in (\frac{n}{p_-},\infty)$.
Then there exists a positive constant $C$ such that, for any $\alpha_1,\alpha_2\in(0,\infty)$ and any measurable function $F$ on $\mathcal{X}^+$, 
\begin{align*}
\|F^\ast_{\alpha_1}\|_{X}
\leq
C\left(1+\frac{\alpha_1}{\alpha_2}\right)^{\lambda}
\|F^\ast_{\alpha_2}\|_{X}.
\end{align*}
\end{proposition}
\begin{proof}
Let $F$ be a measurable function on $\mathcal{X}^+$
and $\alpha_1,\alpha_2\in(0,\infty).$
By the definition of $F^\ast_{\alpha_2}$, we obtain,  for any $t\in(0,\infty)$,
$y\in\mathcal{X},$ and $z\in B(y,\alpha_2 t),$
\begin{align}\label{jiao}
|F(y,t)|
\leq
F^\ast_{\alpha_2}(z).
\end{align}
Observe that, for any $t\in(0,\infty)$ and $x,y\in\mathcal{X}$, $B(y, \alpha_2t)\subset B(x,d(x,y)+\alpha_2t)$, which, combined with \eqref{d2}, further implies that 
\begin{align*}
V(x,d(x,y)+\alpha_2t)
&\lesssim
\left(1+\frac{d(x,y)}{d(x,y)+\alpha_2t}\right)^nV(y,d(x,y)+\alpha_2t)\\
&\leq
\left(1+\frac{d(x,y)}{\alpha_2t}\right)^n
V(y,\alpha_2t)
\leq
\left(1+\frac{\alpha_1}{\alpha_2}\right)^n\left(1+\frac{d(x,y)}{\alpha_1t}\right)^n
V(y,\alpha_2t).
\end{align*}
From this and \eqref{jiao}, it follows that,
for any $t\in(0,\infty)$ and $x,y\in\mathcal{X}$,
\begin{align*}
|F(y,t)|^{\frac{n}{\lambda}}
&\leq
\frac{1}{V(y,\alpha_2t)}\int_{B(y,\alpha_2t)}|F^\ast_{\alpha_2}(z)|^{\frac{n}{\lambda}}\,d\mu(z)
\leq
\frac{1}{V(y,\alpha_2t)}
\int_{B(x,d(x,y)+\alpha_2t)}|F^{\ast}_{\alpha_2}(z)|^{\frac{n}{\lambda}}\,d\mu(z)\\
&\lesssim
\left(1+\frac{\alpha_1}{\alpha_2}\right)^n\left(1+\frac{d(x,y)}{\alpha_1t}\right)^n\mathcal{M}\left([F^\ast_{\alpha_2}]^{\frac{n}{\lambda}}\right)(x).
\end{align*} 
By this, we conclude that, for any $x\in\mathcal{X},$
\begin{align*}
F^\ast_{\alpha_1,\lambda}(x)
\lesssim
\left(1+\frac{\alpha_1}{\alpha_2}\right)^\lambda\mathcal{M}^{(\frac{n}{\lambda})}(F^\ast_{\alpha_2})(x).
\end{align*}
Using this and Assumption \ref{vector1}, we obtain
\begin{align*}
\|F^\ast_{\alpha_1}\|_{X}
\leq
2^\lambda\|F^\ast_{\alpha_1,\lambda}\|_{X}
\lesssim
\left(1+\frac{\alpha_1}{\alpha_2}\right)^{\lambda}
\left\|\mathcal{M}^{(\frac{n}{\lambda})}(F^\ast_{\alpha_2})\right\|_{X}
\lesssim
\left(1+\frac{\alpha_1}{\alpha_2}\right)^{\lambda}
\|F^\ast_{\alpha_2}\|_{X}.
\end{align*}
This finishes the proof of Proposition \ref{kenan1}.
\end{proof}
\begin{proposition}\label{kenan2}
Let $X$ be a $\mathrm{BQBF}$ space satisfying Assumption \ref{vector1} for some $p_-\in (0,\infty).$ Let $\lambda\in (\frac{n}{p_-},\infty)$.
Then there exists a positive constant $C$ such that, for any measurable function $F$
on $\mathcal{X}^+,$
\begin{align*}
\|F^\ast_{1,\lambda}\|_{X}
\leq
C\|F^\ast\|_{X}.
\end{align*}
\end{proposition}
\begin{proof}
Let  $F$ be a measurable function on $\mathcal{X}^+.$
It is easy to verify that, for any $x\in\mathcal{X},$
\begin{align}\label{tilian1}
F^\ast_{1,\lambda}(x)
\leq
F^\ast(x)+\sum_{k\in\mathbb{Z}_+}2^{-k\lambda}F^\ast_{2^{k+1}}(x).
\end{align}
Let $\lambda_1\in (\frac{n}{p_-}, \lambda).$
Using Proposition \ref{kenan1} with $\alpha_1:=2^{k+1}$ and $\alpha_2:=1$, we find that, for any $k\in\mathbb{Z}_+,$
\begin{align}\label{tilian2}
\left\|F^\ast_{2^{k+1}}\right\|_X
\lesssim
(1+2^{k+1})^{\lambda_1}\|F^\ast\|_X
\lesssim
2^{k\lambda_1}\|F^\ast\|_X.
\end{align}
By the Aoki--Rolewicz theorem (see, for instance, \cite[Exercises 1.4.6]{g14c}), we find that, there exists $\beta\in (0,1]$ such that, for any
$\{f_i\}_{i\in\mathbb{Z}_+}\subset \mathscr{M}$,
\begin{align*}
\left\|\sum_{i\in\mathbb{Z}_+}f_i\right\|^\beta_X
\leq
\sum_{i\in\mathbb{Z}_+}4\|f_i\|_{X}^\beta.
\end{align*}
From this, \eqref{tilian1}, and \eqref{tilian2}, we infer that
\begin{align*}
\|F^\ast_{1,\lambda}\|_{X}^\beta
\lesssim
\|F^\ast\|_{X}^\beta+\sum_{k\in\mathbb{Z}_+}2^{-k\beta\lambda}\left\|F^\ast_{2^{k+1}}\right\|_X^\beta
\lesssim
\sum_{k\in\mathbb{Z}_+}2^{-k\beta(\lambda-\lambda_1)}\|F^\ast\|_X^\beta
\sim 
\|F^\ast\|_X^\beta.
\end{align*}
This finishes the proof of Proposition \ref{kenan2}.
\end{proof}
Let 
$\varphi\in\mathcal{S}(\mathbb{R})$ be an even function with $\varphi(0)=1$ and $\lambda\in(0,\infty)$.
For any $f\in L^2$, the \emph{local auxiliary maximal function} $M^\ast_{L,\varphi,\lambda}(f)$, associated
with $\varphi$ and $\lambda$, is defined by setting, for any $x\in\mathcal{X},$
\begin{align*}
M^\ast_{L,\varphi, \lambda}(f)(x):=
\sup_{t\in(0,1)}\sup_{y\in\mathcal{X}}\left|\varphi\left(t\sqrt{L}\right)(f)(y)\right|\left(1+\frac{d(x,y)}{t}\right)^{-\lambda}.
\end{align*}
\begin{proposition}\label{kenan3}
Let $X$ be a $\mathrm{BQBF}$ space satisfying Assumption \ref{vector1} for some $p_-\in(0,\infty).$ Let $\alpha_1,\alpha_2\in (0,\infty)$ and $\varphi_1,\varphi_2\in\mathcal{S}(\mathbb{R})$ be even functions satisfying $\varphi_1(0)=1$ and $\varphi_2(0)=0.$ Then there exists a positive constant $C$ such that, for any $f\in L^2,$
\begin{align}\label{1112}
\left\|(\varphi_2)^\ast_{L,\alpha_2}(f)\right\|_{X}
\leq
C\left\|(\varphi_1)^\ast_{L,\alpha_1}(f)\right\|_{X}.
\end{align}
\end{proposition}
\begin{proof}
By Proposition \ref{kenan1}, we can assume that $\alpha_1=\alpha_2=1.$ Let $\lambda\in (\frac{n}{p_-},\infty)$.
By the proof of \cite[Proposition 4.3]{bdl18}, we obtain, for any $x\in\mathcal{X}, $ $(\varphi_2)_{L}^\ast (f)(x)
\lesssim
M^\ast_{L,\varphi_1,\lambda}(f)(x).$
Using this and Proposition \ref{kenan2}, we find that
\eqref{1112} holds. This finishes the proof of Proposition \ref{kenan3}.
\end{proof}
\begin{proposition}\label{kenan4}
Let $X$ be a $\mathrm{BQBF}$ space satisfying Assumption \ref{vector1} for some $p_-\in(0,\infty)$. 
Let 
$\varphi\in\mathcal{S}(\mathbb{R})$ be even function satisfying $\varphi(0)=1$,
$\lambda\in (\frac{n}{p_-},\infty),$
and $\alpha\in(0,\infty).$
Then there exists a positive constant $C$
such that, for any $f\in h_{X,L,\mathrm{rad}}^\varphi\cap L^2,$
\begin{align}\label{zhanyi}
(1+\alpha)^{-\lambda}\|f\|_{h_{X,L,\mathrm{max}}^{\varphi,\alpha}}
\leq
\left\|M^\ast_{L,\varphi, \lambda}(f)\right\|_{X}
\leq
C\|f\|_{h_{X,L,\mathrm{rad}}^\varphi}.
\end{align}
\end{proposition}
\begin{proof}
Let $f\in h_{X,L,\mathrm{rad}}^\varphi\cap L^2$ and
$\theta\in(0,p_-)$ satisfying $\frac{n}{p_-}<\frac{n}{\theta}<\lambda.$
By the proof of \cite[Proposition 4.4]{bdl18}, we obtain, for any $x\in\mathcal{X},$
\begin{align*}
M^\ast_{L,\varphi, \lambda}(f)(x)
\lesssim
\left[\mathcal{M}\left(\left[\varphi^+_L(f)\right]^\theta\right)\right]^{\frac{1}{\theta}}(x).
\end{align*}
From this, the definitions of $M^\ast_{L,\varphi, \lambda}(f)$ and $(\varphi)^\ast_{L,\alpha}$, and Assumption \ref{vector1}, it follows that \eqref{zhanyi} holds. This finishes the proof of Proposition \ref{kenan4}.
\end{proof}
The following proposition is precisely
\cite[Proposition 3.6]{bdl18}.
\begin{proposition}\label{gj3}
Let $\varphi$ be as in Lemma \ref{gj1}. Let $\psi\in C^\infty_{\rm{c}}(\mathbb{R})$ be an even function satisfying $\int_{\mathbb{R}}\psi(x)\,dx=2\pi$ and $\mathrm{supp}\,(\psi)\subset(-1,1)$.
Denotes the Fourier transform of $\psi$
by $\hat{\psi}$.
For any $\ell,j\in\mathbb{Z}_+$ and $\xi\in\mathbb{C},$ let
\begin{align*}
\Psi^{(\ell)}(\xi):=\frac{d^\ell\hat{\psi}(\xi)}{d\xi^\ell}
\ \ \text{and}\ \ 
\Psi^{(\ell)}_{j}(\xi):=\xi^{j}\Psi^{(\ell)}(\xi).
\end{align*}
Then, for any $M\in\mathbb{N}$ and $f\in L^2,$ 
\begin{align*}
f
&=\sum_{\ell=0}^{M+1}c(M,\ell)\int_0^{\frac{1}{2}}(t^2L)^M\Phi^{(2\ell)}_{1}(t\sqrt{L})\Psi^{(2M-2\ell+2)}_{1}(t\sqrt{L})(f)\,\frac{dt}{t}\\
&\quad+\sum_{\ell=0}^{M}c(M,\ell)\int_0^{\frac{1}{2}}(t^2L)^M\Phi^{(2\ell+1)}_{2}(t\sqrt{L})\Psi^{(2M-2\ell+1)}_2(t\sqrt{L})(f)\,\frac{dt}{t}\\
&\quad+\sum_{\ell=1}^{2M+2}\sum_{j=0}^\ell
c(M,\ell,j)\Phi^{(j)}_j(2^{-1}\sqrt{L})\Psi^{(\ell-j)}_{\ell-j}(2^{-1}\sqrt{L})(f)\ \ \mathrm{in} \ L^2.
\end{align*}
\end{proposition}
The following lemma is a part of \cite[Lemma 3.13]{lyyy26cms}.
\begin{lemma}\label{tent-ge}
Let $\Omega$ be a proper open subset of $\mathcal{X}$. Then there exists a sequence
$\{x_j\}_{j\in\mathbb{N}}\subset \Omega$ such that
\begin{enumerate}
\item[\rm{(i)}]
$\Omega=\bigcup_{j\in\mathbb{N}}B(x_j,r_j),$
where $r_j:=2^{-5}\mathrm{dist}\,(x_j,\Omega^\complement)$ for any $j\in\mathbb{N}$;
\item[\rm{(ii)}] $\{B(x_j,5^{-1}r_j)\}_{j\in\mathbb{N}}$ is
pairwise disjoint.

\noindent Moreover, there exists a family $\{\phi_j\}_{j \in \mathbb{N}}$ of non-negative functions on $\mathcal{X}$ such that
\item[\rm{(iii)}] 
for any $j \in \mathbb{N}$, $\operatorname{supp}(\phi_j) \subset B(x_j, 2r_j)$;
\item[\rm{(iv)}] 
$\sum_{j \in \mathbb{N}} \phi_j = 1_{\Omega}$.
\end{enumerate}
\end{lemma}
\begin{proposition}\label{pro-max-dec}
Let $X$ be a $\mathrm{BQBF}$ space satisfying Assumption \ref{vector1} for some $p_-\in (0,\infty)$ and $M\in\mathbb{N}.$ Then there  exists a positive constant $C$ such that, for any $f\in h_{X,L,\max}\cap L^2,$ $f\in \widetilde{h}_{X,L,\mathrm{at}}^{M,\infty}$ and 
\begin{align*}
\|f\|_{h_{X,L,\mathrm{at}}^{M,\infty}}
\leq
C\|f\|_{h_{X,L,\max}}.
\end{align*}
\end{proposition}
\begin{proof}
Let $f\in h_{X,L,\max}\cap L^2$.  By Proposition \ref{gj3}, we obtain
\begin{align*}
f
&=\sum_{\ell=0}^{M+1}c(M,\ell)\int_0^{\frac{1}{2}}(t^2L)^M\Phi^{(2\ell)}_{1}(t\sqrt{L})\Psi^{(2M-2\ell+2)}_1(t\sqrt{L})(f)\,\frac{dt}{t}\\
&\quad+\sum_{\ell=0}^{M}c(M,\ell)\int_0^{\frac{1}{2}}(t^2L)^M\Phi^{(2\ell+1)}_{2}(t\sqrt{L})\Psi^{(2M-2\ell+1)}_2(t\sqrt{L})(f)\,\frac{dt}{t}\\
&\quad+\sum_{\ell=0}^{2M+2}\sum_{j=0}^\ell
c(M,\ell,j)\Psi^{(j)}_j(2^{-1}\sqrt{L})\Psi^{(\ell-j)}_{\ell-j}(2^{-1}\sqrt{L})(f)\\
&=:\sum_{\ell=0}^{M+1}f_{\ell,1}
+\sum_{\ell=0}^{M}f_{\ell,2}
+\sum_{\ell=1}^{2M+2}\sum_{j=0}^\ell g_{\ell,j}\ \ \text{in} \ L^2.
\end{align*}
Let $\ell\in\{1,\ldots,2M+2\}$ and $j\in\{0,\ldots,\ell\}$. We first  deal with $g_{\ell,j}.$ Let $\{Q_k\}_{k\in\mathrm{I}}$ and $\{B_k\}_{k\in\mathrm{I}}$ be as in Definition  \ref{unit}.
Then $g_{\ell,j}=\sum_{k\in\mathrm{I}}\lambda_ka_k$ in
$L^2,$ where, for any $k\in\mathrm{I},$
\begin{align*}
\lambda_k:=\left\|\mathbf{1}_{\frac{3}{2}B_k}\right\|_{X}\sup_{z\in Q_k}\left|\Psi^{(\ell-j)}_{\ell-j}(2^{-1}\sqrt{L})(f)(z)\right|
\end{align*} 
and
\begin{align*}
a_k:=\frac{c(M,\ell,j)}{\lambda_k}\Phi^{(j)}_j(2^{-1}\sqrt{L})\left[\Psi^{(\ell-j)}_{\ell-j}(2^{-1}\sqrt{L})(f)\mathbf{1}_{Q_k}\right].
\end{align*}
By Lemma \ref{gj1}, we find that $\mathrm{supp}\, (a_k)\subset \frac{3}{2}B_k$ and
\begin{align*}
\|a_k\|_{L^\infty}
\lesssim
\left\|\mathbf{1}_{\frac{3}{2}B_k}\right\|_{X}^{-1}\mu(Q_k)\sup_{x\in \frac{3}{2}B_k}\frac{1}{V(x,2^{-1})}
\lesssim
\left\|\mathbf{1}_{\frac{3}{2}B_k}\right\|_{X}^{-1}.
\end{align*}
This proves that $a_k$ is an $(X,M,\infty)_L$-atom up to a harmless positive constant multiple.
Moreover, observe that, for any $x\in Q_k$,
$Q_k\subset B(x,2)$. This implies that, for almost every $x\in\mathcal{X},$ 
\begin{align*}
\frac{\lambda_k}{\left\|\mathbf{1}_{\frac{3}{2}B_k}\right\|_{X}}\mathbf{1}_{Q_k}(x)
&\leq
\sup_{y\in B(x,2)}\left|\Psi^{(\ell-j)}_{\ell-j}(2^{-1}\sqrt{L})(f)(y)\right|\mathbf{1}_{Q_k}(x)\\
&\leq
\sup_{t\in(0,1)}\sup_{y\in B(x,4t)}\left|\Psi^{(\ell-j)}_{\ell-j}(t\sqrt{L})(f)(y)\right|\mathbf{1}_{Q_k}(x)
=:\left[\Psi^{(\ell-j)}_{\ell-j}\right]^\ast_{L,4}(f)(x)\mathbf{1}_{Q_k}(x).
\end{align*}
Since $\{Q_k\}_{k\in\mathrm{I}}$ is pairwise disjoint, it follows that, for almost every $x\in\mathcal{X},$ 
\begin{align*}
\sum_{k\in\mathrm{I}}\frac{\lambda_k}{\|\mathbf{1}_{\frac{3}{2}B_k}\|_X}\mathbf{1}_{Q_k}(x)
\leq
\sum_{k\in\mathrm{I}}\left[\Psi^{(\ell-j)}_{\ell-j}\right]^\ast_{L,4}(f)(x)\mathbf{1}_{Q_k}(x)
=\left[\Psi^{(\ell-j)}_{\ell-j}\right]^\ast_{L,4}(f)(x).
\end{align*}
This, combined with the fact that $\{Q_k\}_{k\in\mathrm{I}}$ is  unit cube structure on $\mathcal{X}$
and Propositions \ref{prfs} and \ref{kenan3},  implies that
\begin{align*}
&\left\|\left\{\sum_{k\in\mathrm{I}}\left[\frac{\lambda_k}{\|\mathbf{1}_{\frac{3}{2}B_k}\|_X}\right]^s\mathbf{1}_{\frac{3}{2}B_k}\right\}^{\frac{1}{s}}\right\|_X\\
&\quad\lesssim
\left\|\left\{\sum_{k\in\mathrm{I}}\left[\frac{\lambda_k}{\|\mathbf{1}_{\frac{3}{2}B_k}\|_X}\right]^s\mathbf{1}_{\delta B_k}\right\}^{\frac{1}{s}}\right\|_X
\lesssim
\left\|\left\{\sum_{k\in\mathrm{I}}\left[\frac{\lambda_k}{\|\mathbf{1}_{\frac{3}{2}B_k}\|_X}\right]^s\mathbf{1}_{Q_k}\right\}^{\frac{1}{s}}\right\|_X\\
&\quad=
\left\|\sum_{k\in\mathrm{I}}\frac{\lambda_k}{\|\mathbf{1}_{\frac{3}{2}B_k}\|_X}\mathbf{1}_{Q_k}\right\|_X
\leq
\left\|\left[\Psi^{(\ell-j)}_{\ell-j}\right]^\ast_{L,4}(f)\right\|
\lesssim
\|f\|_{h_{X,L,\mathrm{max}}}.
\end{align*}
This proves that $g_{\ell,j}$
has a local atomic $(X,M,\infty)_{L}$-representation.

Let $\ell\in\{0,1,\ldots,M+1\}.$
Now, we deal with $f_{\ell,1}.$ For any $\xi\in\mathbb{R},$ let
\begin{align*}
\eta_\ell(\xi):=\int_0^1(t^2\xi^2)^M\Phi^{(2\ell)}_1(t\xi)\Psi^{(2M-2\ell+2)}_1(t\xi)\,\frac{dt}{t}.
\end{align*}
Then it is easy to verify that
$\eta_\ell\in\mathcal{S}(\mathbb{R})$ and $\eta_{\ell}(0)=0.$ For any $x\in\mathcal{X},$
let
\begin{align*}
N^\ast_L(f)(x):=\sup_{t\in(0,1)}\sup_{y\in B(x,5t)}\left[|\eta_\ell(t\sqrt{L})(f)|+|\Psi^{(2M-2\ell+2)}(t\sqrt{L})(f)(y)|\right].
\end{align*}
By Proposition \ref{kenan3}, we obtain
\begin{align}\label{1044}
\|N^\ast_{L}(f)\|_{X}\lesssim \|f\|_{h_{X,L,\max}}.
\end{align}
For any $i\in\mathbb{Z}$, let $O_i:=\{x\in\mathcal{X}:\ N^\ast_L(f)(x)>2^i\}$ and
$\widehat{O}_i:=\{(x,t)\in\mathcal{X}\times(0,1):\ B(x,4t)\subset O\}.$
Then
\begin{align*}
\mathcal{X}\times(0,1)=\bigcup_{i\in\mathbb{Z}}\widehat{O}_i\setminus \widehat{O}_{i+1}=:\bigcup_{i\in\mathbb{Z}}T_i.
\end{align*}
For any given $i\in\mathbb{Z},$ applying Lemma
\ref{tent-ge} with $\Omega:=O_i$, we obtain a  family of balls $\{B_{i,k}\}_{k\in\mathbb{N}}$ satisfied
the properties in Lemma \ref{tent-ge}.
For any $i\in\mathbb{Z}$
and $k\in\mathbb{Z}$, let $R(B_{i,0}):=\emptyset$ and
$R(B_{i,k}):=\{(x,t)\in\mathcal{X}\times (0,1):\ d(x,B_{i,k})<t\}.$
Moreover, define
\begin{align*}
T_{i,k}:=T_i\cap \left(R(B_{i,k})\setminus \bigcup_{j=0}^{k-1}R(B_{i,j})\right).
\end{align*}
It is easy to verify that $\{T_{i,k}\}_{i\in\mathbb{Z},k\in\mathbb{Z}_+}$ is disjoint and
\begin{align*}
X\times (0,1)=\bigcup_{i\in\mathbb{Z}}T_i=\bigcup_{i\in\mathbb{Z}}\bigcup_{k\in\mathbb{Z}_+}T_{i,k}.
\end{align*}
This implies that
$f_{\ell,1}
=\sum_{i\in\mathbb{Z}}\sum_{k\in\mathbb{Z}_+}\lambda_{i,k}a_{i,k}$ in $L^2,$
where, for any $i\in\mathbb{Z}$ and $k\in\mathbb{Z}_+$,
$\lambda_{i,k}:=2^i\|\mathbf{1}_{3B_{i,k}}\|_{X}$ and
\begin{align*}
a_{i,k}:=\frac{c(M,\ell)}{\lambda_{i,k}}\int_0^{\frac{1}{2}}(t^2L)^M\Phi^{(2\ell)}_{1}(t\sqrt{L})\left[\Psi^{(2M-2\ell+2)}_1(t\sqrt{L})(f)\mathbf{1}_{T_{i,k}}\right]\,\frac{dt}{t}.
\end{align*}
By the proof of \cite[(26)]{bdl18}, we  conclude that, 
for any $i\in\mathbb{Z}$ and $k\in\mathbb{Z}_+$, $a_{i,k}$ is a $(X,M,\infty)_L$-atom,
associated with $3B_{i,k}$ up to
a harmless positive constant multiple.
By Proposition \ref{prfs}, Lemma \ref{tent-ge}(ii),  the definition of $O_i$, and \eqref{1044}, we conclude that 
\begin{align*}
&\left\|\left\{\sum_{i\in\mathbb{Z}}\sum_{k\in\mathbb{Z}_+}\left[\frac{\lambda_{i,k}}{\|\mathbf{1}_{3B_{i,k}}\|_{X}}\right]^s\mathbf{1}_{3B_{i,k}}\right\}^{\frac{1}{s}}\right\|_{X}\\
&\quad\lesssim
\left\|\left\{\sum_{i\in\mathbb{Z}}\sum_{k\in\mathbb{Z}_+}2^{is}\mathbf{1}_{5^{-1}B_{i,k}}\right\}^{\frac{1}{s}}\right\|_{X}
\leq
\left\|\left\{\sum_{i\in\mathbb{Z}}2^{is}\mathbf{1}_{O_i}\right\}^{\frac{1}{s}}\right\|_{X}
\sim
\left\|\left\{\sum_{i\in\mathbb{Z}}2^{is}\mathbf{1}_{O_i\setminus O_{i+1}}\right\}^{\frac{1}{s}}\right\|_{X}\\
&\quad\sim
\left\|\left\{\sum_{i\in\mathbb{Z}}[N^\ast_{L}(f)]^s\mathbf{1}_{O_i\setminus O_{i+1}}\right\}^{\frac{1}{s}}\right\|_{X}
=\|N^\ast_{L}(f)\|_X
\lesssim
\|f\|_{h_{X,L,\max}}.
\end{align*}
Thus,  $f_{\ell,1}$
has a local atomic $(X,M,\infty)_{L}$-representation.
Similarly, $f_{\ell,2}$
also has a local atomic $(X,M,\infty)_{L}$-representation.
This finishes the proof of Proposition \ref{pro-max-dec}.
\end{proof}
\begin{proposition}\label{oneside}
Let $X$ be a $\mathrm{BQBF}$ space satisfying Assumptions \ref{vector1} and  \ref{vector2} for some $p_-\in(0,\infty),$ $s\in(0,1],$ and $q\in (s,\infty)$.
Let $p\in(1,\infty)\cap[q,\infty)$ and $M\in\mathbb{N}\cap(\frac{n}{2\min\{s,p_-\}},\infty)$.
Then there exists a positive constant $C$
such that, for any $f\in \widetilde{h}^{M,p}_{X,L,\mathrm{at}},$
$f\in h_{X,L,\mathrm{max}}\cap L^2$ and
\begin{align}\label{tuozhezou}
\|f\|_{h_{X,L,\mathrm{max}}}
\leq
C\|f\|_{h^{M,p}_{X,L,\mathrm{at}}}.
\end{align}
\end{proposition}
\begin{proof}
Let $f\in\widetilde{h}^{M,p}_{X,L,\mathrm{at}}$. Then there exists a sequence
$\{\lambda_j\}_{j\in\mathbb{N}}$ in $[0,\infty)$
and a sequence $\{a_{j}\}_{j\in\mathbb{N}}$
of local $(X,M,p)_L$-atoms associated, respectively, with balls $\{B_j\}_{j\in\mathbb{N}}$ such that
\begin{align}\label{suomei1}
\left\|\left\{\sum_{j\in\mathbb{N}}
\left[\frac{\lambda_j}{\|\mathbf{1}_{B_j}\|_{X}}\right]^{s}
\mathbf{1}_{B_j}\right\}^{\frac{1}{s}}\right\|_{X}
\lesssim
\|f\|_{h_{X,L,\mathrm{at}}^{M,p}}
\end{align}
and $f=\sum_{j\in\mathbb{N}}\lambda_j a_j$ in $L^2.$ This implies  that
\begin{align}\label{suomei2}
\|f\|_{h_{X,L,\mathrm{max}}}
\notag&\leq
\left\|\sum_{j\in\mathbb{N}}\lambda_j(a_j)^\ast_L\right\|_{X}
\lesssim
\left\|\sum_{j\in\mathbb{N}}\lambda_j(a_j)^\ast_L\mathbf{1}_{4B_j}\right\|_{X}
\notag\\
&\quad+\left\|\sum_{j\in\mathbb{N}}\sum_{i=3}^\infty\lambda_j(a_j)^\ast_L\mathbf{1}_{U_i(B_j)}\right\|_{X}=:\mathrm{I}+\mathrm{II}.
\end{align}
By the $L^p$-boundedness of $\mathcal{M}$
and the definition of  local $(X,M,p)_L$-atoms, we obtain, for any $j\in\mathbb{N},$
\begin{align*}
\|\lambda_j (a_j)^\ast_{L}\|_{L^p}
\lesssim
\|\lambda_j \mathcal{M}(a_j)\|_{L^p}
\lesssim
\|\lambda_j a_j\|_{L^p}
\leq
\lambda_j[\mu(B_j)]^{\frac{1}{p}}\|\mathbf{1}_{B_j}\|_{X}^{-1}.
\end{align*}
This, combined with Proposition \ref{pras} and \eqref{jiben},
implies that
\begin{align}\label{suomei3}
\mathrm{I}
\leq
\left\|\left\{\sum_{j\in\mathbb{N}}\left[\lambda_j(a_j)^\ast_L\mathbf{1}_{4B_j}\right]^s\right\}^{\frac{1}{s}}\right\|_{X}
\lesssim
\left\|\left\{\sum_{j\in\mathbb{N}}
\left[\frac{\lambda_j}{\|\mathbf{1}_{B_j}\|_{X}}\right]^{s}
\mathbf{1}_{B_j}\right\}^{\frac{1}{s}}\right\|_{X}.
\end{align}
Let  $\delta\in(\frac{n}{\min\{s,p_-\}},2M)$. We first show that there exists a positive constant $C$ such that,
for any  local $(X,M,p)_L$-atom $a$, associated with ball $B:=B(x_B,r_B),$
and any $x\in (4B)^\complement,$
we claim that
\begin{align}\label{zidongs}
a^\ast_{L}(x)
\lesssim
\left[\frac{r_B}{d(x,x_B)}\right]^{\delta}\|\mathbf{1}_{B}\|_{X}^{-1}.
\end{align}
For the moment, we take this claim for granted, and proceed with the proof of the present proposition. Applying \eqref{zidongs} with $a:=a_j$, we conclude  that, for any $i\in\mathbb{N}\cap[3,\infty)$,
\begin{align*}
(a_j)^\ast_L\mathbf{1}_{U_i(B_j)}
\lesssim
2^{-\delta i}\|\mathbf{1}_{B_j}\|_{X}^{-1}.
\end{align*}
By this, \eqref{jiben}, and using Proposition \ref{prfs} with $p:=s$, we find that
\begin{align*}
\mathrm{II}^s
&\lesssim
\left\|\sum_{j\in\mathbb{N}}\sum_{i=3}^\infty
2^{-\delta i}\lambda_j\|\mathbf{1}_{B_j}\|_{X}^{-1}\mathbf{1}_{U_i(B_j)}\right\|_{X}^s
\lesssim
\left\|\sum_{i=3}^\infty2^{-s\delta i}\sum_{j\in\mathbb{N}}\left[\frac{\lambda_j}{\|\mathbf{1}_{B_j}\|_{X}}\right]^s\mathbf{1}_{U_i(B_j)}\right\|_{X^{\frac{1}{s}}}\\
&\lesssim
\sum_{i=3}^\infty2^{-s\delta i+nti}
\left\|\sum_{j\in\mathbb{N}}\left[\frac{\lambda_j}{\|\mathbf{1}_{B_j}\|_{X}}\right]^s\mathbf{1}_{B_j}\right\|_{X^{\frac{1}{s}}}
\sim\left\|\sum_{j\in\mathbb{N}}\left[\frac{\lambda_j}{\|\mathbf{1}_{B_j}\|_{X}}\right]^s\mathbf{1}_{B_j}\right\|_{X^{\frac{1}{s}}},
\end{align*}
where $t\in (\max\{1,\frac{s}{p_-}\}, \frac{\delta s}{n})$ is a constant.
From this, \eqref{suomei1}, \eqref{suomei2}, and \eqref{suomei3}, we deduce that \eqref{tuozhezou} holds.

Now, we finish the proof of \eqref{zidongs}. Let the conditions be as in \eqref{zidongs}. For any $x\in\mathcal{X},$ let 
\begin{align*}
a^{\ast,1}_{L}(x):=\sup_{t\in(0,r_B)}\sup_{y\in B(x,t)}\left|e^{-t^2L}(a)(y)\right|
\end{align*}
and
\begin{align}\label{jihua}
a^{\ast,2}_{L}(x):=\sup_{t\in[r_B,1)}\sup_{y\in B(x,t)}\left|e^{-t^2L}(a)(y)\right|.
\end{align}
By the definition of local $(X,M,p)_L$-atom and H\"older's inequality, we obtain
\begin{align*}
[\mu(B)]^{-1}\|a\|_{L^1}
\leq
[\mu(B)]^{-\frac{1}{p}}\|a\|_{L^p}
\leq
\|\mathbf{1}_{B}\|_{X}^{-1}.
\end{align*}
Observe that, for any  $x\in (4B)^\complement$,
$t\in(0,\infty)$, $y\in B(x,t)$, and $z\in B$,
\begin{align}\label{20253}
d(y,z)
\geq d(x,z)-d(x,y)
\geq d(x,x_B)-d(x_B,z)-t
\geq d(x,x_B)-r_B-t
> \frac{d(x,x_B)}{2}-t
\end{align}
and, for any  $x\in (4B)^\complement$ and 
$t\in(0,r_B)$,
\begin{align*}
\frac{1}{V(x,t)}
\lesssim
\frac{1}{V(x,r_B)}\left(\frac{r_B}{t}\right)^n
\lesssim
\frac{1}{\mu(B)}\left[\frac{d(x,x_B)}{t}\right]^n.
\end{align*}
Using this, \eqref{d2}, and \eqref{20253}, we have, for any $x\in(4B)^\complement,$ 
\begin{align}\label{1637}
a^{\ast,1}_{L}(x)
\notag&\lesssim
\sup_{t\in(0,r_B)}\sup_{y\in B(x,t)}
\int_B\frac{1}{V(x,t)}\exp\left\{-\frac{[d(y,z)]^2}{ct^2}\right\}|a(z)|\,d\mu(z)\\
&\lesssim
\sup_{t\in(0,r_B)}
\frac{1}{V(x,t)}\exp\left\{-\frac{1}{c}\left(\frac{d(x,x_B)}{2t}-1\right)^2\right\}\|a\|_{L^1}\notag\\
&\lesssim
\sup_{t\in(0,r_B)}\frac{1}{\mu(B)}\left[\frac{d(x,x_B)}{t}\right]^n\left[\frac{t}{d(x,x_B)}\right]^{n+\delta}\|a\|_{L^1}
\lesssim
\left[\frac{r_B}{d(x,x_B)}\right]^{\delta}\|\mathbf{1}_{B}\|_{X}^{-1}.
\end{align}

Next, we estimate $a^{\ast,2}_{L}.$
In this case, we must have $r_{B}\in(0,1).$
By the definition of local $(X,M,p)_L$-atom,
there exists a function $b$ such that $a=L^M(b)$ and \eqref{atom2} holds.
By the fact that $b$ is supported in $B$,
H\"older's inequality, and \eqref{atom2}, we obtain
\begin{align*}
[\mu(B)]^{-1}\|b\|_{L^1}
\leq
[\mu(B)]^{-\frac{1}{p}}\|b\|_{L^p}
\leq
r_{B}^{2M}\|\mathbf{1}_{B}\|_{X}^{-1}.
\end{align*}
From this, \eqref{lzy}, \eqref{d1}, and 
\eqref{20253}, 
we deduce that, for any $x\in (4B)^\complement,$
\begin{align}\label{ying}
a^{\ast,2}_{L}(x)
\notag&=\sup_{t\in[r_B,1)}\sup_{y\in B(x,t)}
t^{-2M}\left|(t^2L)^Me^{-t^2L}(b)(y)\right|\\
&\lesssim
\sup_{t\in[r_B,1)}
t^{-2M}\frac{1}{V(x,t)}
\exp\left\{-\frac{1}{c}\left(\frac{d(x,x_B)}{2t}-1\right)^2\right\}\|b\|_{L^1}.
\end{align}
By  \eqref{d1} and \eqref{d2}, we find that,
for any  $x\in (4B)^\complement$
and $t\in[r_B,1)$,
\begin{align*}
\mu(B)
\leq
V(x_B,t)
\lesssim
\left(1+\frac{d(x,x_B)}{t}\right)^nV(x,t).
\end{align*}
From this and \eqref{ying}, it follows that, for any $x\in (4B)^\complement$,
\begin{align*}
a^{\ast,2}_{L}(x)
&\lesssim
\sup_{t\in[r_B,1)}
\left(\frac{r_B}{t}\right)^{2M}
\left(1+\frac{d(x,x_B)}{t}\right)^n\exp\left\{-\frac{1}{c}\left(\frac{d(x,x_B)}{2t}-1\right)^2\right\}\|\mathbf{1}_{B}\|_{X}^{-1}\\
&\lesssim
\sup_{t\in[r_B,1)}\left(\frac{r_B}{t}\right)^{2M}\left(\frac{t}{d(x,x_B)}\right)^{\delta}\|\mathbf{1}_{B}\|_{X}^{-1}
\lesssim
\left(\frac{r_B}{d(x,x_B)}\right)^{\delta}\|\mathbf{1}_{B}\|_{X}^{-1},
\end{align*}
where we used $\delta-2M\in(-\infty,0)$
in the last step. This, combined with
\eqref{1637}, implies that \eqref{zidongs}
holds, which completes the proof of Proposition \ref{oneside}.
\end{proof}
\begin{proof}[Proof of Theorem \ref{thm-maxi}]
By Propositions \ref{kenan3} and \ref{kenan4}, we conclude that the spaces 
$h_{X,L,\mathrm{max}}^{\varphi,\alpha}\cap L^2$,
$h_{X,L,\mathrm{rad}}^{\varphi}\cap L^2,$
$h_{X,L,\mathrm{max}}\cap L^2$,
and
$h_{X,L,\mathrm{rad}}\cap L^2$
coincide with equivalent quasi-norms.
On the other hand, 
by Propositions \ref{oneside} and \ref{pro-max-dec}, we find that
\begin{align*}
\widetilde{h}^{M,p}_{X,L,\mathrm{at}}
\subset
h_{X,L,\mathrm{max}}\cap L^2
\subset
\widetilde{h}^{M,\infty}_{X,L,\mathrm{at}}
\subset
\widetilde{h}^{M,p}_{X,L,\mathrm{at}},
\end{align*}
which implies that the spaces
$\widetilde{h}^{M,p}_{X,L,\mathrm{at}}$
and
$h_{X,L,\mathrm{max}}\cap L^2$
coincide with equivalent quasi-norms.
Thus, the spaces
$\widetilde{h}^{M,p}_{X,L,\mathrm{at}}$,
$h_{X,L,\mathrm{max}}^{\varphi,\alpha}\cap L^2$,
$h_{X,L,\mathrm{rad}}^{\varphi}\cap L^2,$
$h_{X,L,\mathrm{max}}\cap L^2$,
and
$h_{X,L,\mathrm{rad}}\cap L^2$
coincide with equivalent quasi-norms.
From this and a density argument, it follows that the desired conclusions hold.
This finishes the proof of Theorem \ref{thm-maxi}.
\end{proof}
Finally, we prove Theorem \ref{key}.
\begin{proof}[Proof of Theorem \ref{key}]
Let  $M\in\mathbb{N}$.  By the proof of \cite[Lemma 9.1]{hlmmy11}, we conclude that,  for any  local $(X,M,p)_L$-atom $a$, $a$
is also  local $(X,p)$-atom. This implies that
\begin{align}\label{huaji2}
h_{X,L,\mathrm{at}}^{M,p}\cap L^2\subset h_{X,\mathrm{at}}\cap L^2.
\end{align}

Next, we prove that 
\begin{align}\label{huaji}
h_{X,\mathrm{at}}\cap L^2\subset
h_{X,L,\mathrm{max}}\cap L^2.
\end{align}
Following the proof of Proposition \ref{oneside},  we only need to prove that 
\begin{align}\label{siqu}
\left\|\sum_{j\in\mathbb{N}}\sum_{i=3}^\infty
\lambda_{j}(a_j)_L^{\ast,2}\mathbf{1}_{U_i(B_j)}\right\|_{X}^s
\lesssim
\left\|\sum_{j\in\mathbb{N}}\left(\frac{\lambda_j}{\|\mathbf{1}_{B_j}\|_{X}}\right)^s\mathbf{1}_{B_j}\right\|_{X^{\frac{1}{s}}}.
\end{align}
We first show that
there exists a positive constant $C$ such that,
for any  local $(X,p)$-atom $a,$ associated with ball $B:=B(x_B,r_B),$
and any $x\in (4B)^\complement,$
\begin{align}\label{1700}
a_L^{\ast,2}(x)
\lesssim
\frac{\mu(B)}{V(x_B,d(x,x_B))}\left[\frac{r_B}{d(x,x_B)}\right]^\delta\|\mathbf{1}_{B}\|_{X}^{-1},
\end{align}
where $a_{L}^{\ast,2}$ is the same as in \eqref{jihua}.
From the definition of $a$, we infer that, for any $x\in(4B)^{\complement}$,
\begin{align*}
a_L^{\ast,2}(x)
&=\sup_{t\in[r_B,1)}\sup_{y\in B(x,t)}\left|\int_BK_{t^2}(y,z)a(z)\,d\mu(z)\right|\\
&=\sup_{t\in[r_B,1)}\sup_{y\in B(x,t)}\left|\int_B\left(K_{t^2}(y,z)-K_{t^2}(y,x_B)\right)a(z)\,d\mu(z)\right|.
\end{align*}
For any $t\in[r_B,1)$, $x\in(4B)^{\complement}$,  $y\in B(x,t)$, and $z\in B$, we have
\begin{align*}
t+d(x_B,y)\geq t+d(x_B,x)-d(x,y)>d(x,x_B)> 2d(x_B,z).
\end{align*}
By this, \eqref{A3}, and \eqref{d1}, we obtain, for any $t\in[r_B,1)$, $x\in(4B)^{\complement}$,  $y\in B(x,t)$, and $z\in B$,
\begin{align*}
&\left|K_{t^2}(y,z)-K_{t^2}(y,x_B)\right|\\
&\quad\lesssim\left(\frac{d(z,x_B)}{t}\right)^{\delta}\frac{1}{V(x_B,t)}e^{-d(y,x_B)^2/ct^2}
\lesssim\left(\frac{r_B}{t}\right)^{\delta}\frac{1}{V(x_B,t)}e^{-d(x,x_B)^2/ct^2}\\
&\quad\lesssim
\left(\frac{r_B}{t}\right)^{\delta}\frac{1}{V(x_B,d(x,x_B))}\left[\frac{t}{d(x,x_B)}\right]^\delta
=
\frac{1}{V(x_B,d(x,x_B))}\left[\frac{r_B}{d(x,x_B)}\right]^\delta.
\end{align*}
This implies that \eqref{1700} holds. 
Let $r\in (1,\infty)$ satisfy 
$s\delta>(r-s)n>(1-s)n$.
By \eqref{1700}, we have, for any $j,i\in\mathbb{N}$ with $i\geq 3,$
\begin{align*}
(a_j)_L^{\ast,2}\mathbf{1}_{U_i(B_j)}
\lesssim
2^{-\delta i}\frac{\mu(B_j)}{\mu(2^iB_j)}\|\mathbf{1}_{B_j}\|_{X}^{-1}.
\end{align*}
From this and Assumption \ref{vector1}, we deduce that 
\begin{align*}
&\left\|\sum_{j\in\mathbb{N}}\sum_{i=3}^\infty
\lambda_{j}(a_j)_L^{\ast,2}\mathbf{1}_{U_i(B_j)}\right\|_{X}^s\\
&\quad\lesssim
\left\|\sum_{j\in\mathbb{N}}\sum_{i=3}^\infty\left[2^{-\delta i}\frac{\lambda_j}{\|\mathbf{1}_{B_j}\|_{X}}\frac{\mu(B_j)}{\mu(2^iB_j)}\right]^s\mathbf{1}_{U_i(B_j)}\right\|_{X^{\frac{1}{s}}}\\
&\quad\lesssim
\sum_{i=3}^\infty2^{-s\delta i}
\left\|\sum_{j\in\mathbb{N}}\left(\frac{\lambda_j}{\|\mathbf{1}_{B_j}\|_X}\right)^s\left[\frac{\mu(2^iB_j)}{\mu(B_j)}\right]^{r-s}[\mathcal{M}(\mathbf{1}_{B_j})]^r\right\|_{X^{\frac{1}{s}}}\\
&\quad\lesssim
\sum_{i=3}^\infty2^{-s\delta i+(r-s)ni}\left\|\sum_{j\in\mathbb{N}}\left(\frac{\lambda_j}{\|\mathbf{1}_{B_j}\|_{X}}\right)^s\mathbf{1}_{B_j}\right\|_{X^{\frac{1}{s}}}
\sim
\left\|\sum_{j\in\mathbb{N}}\left(\frac{\lambda_j}{\|\mathbf{1}_{B_j}\|_{X}}\right)^s\mathbf{1}_{B_j}\right\|_{X^{\frac{1}{s}}}.
\end{align*}
This proves \eqref{siqu}.

By \eqref{huaji2}, \eqref{huaji},  Theorem \ref{thm-maxi}, and a density argument, we find that
the spaces $h_{X,L,\mathrm{max}}$ and $h_{X,\mathrm{at}}^p$ coincide with equivalent quasi-norms.  
This finishes the proof of Theorem \ref{key}.
\end{proof}
\section{Applications to Specific Function Spaces}\label{apply}
In this section, we apply the main results obtained in Section \ref{main} to
some specific function spaces, including  local Hardy spaces (see Subsection \ref{Lp}), local Orlicz--Hardy spaces
(see Subsection \ref{19961}), local weighted Hardy spaces (see Subsection \ref{19962}), and local variable
Hardy spaces (see Subsection \ref{19963}). 
These examples prove that the results obtained in this article
have wide generality and flexibility, which are obviously possible to be applied
to some newfound  function spaces.
\subsection{Local Hardy Spaces}\label{Lp}
By using an argument similar to that used in  the proof of \cite[Theorem 3.7]{lyyy25jga} and applying Theorems \ref{local-mo-at}, \ref{hH}, \ref{thm-maxi}, and \ref{key} and  with $X:=L^{r}$, we have the following conclusions; we omit the details here.
\begin{theorem}\label{Lr}
\begin{enumerate}[(i)]
\item  Let $r\in (0,2)$, $p\in (1,\infty)\cap ({r},\infty)$, $M\in\mathbb{N}\cap (\frac{n}{2}(\frac{1}{\min\{{r},1\}}-\frac{1}{p}),\infty)$, and $\epsilon\in (\frac{n}{\min\{{r},1\}},\infty)$. Then Theorem \ref{local-mo-at} and Theorem \ref{hH}(i) and (ii) hold with $X:=L^{r}.$
\item Let $r\in(1,2).$  Theorem \ref{hH}(iii) holds with $X:=L^{r}.$
\item Let $p\in (1,\infty]\cap ({r},\infty]$ and  $M\in\mathbb{N}\cap (\frac{n}{2\min\{r,1\}},\infty)$. Then Theorem \ref{thm-maxi} holds with $X:=L^{r}.$
\item Let 
$\delta\in (0,\infty)$, $r\in (\frac{n}{n+\delta},\infty)$, and 
$p\in (1,\infty]\cap ({r},\infty]$. Then Theorem \ref{key} holds with $X:=L^{r}.$ 
\end{enumerate}
\end{theorem}
\begin{remark}
\begin{itemize}
\item [$\mathrm{(i)}$] 
In  the case where $X:=L^1,$ we denote $h_{X,L}$ by $h_{L}^1.$
We point out that
Gong, Li and Yan \cite{gly13}
obtained the atomic characterization of the local Hardy space  $h_{L}^1$
under a slight different
definition of $h_{L}^1.$
\item [$\mathrm{(ii)}$] We  point out  that, in the  case where $r\in (0,1]$,  
Theorem \ref{Lr}(iii) and (iv) 
is a part of \cite[Theorems 2.4 and 2.7]{bdl18}; in the other case, Theorem \ref{Lr}(iii) and (iv) is new.
\end{itemize}
\end{remark}
\subsection{Local Orlicz--Hardy Spaces}\label{19961}
A function $\Phi:\ [0,\infty)\to[0,\infty)$ is called an \emph{Orlicz function} if $\Phi$ is non-decreasing, $\Phi(0)=0$, $\Phi(t)>0$
for any $t\in(0,\infty)$, and $\lim_{t\to\infty}\Phi(t)=\infty.$ Moreover, an Orlicz function $\Phi$ is
said to be of \emph{lower} [resp. \emph{upper}] \emph{type} $r$ for some $r\in\mathbb{R}$ if there exists
a positive constant $C_{r}$ such that, for any $t\in[0,\infty)$ and $s\in(0,1)$ [resp. $s\in[1,\infty)$],
$\Phi(st)\le C_{r} s^r\Phi(t).$
In what follows, we \emph{always} assume that $\Phi:\ [0,\infty)\to[0,\infty)$ is an Orlicz function with
positive lower type $r_{\Phi}^-$ and positive upper type $r_{\Phi}^+$.

\begin{definition}
The \emph{Orlicz space $L^\Phi$} is defined to be the set of all the $\mu$-measurable functions $f$ on $\mathcal{X}$
with the following finite \emph{quasi-norm}
\begin{align*}
\|f\|_{L^\Phi}:=\inf\left\{\lambda\in(0,\infty):\ \int_{\mathcal{X}}\Phi\left(\frac{|f(x)|}
{\lambda}\right)\,d\mu(x)\le1\right\}.
\end{align*}
\end{definition}

\begin{remark}
As pointed out in \cite[Subsection 8.3]{syy22jga}, the Orlicz space $L^\Phi$ is a quasi-Banach function space and hence
a ball quasi-Banach function space.
\end{remark}
By using an argument similar to that used in  the proof of \cite[Theorem 5.3]{lyyy25jga} and applying Theorems \ref{local-mo-at}, \ref{hH}, \ref{thm-maxi}, and \ref{key} and  with $X:=L^{\Phi}$, we have the following conclusions; we omit the details here.
\begin{theorem}\label{or-mol}
Let $\Phi:\ [0,\infty)\to[0,\infty)$ be an Orlicz function with
positive lower type $r_{\Phi}^-$ and positive upper type $r_{\Phi}^+$. Assume that $0<r_{\Phi}^-\leq r_{\Phi}^+<\infty$.
\begin{enumerate}[(i)]
\item  Let ${r}_{\Phi}^+\in (0,2)$, $p\in (1,\infty)\cap ({r}_{\Phi}^+,\infty)$, $M\in\mathbb{N}\cap (\frac{n}{2}(\frac{1}{\min\{{r}_{\Phi}^-,1\}}-\frac{1}{p}),\infty)$, and $\epsilon\in (\frac{n}{\min\{{r}_{\Phi}^-,1\}},\infty)$. Then Theorem \ref{local-mo-at} and Theorem \ref{hH}(i) and (ii) hold with $X:=L^{\Phi}.$
\item Let $1<{r}_{\Phi}^-\leq {r}_{\Phi}^+<2.$  Theorem \ref{hH}(iii) holds with $X:=L^{\Phi}.$
\item Let 
$p\in (1,\infty]\cap ({r}_{\Phi}^+,\infty]$ and  $M\in\mathbb{N}\cap (\frac{n}{2\min\{{r}_{\Phi}^-,1\}},\infty)$. Then Theorem \ref{thm-maxi} holds with $X:=L^{\Phi}.$
\item Let 
$\delta\in (0,\infty)$,
${r}_{\Phi}^-\in (\frac{n}{n+\delta},\infty),$
and
$p\in (1,\infty]\cap ({r}_{\Phi}^+,\infty]$. Then Theorem \ref{key} holds with $X:=L^{\Phi}.$ 
\end{enumerate}
\end{theorem}
\begin{remark}
To the best of our knowledge, Theorem \ref{or-mol} is new.
\end{remark}
\subsection{Local Weighted Hardy Spaces}\label{19962}
\begin{definition}
Let $r\in(0,\infty)$ and $\omega\in A_\infty$. The \emph{weighted Lebesgue space}
$L^r_\omega $ is defined to be the set of
all  $\mu$-measurable functions $f$ on $\mathcal{X}$ with the following finite \emph{quasi-norm}
\begin{align*}
\|f\|_{L^r_\omega }:=\left[\int_\mathcal{X} |f(z)|^r\omega(z)\,d\mu(z)\right]^{\frac{1}{r}}.
\end{align*}
\end{definition}

\begin{remark}
It is worth pointing out that, for any given $r\in(0,\infty)$ and $\omega\in A_\infty$, $L^r_\omega$ is a ball quasi-Banach
function space, but not necessarily a quasi-Banach function space
(see, for example, \cite[Subsection 7.1]{shyy17} for the Euclidean space case).
\end{remark}
By using an argument similar to that used in  the proof of \cite[Theorem 5.15]{dlyyz23} and applying Theorems \ref{local-mo-at}, \ref{thm-maxi}, and \ref{key} and Theorem \ref{hH}(i) and (ii) with $X:=L^r_\omega$, we have the following conclusions; we omit the details here.
\begin{theorem}\label{weighed-mol}
Let $\omega\in A_{1}$ and $r\in (0,1].$
\begin{enumerate}[(i)]
\item  Let 
$p\in (1,\infty)$, $M\in(\frac{n}{2}(\frac{1}{r}-\frac{1}{p}),\infty)\cap\mathbb{N}$, and $\epsilon\in (\frac{n}{r},\infty)$. 
Then Theorem \ref{local-mo-at} and Theorem \ref{hH}(i) and (ii) hold with $X:=L_{\omega}^{r}.$
\item Let $p\in (1,\infty]$ and  $M\in(\frac{n}{2r},\infty)\cap\mathbb{N}$.
Then Theorem \ref{thm-maxi} holds with $X:=L_{\omega}^{r}.$
\item Let $\delta\in (0,\infty)$, $r\in (\frac{n}{n+\delta},1]$, and $p\in (1,\infty].$
Then Theorem \ref{key} holds with $X:=L_{\omega}^{r}.$ 
\end{enumerate}
\end{theorem}
\begin{remark}
\begin{itemize}
\item [$\mathrm{(i)}$] 
In  the case where $X:=L^1_{\omega},$ we denote $h_{X,L}$ by $h_{\omega,L}^1.$
We point out that
Gong, Song and Xie \cite[Theorem 1.3]{gsx}
obtained the atomic characterization of the local weighted Hardy space  $h_{\omega,L}^1$
under a slight different
definition of $h_{\omega, L}^1.$
\item [$\mathrm{(ii)}$] We  point out that Theorem \ref{weighed-mol}(ii) and (iii) is new.
\end{itemize}
\end{remark}
\subsection{Local Variable Hardy Spaces}\label{19963}
We first recall the concept of variable Lebesgue spaces on $\mathcal{X}$.
Let $r(\cdot):\ \mathcal{X}\to(0,\infty)$ be a $\mu$-measurable function. Define
$$
\widetilde{r}_-:=\underset{x\in\mathcal{X}}{\operatorname{ess\,inf}}\,r(x)\ \text{and}\ \
\widetilde{r}_+:=\underset{x\in\mathcal{X}}{\operatorname{ess\,sup}}\,r(x).
$$
\begin{definition}
The \emph{variable Lebesgue space $L^{r(\cdot)}$}, associated with the function $r:\ \mathcal{X}
\to(0,\infty)$, is defined to be the set of all  $\mu$-measurable functions $f$ on $\mathcal{X}$ with
the following finite \emph{quasi-norm}
$$
\|f\|_{L^{r(\cdot)}}:=\inf\left\{\lambda\in(0,\infty):\ \int_{\mathcal{X}}\left[\frac{|f(x)|}
{\lambda}\right]^{r(x)}\,d\mu(x)\le1\right\}.
$$
\end{definition}

\begin{remark}
As pointed out in \cite[Remark 2.7(iv)]{yhyy23}, for any $r(\cdot):\ \mathcal{X}\to(0,\infty)$ with
$0<\widetilde{r}_-\leq \widetilde{r}_+<\infty$, the variable Lebesgue space $L^{r(\cdot)}$
is a quasi-Banach function space and hence a ball quasi-Banach function space.
\end{remark}
The variable exponent $r(\cdot)$ is said to be
\emph{locally log-H\"older continuous} if there exists a positive constant $c_{\rm{log}}$ such that,
for any $x,y\in\mathcal{X}$,
$$
|r(x)-r(y)|\le \frac{c_{\rm{log}}}{\log(e+1/d(x,y))}
$$
and that $r(\cdot)$ is said to satisfy the \emph{log-H\"older decay condition} with a basepoint $x_r\in\mathcal{X}$
if there exist $r_\infty\in\mathbb{R}$ and a positive constant $c_\infty$ such that, for any $x\in\mathcal{X},$
\begin{align*}
|r(x)-r_\infty|\leq \frac{c_\infty}{\log(e+d(x,x_r))}.
\end{align*}
The variable exponent $r(\cdot)$ is said to be \emph{log-H\"older continuous} if $r(\cdot)$ satisfies both the locally
log-H\"older continuous condition and the log-H\"older decay condition. By using an argument similar to that used in  the proof of \cite[Theorem 5.21]{lyyy25jga} and applying Theorems \ref{local-mo-at}, \ref{hH}, \ref{thm-maxi}, and \ref{key} with $X:=L^{r(\cdot)}$, we have the following conclusions; we omit the details here.
\begin{theorem}\label{vai-mol}
Let
$r:\ \mathcal{X}\to(0,\infty)$ be log-H\"older continuous  with $0<\widetilde{r}_-\leq \widetilde{r}_+<\infty.$
\begin{enumerate}[(i)]
\item  Let $\widetilde{r}_+\in (0,2)$, $p\in (1,\infty)\cap (\widetilde{r}_+,\infty)$, $M\in\mathbb{N}\cap (\frac{n}{2}(\frac{1}{\min\{\widetilde{r}_-,1\}}-\frac{1}{p}),\infty)$, and $\epsilon\in (\frac{n}{\min\{\widetilde{r}_-,1\}},\infty)$. Then Theorem \ref{local-mo-at} and Theorem \ref{hH}(i) and (ii) hold with $X:=L^{r(\cdot)}.$
\item Let $1<\widetilde{r}_-\leq \widetilde{r}_+<2.$ Then Theorem \ref{hH}(iii) holds with $X:=L^{r(\cdot)}.$
\item Let $p\in (1,\infty]\cap (\widetilde{r}_+,\infty]$ and  $M\in\mathbb{N}\cap (\frac{n}{2\min\{\widetilde{r}_-,1\}},\infty)$. Then Theorem \ref{thm-maxi} holds with $X:=L^{r(\cdot)}.$
\item Let $\delta\in (0,\infty)$, $\widetilde{r}_-\in (\frac{n}{n+\delta},\infty)$, and
$p\in (1,\infty]\cap (\widetilde{r}_+,\infty]$. Then Theorem \ref{key} holds with $X:=L^{r(\cdot)}.$ 
\end{enumerate}
\end{theorem}
\begin{remark}
\begin{itemize}
\item  [$\mathrm{(i)}$]
We point out that,  in the  case where $\mathcal{X}:=\mathbb{R}^n$ and $p:=2$, 
\cite[Theorems 1.1, 1.5, and 1.6 and Corollaries 1.4 and 1.7 ]{abdr20} is a part of Theorem \ref{vai-mol}(i) and (ii);  in the other case, Theorem \ref{vai-mol}(i) and (ii) is new. 
\item [$\mathrm{(ii)}$] We point out that,  in the  case where $\widetilde{r}_+\in (0,1]$, Theorem \ref{vai-mol}(iii) and (iv) is a part of \cite[Theorems 1.6 and 4.3]{hl25};  in the other case, Theorem \ref{vai-mol}(iii) and (iv) is new.
\end{itemize}
\end{remark}

\section*{Declaration of Competing Interest}

There is no competing interest.

\section*{Data Availability}

No data was used for the research described in the article.

\bigskip

\noindent Xiaosheng Lin (Corresponding author) and Xiaotian Zhu

\medskip

\noindent School of Mathematical Sciences, Jimei University,
Xiamen 361005, The People's Republic of China

\smallskip

\smallskip

\noindent {\it E-mails}: \texttt{xslin@jmu.edu.cn} (X. Lin)

\noindent\phantom{{\it E-mails:} }\texttt{202221391022@jmu.edu.cn} (X. Zhu)


\begin{thebibliography}{10}

\bibitem{abdr20}
V. Almeida, J. J. Betancor, E. Dalmasso and L. Rodr\'iguez-Mesa,
Local Hardy spaces with variable exponents associated with non-negative self-adjoint operators satisfying Gaussian estimates,
J. Geom. Anal. 30 (2020), 3275--3330.

\vspace{-0.3cm}

\bibitem{ans21}
R. Arai, E. Nakai and Y. Sawano,
Generalized fractional integral operators on Orlicz--Hardy spaces,
Math. Nachr. 294 (2021),  224--235.


\vspace{-0.3cm}

\bibitem{adm05}
P. Auscher, X. T. Duong and A. McIntosh,
Boundedness of Banach space valued singular integral operators and Hardy spaces,
Unpublished manuscript, 2005.

\vspace{-0.3cm}

\bibitem{bs88}
C. Bennett and R. Sharpley,
Interpolation of Operators,
Pure and Applied Mathematics 129, Academic Press, Boston, MA, 1988.

\vspace{-0.3cm}

\bibitem{b14}
T. A. Bui,
Weighted Hardy spaces associated to discrete Laplacians on graphs and applications,
Potential Anal. 41 (2014), 817--848.

\vspace{-0.3cm}

\bibitem{bckyy13}
T. A. Bui, J. Cao, L. D. Ky, D. Yang and S. Yang,
Musielak--Orlicz--Hardy spaces associated with operators satisfying reinforced off-diagonal estimates,
Anal. Geom. Metr. Spaces 1 (2013), 69--129.

\vspace{-0.3cm}

\bibitem{bckyy131}
T. A. Bui, J. Cao, L. D. Ky, D. Yang and S. Yang,
Weighted Hardy spaces associated with operators satisfying reinforced off-diagonal estimates,
Taiwanese J. Math. 17 (2013), 1127--1166.

\vspace{-0.3cm}

\bibitem{bdl18}
T. A. Bui, X. T. Duong and F. K. Ly,
Maximal function characterizations for new local Hardy-type spaces on spaces of homogeneous type,
Trans. Amer. Math. Soc. 370 (2018), 7229--7292.

\vspace{-0.3cm}

\bibitem{bl11}
T. A. Bui and J. Li,
Orlicz--Hardy spaces associated to operators satisfying bounded $H^\infty$ functional calculus and Davies--Gaffney estimates,
J. Math. Anal. Appl. 373 (2011), 485--501.

\vspace{-0.3cm}

\bibitem{c77}
A.-P. Calder\'on,
Cauchy integrals on Lipschitz curves and related operators,
Proc. Natl. Acad. Sci. USA 74 (1977), 1324--1327.

\vspace{-0.3cm}

\bibitem{ccfjr78}
A.-P. Calder\'on, C. P. Calder\'on, E. Fabes, M. Jodeit and N. M. Rivi\`ere,
Applications of the Cauchy integral on Lipschitz curves,
Bull. Amer. Math. Soc. 84 (1978), 287--290.

\vspace{-0.3cm}

\bibitem{cmy17}
J. Cao, S. Mayboroda and D. Yang,
Local Hardy spaces associated with inhomogeneous higher order elliptic operators,
Anal. Appl. (Singap.) 15 (2017), 137--224.

\vspace{-0.3cm}

\bibitem{cmm11}
A. Carbonaro, A. McIntosh and A. J. Morris,
Local Hardy spaces of differential forms on Riemannian manifolds,
J. Geom. Anal. 23 (2013), 106--169.

\vspace{-0.3cm}

\bibitem{cw77}
R. R. Coifman and G. Weiss,
Extensions of Hardy spaces and their use in analysis,
Bull. Amer. Math. Soc. 83 (1977), 569--645.

\vspace{-0.3cm}

\bibitem{cs08}
T. Coulhon and A. Sikora,
Gaussian heat kernel upper bounds via the Phragm\'en--Lindel\"of theorem,
Proc. Lond. Math. Soc. (3) 96 (2008), 507--544.

\vspace{-0.3cm}

\bibitem{cmp}
D. V. Cruz-Uribe, J. M. Martell and C. P\'erez,
Weights, Extrapolation and the Theory of Rubio de Francia,
Operator Theory: Advances and Applications 215,
Birkh\"auser/Springer, Basel, 2011.

\vspace{-0.3cm}

\bibitem{dlyyz23}
F. Dai, X. Lin, D. Yang, W. Yuan and Y. Zhang,
Brezis--Van Schaftingen--Yung formulae in ball Banach function spaces with applications to fractional Sobolev and Gagliardo--Nirenberg inequalities,
Calc. Var. Partial Differential Equations 62 (2023), Paper No. 56, 73 pp.

\vspace{-0.3cm}

\bibitem{dfmn21}
R. del Campo, A. Fern\'andez, F. Mayoral and F. Naranjo,
Orlicz spaces associated to a quasi-Banach function space: applications to vector measures and interpolation,
Collect. Math. 72 (2021), 481--499.

\vspace{-0.3cm}

\bibitem{dhmmy13}
X. T. Duong, S. Hofmann, D. Mitrea, M. Mitrea and L. Yan,
Hardy spaces and regularity for the inhomogeneous Dirichlet and Neumann problems,
Rev. Mat. Iberoam. 29 (2013), 183--236.

\vspace{-0.3cm}

\bibitem{dy05jams}
X. T. Duong and L. Yan,
Duality of Hardy and BMO spaces associated with operators with heat kernel bounds,
J. Amer. Math. Soc. 18 (2005), 943--973.

\vspace{-0.3cm}

\bibitem{dy05cpam}
X. T. Duong and L. Yan,
New function spaces of BMO type, the John--Nirenberg inequality, interpolation, and applications,
Comm. Pure Appl. Math. 58 (2005), 1375--1420.

\vspace{-0.3cm}

\bibitem{fs72}
C. Fefferman and E. M. Stein,
$H^p$ spaces of several variables,
Acta Math. 129 (1972), 137--193.

\vspace{-0.3cm}

\bibitem{g79}
D. Goldberg,
A local version of real Hardy spaces,
Duke Math. J. 46 (1979), 27--42.

\vspace{-0.3cm}

\bibitem{gly13}
R. Gong, J. Li and L. Yan,
A local version of Hardy spaces associated with operators on metric spaces,
Sci. China Math. 56 (2013), 315--330.

\vspace{-0.3cm}

\bibitem{gsx}
R. Gong, L. Song and P. Xie,
Weighted local Hardy spaces associated with operators,
Proc. Indian Acad. Sci. Math. Sci. 128 (2018), Paper No. 24, 20 pp.

\vspace{-0.3cm}

\bibitem{gx12}
R. Gong and P. Xie,
Weighted $L^p$ estimates for the area integral associated with self-adjoint operators on homogeneous space,
J. Math. Anal. Appl. 393 (2012), 590--604.

\vspace{-0.3cm}

\bibitem{g14c}
L. Grafakos,
Classical Fourier Analysis,
Third edition, Graduate Texts in Mathematics 249,
Springer, New York, 2014.

\vspace{-0.3cm}

\bibitem{hl25}
Y. He and J. Liao,
Maximal function characterizations for local variable Hardy spaces associated to non-negative self-adjoint operators on spaces of homogeneous type,
Complex Anal. Oper. Theory 19 (2025), Paper No. 70, 20 pp.

\vspace{-0.3cm}

\bibitem{h13}
K.-P. Ho,
Atomic decompositions of weighted Hardy--Morrey spaces,
Hokkaido Math. J. 42 (2013), 131--157.

\vspace{-0.3cm}

\bibitem{h15}
K.-P. Ho,
Atomic decomposition of Hardy--Morrey spaces with variable exponents,
Ann. Acad. Sci. Fenn. Math. 40 (2015), 31--62.

\vspace{-0.3cm}

\bibitem{h17scm}
K.-P. Ho,
Atomic decompositions and Hardy's inequality on weak Hardy--Morrey spaces,
Sci. China Math. 60 (2017), 449--468.

\vspace{-0.3cm}

\bibitem{h17tmj}
K.-P. Ho,
Atomic decompositions of weighted Hardy spaces with variable exponents,
Tohoku Math. J. (2) 69 (2017), 383--413.

\vspace{-0.3cm}

\bibitem{hlmmy11}
S. Hofmann, G. Lu, D. Mitrea, M. Mitrea and L. Yan,
Hardy spaces associated to non-negative self-adjoint operators satisfying Davies--Gaffney estimates,
Mem. Amer. Math. Soc. 214 (2011), no. 1007, vi+78 pp.

\vspace{-0.3cm}

\bibitem{hmm11}
S. Hofmann, S. Mayboroda and A. McIntosh,
Second order elliptic operators with complex bounded measurable coefficients in $L^p$, Sobolev and Hardy spaces,
Ann. Sci. \'Ecole Norm. Sup. (4) 44 (2011), 723--800.

\vspace{-0.3cm}

\bibitem{ins22}
M. Izuki, E. Nakai and Y. Sawano,
Atomic and wavelet characterization of Musielak--Orlicz Hardy spaces for generalized Orlicz functions,
Integral Equ. Oper. Theory 94 (2022), Paper No. 3, 33 pp.

\vspace{-0.3cm}

\bibitem{inns23}
M. Izuki, T. Nogayama, T. Noi and Y. Sawano,
Weighted local Hardy spaces with variable exponents,
Math. Nachr. 296 (2023), 5710--5785.

\vspace{-0.3cm}

\bibitem{jy10}
R. Jiang and D. Yang,
New Orlicz--Hardy spaces associated with divergence form elliptic operators,
J. Funct. Anal. 258 (2010), 1167--1224.

\vspace{-0.3cm}

\bibitem{jy11}
R. Jiang and D. Yang,
Orlicz--Hardy spaces associated with operators satisfying Davies--Gaffney estimates,
Commun. Contemp. Math. 13 (2011), 331--373.

\vspace{-0.3cm}

\bibitem{k}
M. Kemppainen,
A note on local Hardy spaces,
Forum Math. 29 (2017), 941--949.

\vspace{-0.3cm}

\bibitem{k80}
C. E. Kenig,
Weighted $H^p$ spaces on Lipschitz domains,
Amer. J. Math. 102 (1980), 129--163.

\vspace{-0.3cm}

\bibitem{lyyy25jga}
X. Lin, D. Yang, S. Yang and W. Yuan,
Weak Hardy spaces associated with operators and ball quasi-Banach function spaces on doubling metric measure spaces,
J. Geom. Anal. 35 (2025), Paper No. 310, 54 pp.

\vspace{-0.3cm}

\bibitem{lyyy26cms}
X. Lin, D. Yang, S. Yang and W. Yuan,
Hardy spaces associated with non-negative self-adjoint operators and ball quasi-Banach function spaces on doubling metric measure spaces and their applications,
Commun. Math. Stat. 14 (2026), 141--193.

\vspace{-0.3cm}

\bibitem{lt79}
J. Lindenstrauss and L. Tzafriri,
Classical Banach Spaces. II: Function Spaces,
Ergebnisse der Mathematik und ihrer Grenzgebiete,
Springer-Verlag, Berlin--New York, 1979.

\vspace{-0.3cm}

\bibitem{ls13}
S. Liu and L. Song,
An atomic decomposition of weighted Hardy spaces associated to self-adjoint operators,
J. Funct. Anal. 265 (2013), 2709--2723.

\vspace{-0.3cm}

\bibitem{ln20231}
E. Lorist and Z. Nieraeth,
Banach function spaces done right,
Indag. Math. (N.S.) 35 (2024), 247--268.

\vspace{-0.3cm}

\bibitem{ns12}
E. Nakai and Y. Sawano,
Hardy spaces with variable exponents and generalized Campanato spaces,
J. Funct. Anal. 262 (2012), 3665--3748.

\vspace{-0.3cm}

\bibitem{ns14}
E. Nakai and Y. Sawano,
Orlicz--Hardy spaces and their duals,
Sci. China Math. 57 (2014), 903--962.

\vspace{-0.3cm}

\bibitem{ny97}
E. Nakai and K. Yabuta,
Pointwise multipliers for functions of weighted bounded mean oscillation on spaces of homogeneous type,
Math. Japon. 46 (1997), 15--28.

\vspace{-0.3cm}

\bibitem{O05}
E. M. Ouhabaz,
Analysis of Heat Equations on Domains,
London Mathematical Society Monographs Series 31,
Princeton University Press, Princeton, NJ, 2005.

\vspace{-0.3cm}

\bibitem{s13}
Y. Sawano,
Atomic decompositions of Hardy spaces with variable exponents and its application to bounded linear operators,
Integral Equ. Oper. Theory 77 (2013), 123--148.

\vspace{-0.3cm}

\bibitem{s14}
Y. Sawano,
The Riesz potential operator $I_\alpha$ of constant order from the Hardy spaces $H^{p(\cdot)}(\mathbb{R}^n)$ with variable exponent to generalized Campanato spaces $\mathcal{L}_{\psi,1,\alpha}(\mathbb{R}^n)$,
In: Banach and Function Spaces IV (ISBFS 2012), pp. 231--250,
Yokohama Publishers, Yokohama, 2014.

\vspace{-0.3cm}

\bibitem{shyy17}
Y. Sawano, K.-P. Ho, D. Yang and S. Yang,
Hardy spaces for ball quasi-Banach function spaces,
Dissertationes Math. 525 (2017), 1--102.

\vspace{-0.3cm}

\bibitem{sw60}
E. M. Stein and G. Weiss,
On the theory of harmonic functions of several variables. I. The theory of $H^p$-spaces,
Acta Math. 103 (1960), 25--62.

\vspace{-0.3cm}

\bibitem{syy22fm}
J. Sun, D. Yang and W. Yuan,
Molecular characterization of weak Hardy spaces associated with ball quasi-Banach function spaces on spaces of homogeneous type with its applications to Littlewood--Paley function characterizations,
Forum Math. 34 (2022), 1539--1589.

\vspace{-0.3cm}

\bibitem{syy22jga}
J. Sun, D. Yang and W. Yuan,
Weak Hardy spaces associated with ball quasi-Banach function spaces on spaces of homogeneous type: Decompositions, real interpolation, and Calder\'on--Zygmund operators,
J. Geom. Anal. 32 (2022), Paper No. 191, 85 pp.

\vspace{-0.3cm}

\bibitem{y08}
L. Yan,
Classes of Hardy spaces associated with operators, duality theorem and applications,
Trans. Amer. Math. Soc. 360 (2008), 4383--4408.

\vspace{-0.3cm}

\bibitem{yhyy23}
X. Yan, Z. He, D. Yang and W. Yuan,
Hardy spaces associated with ball quasi-Banach function spaces on spaces of homogeneous type: Characterizations of maximal functions, decompositions, and dual spaces,
Math. Nachr. 296 (2023), 3056--3116.

\vspace{-0.3cm}

\bibitem{yzz18}
D. Yang, J. Zhang and C. Zhuo,
Variable Hardy spaces associated with operators satisfying Davies--Gaffney estimates,
Proc. Edinb. Math. Soc. (2) 61 (2018), 759--810.

\vspace{-0.3cm}

\bibitem{yz16}
D. Yang and C. Zhuo,
Molecular characterizations and dualities of variable exponent Hardy spaces associated with operators,
Ann. Acad. Sci. Fenn. Math. 41 (2016), 357--398.

\vspace{-0.3cm}

\bibitem{yy14} D. Yang and S. Yang,  Musielak--Orlicz--Hardy spaces associated with operators and their
applications, J. Geom. Anal. 24 (2014), 495--570.

\vspace{-0.3cm}

\end{thebibliography}
\end{document}